\documentclass[11pt, reqno, a4paper]{amsart}

\usepackage{amsfonts}
\usepackage{amsmath}
\usepackage{color}
\usepackage{hyperref}
\usepackage[inline]{enumitem}

\usepackage{pict2e}
\usepackage{graphicx}
\usepackage{tikz}

\usepackage{fullpage}

\newtheorem{thm}{Theorem}[section]
\newtheorem{claim}{Claim}

\newtheorem{lem}[thm]{Lemma}

\newtheorem*{prob*}{Problem}
\newtheorem*{thm*}{Theorem}

\theoremstyle{definition}

\newtheorem*{defn*}{Definition}
\newtheorem{rem}{Remark}[section]

\newtheorem*{globalestA}{Global estimates A}
\newtheorem*{globalestB}{Global estimates B}
\newtheorem{que}{Question}
\newtheorem*{rem*}{Remark}
\numberwithin{equation}{section}

\newcommand{\K}{\widetilde{K}}

\newcommand{\natN}{\mathbb{N}}
\newcommand{\realR}{\mathbb{R}}
\newcommand{\compC}{\mathbb{C}}
\newcommand{\intZ}{\mathbb{Z}}
\newcommand{\bigO}{\mathcal{O}}
\renewcommand{\Re}{\operatorname{Re}}
\renewcommand{\Im}{\operatorname{Im}}
\DeclareMathOperator{\PV}{P.V.}
\DeclareMathOperator{\loc}{local}
\DeclareMathOperator{\glob}{global}
\DeclareMathOperator{\out}{out}
\DeclareMathOperator{\inner}{in}

\DeclareMathOperator{\Airy}{Ai}

\DeclareMathOperator{\crit}{crit}
\DeclareMathOperator{\bulk}{bulk}

\newcommand{\Prob}{\mathbb{P}}
\newcommand{\re}{\mathop{\mathrm{Re}}}
\newcommand{\im}{\mathop{\mathrm{Im}}}

\newcommand{\eins}{\leavevmode\hbox{\small1\kern-3.8pt\normalsize1}}

\title[Lyapunov exponent, universality and phase Transition]{\bf{Lyapunov exponent,  universality and phase transition     for  products of random matrices}}

\begin{document}
\author{Dang-Zheng Liu}
\address{CAS Key Laboratory of Wu Wen-Tsun Mathematics, School of Mathematical Sciences, University of Science and Technology of China, Hefei 230026, P.R.~China}
 \email{dzliu@ustc.edu.cn}

\author{Dong Wang}
\address{School of Mathematical Sciences, University of Chinese Academy of Sciences, Beijing 100047, P.R.~China}  \email{wangdong@wangd-math.xyz}
\author{Yanhui Wang}
\address{School of Mathematics and Statistics, Henan University, Kaifeng, 475001, P.R.~China}
\email{yhwang@henu.edu.cn}

\keywords{Lyapunov exponent, Products of random matrices, GUE statistics, Phase transition}

\commby{}
\begin{abstract}
  Products of $M$ i.i.d.~random matrices of size $N \times N$ are related to classical limit theorems in probability theory ($N=1$ and large $M$), to Lyapunov exponents in dynamical systems (finite $N$ and large $M$), and to  universality in random matrix theory (finite $M$ and large $N$). Under the two different limits of $M\to \infty$ and $N\to \infty$,  the local singular value statistics    display  Gaussian and   random matrix theory  universality, respectively.
  However, it is unclear what happens if both $M$ and $N$ go to infinity. This problem, proposed by Akemann, Burda, Kieburg \cite{Akemann-Burda-Kieburg14} and Deift \cite{Deift17}, lies at the heart of understanding  both kinds of universal limits. In the case of complex Gaussian random matrices, we prove  that there exists a crossover phenomenon as the relative ratio of $M$ and $N$ changes from $0$ to $\infty$: sine and Airy kernels from the Gaussian Unitary Ensemble  (GUE) when $M/N \to 0$, Gaussian fluctuation when $M/N \to \infty$, and new critical phenomena when $M/N \to \gamma \in (0,\infty)$.
 Accordingly, we  further  prove that  the largest singular value   undergoes a phase transition between the Gaussian  and GUE Tracy-Widom distributions.
\end{abstract}
\date{\today}

\maketitle

\tableofcontents

\section{Introduction and main results}

\subsection{Lyapunov exponents}

In his famous 1892 monograph \cite{Lyapunov92}, Alexandr Lyapunov introduced the concept of Lyapunov exponent, which originated from the problem of the stability of solutions of differential equations. For a linearized differential equation
\begin{equation} \label{linearEq}
  \dot{v} (t) =X_t v, \quad v(0)=v_0\in  \mathbb{R}^N,
\end{equation}
where $X_{(\cdot)}$ is a continuous and bounded function from $\mathbb{R}_{+}$ to the space of $N\times N$ real matrices, the (largest) Lyapunov exponent of a solution $v(t;v_0)$ of \eqref{linearEq} is defined in the following manner
\begin{equation} \label{LE1}
  \lambda(v_0):=\limsup_{t\rightarrow \infty} \frac{1}{t} \log\|v(t)\|.
\end{equation}
Moreover, Lyapunov proved that $\lambda(v_0)$ is finite for every solution  with $v_0 \neq 0$. Later, through the works of Furstenberg, Kesten, Oseledets, Kingman, Ruelle, Margulis, Avila and other mathematicians, Lyapunov exponents have recently emerged as an important concept in various fields of mathematics and physics, such as linear stochastic systems and stability theory, products of random matrices and random maps, spectral theory of random Schr\"{o}dinger operators, and smooth dynamics;  see e.g.~\cite{Arnold-Wihstutz86,Viana14,Wilkinson17}. 

In this paper, we are interested in the discrete version of Lyapunov exponents. A discrete-time evolution of an $N$-dimensional real or complex stochastic system which is described by linear difference equations
\begin{equation}
  v(t+1) =X_{t+1} v(t), \quad t=0, 1, 2, \ldots, \label{linearEq2}
\end{equation}
then the total evolution is effectively driven by the product of random matrices at time $t=M$
\begin{equation} \label{Mproduct}
  \Pi_{M}^{}=X_{M}\cdots X_{2}X_{1}.
\end{equation}
The study on products of random matrices can be dated at least from the seminal articles by Bellman \cite{Bellman54} in 1954 and further by Furstenberg and Kesten \cite{Furstenberg-Kesten60} in 1960, in which classical limit theorems in probability theory were obtained under certain assumptions when $M$ goes to infinity. In particular, if $X_{1}, X_2, \dotsc, X_{M}$ are i.i.d.~$N\times N$ random matrices,  each of which has independent (and identically distributed) entries with mean zero and variance one, then the theorem of Furstenberg and Kesten \cite[Theorem 2]{Furstenberg-Kesten60} shows that for any fixed $N$ the largest Lyapunov exponent,  defined as
\begin{equation} \label{LE2}
  \lambda_{\mathrm{max}}:=\lim_{M\rightarrow \infty} \frac{1}{M} \log\|\Pi_{M}\|,
\end{equation}
exists with probability $1$.  Furthermore, all Lyapunov exponents $\lambda_{k}:=\lim_{M\rightarrow \infty} \lambda_{k,M}$ (Lyapunov spectrum),   with
\begin{equation} \label{LE3}
   \lambda_{k, M}:= \frac{1}{2M}\log\left( k\mathrm{^{th}\,   largest\,  eigenvalue\,   of\, } \Pi_{M}^{*}\Pi_{M}\right), \quad k=1,2, \ldots,N, \end{equation}
exist with probability $1$ by the multiplicative ergodic theorem of Oseledets \cite{Oseledec68,Raghunathan79}. Here it is worth stressing that the $M$-dependent quantities  $ \lambda_{k, M}$ ($k=1,2,\ldots, N$)    are typically  referred to as finite-time Lyapunov exponents, which are equivalent to singular values of $\Pi_{M}$ up to a one-to-one mapping.

However, usually it's very hard to find either explicit formulae or effective algorithms  of accurate approximation for the Lyapunov exponents. This was posed by Kingman \cite{Kingman73} as an outstanding problem in the field. Some noteworthy exceptions occur in the case of $N = 2$, see e.g. \cite{Comtet-Luck-Texier-Tourigny13,Mannion93,Marklof-Tourigny-Wolowski08}. For general $N$, when each $X_j$ is randomly chosen from a finite set of matrices with positive entries, in recent work \cite{Pollicott10} Pollicott solves this problem for the largest Lyapunov exponent. Another special case is when $\{X_j\}$ are independent real/complex \emph{Ginibre} matrices that have i.i.d.~standard real/complex Gaussian entries. This case has high interest in random matrix theory. Then the results of Newman \cite{Newman86} (real case, $\beta=1$) and Forrester \cite{Forrester13, Forrester15} (real and complex cases with $\beta=1, 2$ respectively) show that the Lyapunov spectrum
\begin{equation}
  \lambda_{k} = \frac{1}{2} \left( \log \frac{2}{\beta} +  \psi\big(\frac{\beta}{2}(N-k+1)\big)\right), \quad k=1, \dotsc, N,
\end{equation}
where $\psi(z)=\Gamma'(z)/\Gamma(z)$ denotes the digamma function, see \eqref{eq:digamma_series} below. Forrester \cite{Forrester13, Forrester15} also studied Gaussian random matrices with correlated entries; for more relevant works, see e.g.~\cite{Akemann-Burda-Kieburg14,Ipsen15,Ipsen-Schomerus16,Kargin14,Reddy16} and references therein.

The fundamental  result by Furstenberg and Kesten \cite{Furstenberg-Kesten60} about the asymptotic behavior for products of random matrices has initiated great interest in the topic over the last sixty years, see \cite{Arnold-Wihstutz86,Cohen-Kesten-Newman86} for the early articles. Recently, significant progresses have been achieved in the study of products of random matrices, which have important applications in Schr\"{o}dinger operator theory \cite{Bougerol-Lacroix85}, in statistical physics relating to disordered and chaotic dynamical systems \cite{Crisanti-Paladin-Vulpiani93}, in wireless communication like MIMO (multiple-input and multiple-output) networks \cite{Tulino-Verdu04} and in free probability theory \cite{Mingo-Speicher17}.

\subsection{Universality}

Historically, the pioneering work of Furstenberg and Kesten \cite{Furstenberg-Kesten60} and lots of subsequent works focused on statistical behavior of singular values for the products such as Lyapunov exponents, as the number of factors $M$ tends to infinity. However, the more recent interest in products of random matrices lies in statistical properties of eigenvalues and singular values as the matrix size $N$ goes to infinity, like a single random matrix. The study of one single random  matrix originated from the work of Wigner, Dyson, Mehta and others in 1950-60s, and has become a quite active research field under the name of Random Matrix Theory (RMT), which relates to many important branches of mathematics and physics; see a handbook \cite{Akemann-Baik-Di_Fransesco11}, monographs \cite{Anderson-Guionnet-Zeitouni10, Bai-Silverstein10, Deift99,Deift-Gioev09, Erdos-Yau17, Forrester10, Mehta04,Mingo-Speicher17,Pastur-Shcherbina11,Tao12,Tulino-Verdu04} and references therein.

Local statistical properties of eigenvalues in RMT are usually described by special type of correlation functions which are given by explicit correlation kernels, like sine and Airy kernels. The properties of these correlation functions stem from the repulsion of eigenvalues, which is expected for many randomly disordered systems of the same symmetry class that have delocalized eigenfunctions. This is referred to as universality in RMT, which is different from classical Gaussian universality. Many random matrix ensembles, like Wigner matrices and  invariant ensembles, have been rigorously proved to exhibit universal phenomena, see e.g.~\cite{Deift99,Deift-Gioev09, Erdos-Yau17} and references therein. As to finite products of large random matrices, statistical properties have been extensively studied in \cite{Akemann-Ipsen-Kieburg13,Akemann-Kieburg-Wei13, Alexeev-Gotze-Tikhomirov10,Forrester-Liu15,Ipsen-Kieburg14,Kieburg-Kuijlaars-Stivigny15,Kuijlaars-Stivigny14,Kuijlaars-Zhang14}; see a recent survey \cite{Akemann-Ipsen15} and references therein.

In the present  paper, we consider a concrete product model   that is defined by \eqref{Mproduct} with $X_{1}, \ldots, X_{M}$ being independent complex Ginibre matrices of size $N\times N$.
The squared singular values $x_{1}, \ldots, x_{N}$ of $\Pi_{M}$ (that is, eigenvalues of $\Pi_{M}^{*}\Pi_{M}$) have  joint probability density function, denoted by  $\mathcal{P}_{N}(x)$,    and  are further proved to form a determinantal point process with correlation kernel
\begin{equation} \label{2integral}
  K_{N}(x, y) = \int_{c-i\infty}^{c + i\infty} \frac{ds}{2\pi i} \oint_{\Sigma} \frac{dt}{2 \pi i} \frac{x^{t} y^{-s-1}}{s-t} \frac{\Gamma(t)}{\Gamma(s)}  \left( \frac{\Gamma(s+N)}{\Gamma(t+N)}\right)^{M+1},
\end{equation}
such that
 the
 $n$-point correlation functions (see e.g. \cite{Forrester10,Mehta04}) are specified as
\begin{equation} \label{eq:n-pt_corr_func}
  R^{(n)}_{N}(x_1, \ldots, x_{n}) :=  \frac{N!}{(N-n)!}\idotsint \mathcal{P}_{N}(x) \,dx_{n+1}\cdots dx_N 
   =\det[ K_N(x_i, x_j) ]_{i,j=1}^n,
\end{equation}
where $\Sigma$ is a counter-clockwise  contour encircling $0, -1, \ldots, -N+1$ and $c$ is chosen to make the vertical $s$-contour disjoint from $\Sigma$. See \cite{Akemann-Kieburg-Wei13} for an exact expression  of the  density  $\mathcal{P}_{N}(x)$  and \cite{Kuijlaars-Zhang14} for the  derivation of the kernel $K_{N}(x, y)$. Note that \eqref{2integral} is equivalent to \cite[Proposition 5.1]{Kuijlaars-Zhang14} by shifting the variables by $N$ and conjugating the kernel formula by $(x/y)^N$. We also note that in \cite{Kuijlaars-Zhang14} it is required that the $s$-contour is to the left of $\Sigma$, and this technical requirement can be removed, see \cite[Equation (2.8)]{Liu-Wang-Zhang14}. With the help of this structure, for any fixed $M$ and as $N\rightarrow \infty$, the two first-named authors with Zhang proved the sine and Airy kernels for singular values of $\Pi_M$ in \cite{Liu-Wang-Zhang14}. In an opposite direction, for any fixed $N$ and as $M\rightarrow \infty$, Akemann, Burda and Kieburg proved in \cite{Akemann-Burda-Kieburg14} that $N$ finite-time Lyapunov exponents for $\Pi_M$ are asymptotically independent Gaussian random variables.

So a very natural question arises: What will happen when both the matrix size and the number of factors tend to infinity? Precisely, will the largest Lyapunov exponent undergo a crossover from Gausssian to Tracy-Widom distribution \cite{Tracy-Widom94} at some proper scaling of $M$ and $N$? At the end of \cite[Section 5]{Akemann-Burda-Kieburg14} Akemann, Burda and Kieburg commented
\begin{quote}
  ``Since the two limits commute on the global scale while they do not commute on the local one, we claim that there should be a non-trivial double-scaling limit where new results should show up. In particular we expect a mesoscopic scale of the spectrum which may also show a new kind of universal statistics''.
\end{quote}
Also, in his 2017 list of open problems in random matrix theory and the theory of integrable systems, P.~Deift ended in \cite{Deift17} with
\begin{quote}
 ``There are many other areas, closely related to the problems in the above list, where much progress has been made in recent years, and where much remains to be done. These include:  $\ldots$, singular values of $n$ products of $m\times m$  random matrices as $n, m \rightarrow \infty$, and many others''.
\end{quote}
It is our main goal in the present paper to solve this problem when complex Ginibre matrices are involved.

\subsection{Main results} \label{mainresults}

For the matrix product \eqref{Mproduct} where independent complex Ginibre matrices are involved, when $M$ and $N$ may go to infinity simultaneously, we will  place emphasis on local statistical properties of its singular values, especially the largest one, while as to the global property it has been argued in \cite{Akemann-Burda-Kieburg18} that the limiting eigenvalue density is a constant up to some proper scaling transform. At the soft edge (also the right edge) of the spectrum we completely characterize the limiting behavior of the largest singular value (equivalently, the largest finite Lyapunov exponent), which plays a key role in dynamical systems and in statistics.  In the bulk of the spectrum, together with Akemann, Burda and Kieburg's interpolating kernel \cite[Equation (13)]{Akemann-Burda-Kieburg18} in the critical regime (see also Theorem  \ref{thm:bulkcrit} in Section \ref{sect:discuss}), we have a pretty good understanding of the bulk behavior. At the hard edge of the spectrum, or equivalently around the smallest singular value, the phase transition of the local statistics is of a different nature. As $N \to \infty$, for every finite and fixed $M$ there exists a known local kernel (Meijer G-kernel) labelled by $M$ (see \cite{Kuijlaars-Zhang14}), so there is no phase transition as $M \to \infty$ in the usual sense. We refer to  \cite{Akemann-Burda-Kieburg18} for more discussion on the hard edge.

To state our main results exactly, recalling the product model $\Pi_M=X_M \cdots X_2 X_1$ where $X_1, \ldots, X_M$ are assumed to be i.i.d.~$N\times N$ complex Ginibre matrices, it is easy to see from \eqref {2integral}  that  the  eigenvalues of a random  Hermitian matrix  $\log (\Pi_{M}^{*}\Pi_{M})$  also build a determinantal point process with  correlation kernel
\begin{equation} \label{transformK}
  \K_N(x, y)
  = \int_{c-i\infty}^{c + i\infty} \frac{ds}{2\pi i} \oint_{\Sigma} \frac{dt}{2 \pi i} \frac{e^{xt-ys}}{s-t} \frac{\Gamma(t)}{\Gamma(s)}  \left( \frac{\Gamma(s+N)}{\Gamma(t+N)}\right)^{M+1}, \quad x, y \in \realR.
\end{equation}
 Let   $\xi_k$ be the $k$-th largest eigenvalue of  $\log(\Pi^*_M \Pi_M)$, then    the distribution function of   the largest eigenvalue  $\xi_1$ admits  a Fredholm determinant representation   (see e.g. \cite[Lemma 3.2.4]{Anderson-Guionnet-Zeitouni10}, \cite[Chapter 9]{Forrester10})
  \begin{equation} \label{larg}
    \Prob(\xi_1 \leq x) = \det(I - \mathbf{\K}_N) = 1+\sum_{n=1}^{N} \frac{(-1)^n}{n!} \int_{x}^{\infty} \cdots \int_{x}^{\infty}    \det[ \K_N(t_i, t_j) ]_{i,j=1}^n \,dt_1 \cdots dt_n,
  \end{equation}
where $\mathbf{\K}_N$ is the integral operator acting on  $L^2((x, +\infty))$ with the kernel $\K_N$.

We always assume that both $M$ and $N$ go to infinity and $M$ may depend on $N$. Specifically, we need to
  divide three different regimes  according to the relative ratio of  $M$ and $N$:
\begin{enumerate}[label=\Roman*), ref=\Roman*]
\item \label{enu:case_1}
  Weakly correlated regime where $M/N \rightarrow \infty$;
\item \label{enu:case_2}
  Intermediate regime where  $M/N \rightarrow \gamma \in (0,\infty)$;
\item \label{enu:case_3}
  Strongly correlated regime where $M/N \rightarrow 0$.
\end{enumerate}

Below in all the three cases, we use the scaling function $g(\cdot)$ or $g(k; \cdot)$ to facilitate our statement of the results. Although the scaling function clearly depends on $M$ and $N$, we suppress the dependence for notational simplicity.

\begin{thm}[Normality in case \ref{enu:case_1}] \label{cor:normality}
  Suppose that   $\lim_{N\to\infty} M/N = \infty$. For a fixed  $k \in \natN$,  let $x_N(k) = N(\psi(1 - k + N) - \log N)$ and
   \begin{equation} \label{eq:change_of_v_norm}
    g(k; \xi) = N\log N + x_N(k) + \xi\sqrt{\frac{N}{M + 1}}.
  \end{equation}
 Then the following hold uniformly for any $\xi, \eta$ in a compact set  of $\realR$.
  \begin{enumerate}
  \item \label{enu:thm:normal_1} For the correlation kernel  \eqref{transformK},
    \begin{equation}  \label{M>N}
      \lim_{N \to \infty} \sqrt{\frac{M + 1}{N}}
      e^{\sqrt{\frac{M+1}{N}}(k-1)(\xi-\eta)}
       \K_N\left(\frac{M + 1}{N} g(k; \xi), \frac{M + 1}{N} g(k; \eta)\right) = \frac{1}{\sqrt{2\pi}}  e^{-\frac{1}{2}\eta^2}.
    \end{equation}
  \item \label{largnormal}
    For   the $k$-th largest eigenvalue  $\xi_k$ of  $\log(\Pi^*_M \Pi_M)$,
    \begin{equation}
      \lim_{N \to \infty}  \Prob\Big(  \xi_k \leq   \frac{M + 1}{N}  g(k;\xi )\Big)=\int^{\xi}_{-\infty} \frac{1}{\sqrt{2\pi}} e^{-\frac{1}{2}t^2} dt.
    \end{equation}
  \end{enumerate}
\end{thm}

To  describe the edge statistics   in the critical case    (see [4, eq(13)] or Section 3.2 below for the bulk statistics),  we need to  define   a new family of correlation kernels,  depending on a parameter $\gamma \in (0,\infty) $,  as
    \begin{equation} \label{softlimitK}
      K_{\crit}(x, y; \gamma) = \int^{1 + i\infty}_{1 - i\infty} \frac{ds}{2\pi i} \oint_{\Sigma_{-\infty}} \frac{dt}{2\pi i} \frac{1}{s - t} \frac{\Gamma(t)}{\Gamma(s)} \frac{e^{\frac{\gamma s^2}{2} - y s}}{e^{\frac{\gamma t^2}{2} - xt}},
    \end{equation}
    where the    anticlockwise contour $\Sigma_{-\infty}\subset\{z\in \compC: \Re z<1\}$, starts from $-\infty - i\epsilon$, encircles   $\{ 0, -1, -2, \dotsc \}$, and then goes to $-\infty + i\epsilon$ for some $\epsilon>0$.  Accordingly,  we introduce  a family of Fredholm determinants
  \begin{equation} \label{largcritical}
    F_{\text{crit}}(x; \gamma):= \det(I - \mathbf{ K_{\crit}}) = 1+\sum_{n=1}^{\infty} \frac{(-1)^n}{n!} \int_{x}^{\infty} \cdots \int_{x}^{\infty}    \det[  K_{\crit}(t_i, t_j; \gamma)  ]_{i,j=1}^n \,dt_1 \cdots dt_n,
  \end{equation}
where $\mathbf{ K_{\crit}}$ is the integral operator acting on  $L^2((x, +\infty))$ with the kernel $  K_{\crit}$.
  $ F_{\text{crit}}(x; \gamma)$ is  a    continuous and nondecreasing   function in $x$, such that
$ F_{\text{crit}}(x; \gamma) \to 1$ as $x \to \infty$ (cf. proof of part  \ref{enu:thm:crit_2} of Theorem \ref{thm:crit}).
 Although a direct proof that $ F_{\text{crit}}(x; \gamma) \to 0$ as $x \to -\infty$ is not trivial, and it is not given here, we do not pursue it and leave it as an open problem, see Question \ref{que:1} in Section \ref{sec:open_ques}.

\begin{thm}[Criticality in case \ref{enu:case_2}] \label{thm:crit}
  Suppose that  $\lim_{N\to\infty} M/N =\gamma \in (0, \infty)$. Let
  \begin{equation}
    g(\xi) = (M+1)\left(\log N - \frac{1}{2N}\right) + \xi.
  \end{equation}
Then the following hold uniformly for any $\xi, \eta$ in a compact set  of $\realR$.
  \begin{enumerate}
  \item \label{enu:thm:crit_1}
     For the correlation kernel  \eqref{transformK},
    \begin{equation} \label{M=N}
      \lim_{N \to \infty}   \K_N \big(g(\xi) , g(\eta) \big) = K_{\crit}(\xi, \eta; \gamma).
    \end{equation}
  \item \label{enu:thm:crit_2}  For the largest eigenvalue $\xi_1$ of $\log(\Pi_{M}^{*}\Pi_M)$,
  \begin{equation} \label{distri}
    \lim_{N \to \infty} \Prob(\xi_1 \leq g(\xi)) =F_{\mathrm{crit}}(\xi; \gamma).
  \end{equation}
  \end{enumerate}
\end{thm}
\begin{rem}
  Part \ref{enu:thm:crit_1} of Theorem \ref{thm:crit} can be stated in an alternative way. We let
  \begin{equation} \label{eq:two_crit_rel}
    e^{(\eta-\xi)t_0}K_{\crit}(\xi, \eta; \gamma)= \widehat{K}_{\crit}(\xi-\gamma t_0, \eta-\gamma t_0; \gamma),
  \end{equation}
  where $t_0$ is the unique positive solution of $\psi'(t_0) = \gamma$, such that $\psi$ is the digamma function as in \eqref{eq:digamma_series}. Then
  \begin{equation} \label{M=Nhat}
    \widehat{K}_{\crit}(\xi, \eta; \gamma) = \int^{1 + i\infty}_{1 - i\infty} \frac{ds}{2\pi i} \oint_{\widehat{\Sigma}_{-\infty}} \frac{dt}{2\pi i} \frac{1}{s - t} \frac{\Gamma(t + t_0)}{\Gamma(s + t_0)} \frac{e^{\frac{\gamma s^2}{2} - \eta s}}{e^{\frac{\gamma t^2}{2} - \xi t}}
  \end{equation}
  with $\widehat{\Sigma}_{-\infty}$ starting from $-\infty - i\epsilon$, encircling $\{ -t_0, -t_0 - 1, -t_0 - 2, \dotsc \}$ in positive direction, and then going to $-\infty + i\epsilon$. $t_0$ decreases as $\gamma$ increases, as
  \begin{equation} \label{eq:behaviour_t_0}
    t_0 = \frac{1}{\sqrt{\gamma}} + \bigO\Big(\frac{1}{\gamma}\Big) \quad \text{as $\gamma \to +\infty$}, \quad t_0 = \frac{1}{\gamma} + \bigO(1) \quad \text{as $\gamma \to 0$}.
  \end{equation}
  With this choice, we know that  the Taylor expansion of $\log \Gamma(t+t_0) -\frac{1}{2}\gamma t^2$ at zero has vanishing quadratic  and    non-vanishing cubic  terms, which are  consistent with the definition of the Airy function. 
  The explicit transition from the critical case (case II) to cases I and III will be given in Theorem \ref{thm1.4}.
\end{rem}


In case \ref{enu:case_3}, we state a more general result that covers not just the limiting distribution of the largest singular values, but also local limiting correlations of the singular values in the ``bulk''. We also give a definition of the celebrated Airy kernel by
\begin{equation} \label{eq:Airy}
  K_{\Airy}(x, y) = \int_{\mathcal{C}^{\infty}_<} \frac{ds}{2\pi i} \int_{\Sigma^{\infty}_>} \frac{dt}{2\pi i} \frac{1}{s - t} \frac{e^{\frac{s^3}{3} - ys}}{e^{\frac{t^3}{3} - xt}}, \end{equation}
where the contours are upwards and parametrized as
\begin{equation} \label{eq:defn_Airy_contours}
  \begin{split}
    \Sigma^\infty_> = {}& \{ -1 + re^{2\pi i/3} \mid 0 \leq r < \infty \} \cup \{ -1 + re^{\pi i/3} \mid -\infty < r \leq 0 \}, \\
    \mathcal{C}^\infty_< = {}& \{ 1 + re^{4\pi i/3} \mid -\infty < r \leq 0 \} \cup \{ 1 + re^{5\pi i/3} \mid 0 \leq r <\infty \}.
  \end{split}
\end{equation}
The GUE Tracy-Widom distribution is thus given by
\begin{equation}
  F_{\text{GUE}}(x) = \det(I - \mathbf{K}_{\Airy}),
\end{equation}
where $\mathbf{K}_{\Airy}$ is the integral operator on $L^2((x, +\infty))$ with the kernel $K_{\Airy}$.

\begin{thm} [GUE statistics in case \ref{enu:case_3}] \label{thm1.3}
  Suppose that  $\lim_{N\to\infty} M/N =0$ and $\lim_{N\to\infty} M =\infty$. With $\theta \in [0, \pi)$, let
  \begin{equation}
    g(\xi) = M \log N + \log(M+1) + v_M(\theta) + \frac{\xi}{\rho_{M, N}},
  \end{equation}
  with
  \begin{equation} \label{parameter-vm}
    v_M(\theta) =
    \begin{cases}
      (M + 1)\log \sin\theta  -M\log \sin\frac{M\theta}{M + 1}   -\log  \sin \frac{\theta}{M + 1} -\log (M + 1), & \theta \in (0, \pi), \\
      M \log(1 + M^{-1}), & \theta = 0,
    \end{cases}
  \end{equation}
  and
  \begin{equation} \label{eq:form_rho}
    \rho_{M, N} = \rho_{M, N}(\theta) =
    \begin{cases}
      \frac{N \sin\theta \sin\frac{\theta}{M + 1}}{\pi \sin(1-\frac{1}{M + 1})\theta} ,
      & \theta \in (0, \pi), \\
      2^{\frac{1}{3}} \Big(\frac{N}{M+1}\Big)^{\frac{2}{3}}, & \theta = 0.
    \end{cases}
  \end{equation}
  Then the  following hold uniformly  for $\xi,\eta$ in a compact subset of $\mathbb{R}$.
    \begin{enumerate}
      \item \label{enu:thm:sine}
        When $\theta \in (0, \pi)$,
        \begin{equation} \label{eq:thm:sine}
          \lim_{N \rightarrow \infty} e^{- \pi (\xi - \eta)   \cot\theta}  \frac{1}{\rho_{M, N}}  \K_{N}\big(g(\xi), g(\eta)\big) = \frac{\sin\pi(\xi-\eta)}{\pi(\xi - \eta)}.
        \end{equation}
      \item \label{enu:thm:airy}
       When  $\theta = 0$,
        \begin{equation} \label{eq:thm:airy}
          \lim_{N \rightarrow \infty} e^{-\frac{N}{M+1}  \frac{\xi -\eta}{\rho_{M,N}}} \frac{1}{\rho_{M,N}} \K_{N}\big(g(\xi), g(\eta)\big) = K_{\Airy}(\xi, \eta).
        \end{equation}
      \item \label{enu:thm:airy_trace} For the largest eigenvalue $\xi_1$ of $\log(\Pi_{M}^{*}\Pi_M)$,
  \begin{equation} \label{distri_TW}
    \lim_{N \to \infty} \Prob( \xi_1  \leq g(\xi)) = F_{\mathrm{GUE}}(\xi).
  \end{equation}
    \end{enumerate}
  \end{thm}

  Note that $v_{M}(\theta)$ is monotonically decreasing  as $\theta \in [0,\pi)$  such that for a fixed $M$, $v_{M}(0) \to -M\log(M/(M + 1))$ as $\theta \to 0$ and $v_{M}(\theta) \to -\infty$ as $\theta \to \pi$. In Theorem \ref{thm1.3}, $\theta \in [0, \pi)$ parmetrizes the spectrum of singular values from the right end through the bulk.

\begin{rem}
   Parts  \ref{enu:thm:sine}  and  \ref{enu:thm:airy} of Theorem \ref{thm1.3} are  straightforward extensions of \cite[Theorems 1.1 and 1.2]{Liu-Wang-Zhang14} where $M$ is assumed to be a fixed integer, such that formally part \ref{enu:thm:sine} differs from \cite[Theorem 1.1]{Liu-Wang-Zhang14} by change of variables and the conjugation $e^{(x - y)N}$, while part \ref{enu:thm:airy} agrees with the $M \to \infty$ formal limit of \cite[Theorem 1.2]{Liu-Wang-Zhang14}.
   But it is not obvious   from \cite{Liu-Wang-Zhang14}  that the sine and Airy kernels will still hold true  when $M$ goes to  infinity much more slowly  than  $N$. The convergence of the largest eigenvalue  in part \ref{enu:thm:airy_trace}  of Theorem \ref{thm1.3}  was not studied  in \cite{Liu-Wang-Zhang14}.
Actually,   in Theorems \ref{cor:normality}  and \ref{thm1.3} we need to choose $N$- and $M$-dependent spectral representations $x_{N}(k)$ and $v_M(\theta)$ respectively in order to investigate the local statistics. These may indicate the complexity and precision of local fluctuation of eigenvalues; see \cite{Akemann-Burda-Kieburg18} for relevant numerical stimulation.  But, noting $x_{N}(k)=0.5-k+\bigO(1/N)$ by \eqref{eq:defn_u_w} bellow,
    $\rho_{M, N} =  N/(\pi (M+1))\big(1 + \bigO( 1/(M+1) ) \big)$  as in \eqref{eq:form_rho} and $v_M(\theta)=v_{\infty}(\theta)+\bigO(1/(M+1))$ where $v_{\infty}(\theta)$ is defined in \eqref{vtheta} that is independent of $M$, one can replace $x_{N}(k)$ by $0.5-k$ in part \ref{enu:thm:normal_1} of Theorem \ref{cor:normality}  if $N \ll M \ll N^3$ and replace $v_M(\theta)$ by $v(\theta)$ in part \ref{enu:thm:sine} of Theorem \ref{thm1.3} if $M \ll N \ll M^2$.
\end{rem}

\begin{rem} To the best of our knowledge, there are at least two kinds of transitions between the GUE Tracy-Widom  and  Gaussian distributions  in RMT, one for  the largest eigenvalue of the deformed GUE ensemble, which is  found by Johansson (see \cite{Johansson07}, \cite{Liechty-Wang18}), and the other for  the largest eigenvalue of  the spiked  complex Wishart matrix,  proved by Baik, Ben Arous and P\'{e}ch\'{e} \cite{Baik-Ben_Arous-Peche05}.
  We believe that   the  phase transition  in Theorem \ref{thm:crit}  is  different from the former  two,  at least   seen from the  expression of  correlation kernels. To see the difference between these transitions, one may consider the correlation kernel $K(x, y)$ in each case, and compute  the asymptotics of $K(x, x)$ as $x \to + \infty$. But we omit the details here. Besides, the soft edge kernel   in  part  \ref{enu:thm:crit_1}  of Theorem \ref{thm:crit} was independently obtained by Akemann, Burda and Kieburg \cite{Akemann-Burda-Kieburg18}. Though different in form, our integral representation \eqref{softlimitK} should be equivalent to \cite[Equation (19)]{Akemann-Burda-Kieburg18}.
\end{rem}

The rest of this article is organized as follows. In the next Section \ref{sec:proofs} we prove the main theorems stated above. In Section \ref{sect:discuss} we discuss a few relevant questions.

\section{Proofs of main theorems} \label{sec:proofs}

In the proofs, $B(q,r)$ denotes an open ball in the complex plane with center $q \in \compC$ and radius $ r>0$, $\epsilon$ may mean different constants in different formulas, and $\bigO$ and $o$ are used in the usual sense. We use $\psi(z)$ to denote the digamma function \cite[5.2.2]{Boisvert-Clark-Lozier-Olver10}, which admits a series representation \cite[5.7.6, 5.15.1]{Boisvert-Clark-Lozier-Olver10} for $z\neq 0, -1,-2,\ldots$
\begin{equation} \label{eq:digamma_series}
  \psi(z) = -\gamma_0 + \sum^{\infty}_{n = 0} \left( \frac{1}{n + 1} - \frac{1}{n + z} \right), \quad \psi'(z) = \sum^{\infty}_{n = 0} \frac{1}{(n + z)^2},
\end{equation}
where $\gamma_0$ is the Euler constant. By Stirling's formula (see e.g.  \cite[5.11.1, 5.11.2]{Boisvert-Clark-Lozier-Olver10}),
 we have  as $z\to \infty$  in the sector  $|\arg z| \leq \pi - \epsilon$ for all $\epsilon > 0$, uniformly
\begin{gather}
  \log\Gamma(z)= (z-\frac{1}{2})\log z  -z+\log\sqrt{2\pi} +\frac{1}{12z} + \bigO \Big( \frac{1}{z^2} \Big), \label{stirling} \\
  \psi(z) =   \log z -\frac{1}{2z} + \bigO \Big( \frac{1}{z^2} \Big). \label{digammaa}
\end{gather}


\subsection{Proofs of Theorems \ref{cor:normality} and \ref{thm:crit}} \label{subsec:proofs_1.1_1.2}

Before the proof of the theorems, we define several functions to be used later. Let $N, M$ be fixed positive integers, and $w$ be a real parameter. We define  a function $F(t; w)$ of complex variable $t$, depending on $N$ and $w$, as
\begin{equation} \label{eq:defn_F}
  F(t; w) = (\log N + w/N)t - \log \Gamma(t + N) + \log \Gamma(N),
\end{equation}
where $\log$ takes the principal branch. It is easy to see that
\begin{equation} \label{eq:F_derivatives}
  F'(t; w) = (\log N + w/N) - \psi(t + N), \quad F''(t; w) = -\psi'(t + N).
\end{equation}

We see from \eqref{eq:digamma_series} that $\Re \psi'(z) > 0$ if $\arg z \in (-\pi/4, \pi/4)$, and as $x \in \realR_+$ runs from $0$ to $+\infty$, $\psi(x)$ runs from $-\infty$ to $+\infty$ monotonically. Hence for any $w \in \realR$, there is a unique $t_w$ such that
\begin{equation} \label{eq:defn_t_w}
  F'(t_w; w) = 0 \quad \text{and} \quad t_w \in (-N, +\infty),
\end{equation}
and we have that
\begin{equation} \label{eq:prop_t_w}
  \text{$t_w$ depends on $w$ monotonically, $t_w \to +\infty$ as $w \to +\infty$ and $t_w \to -N$ as $w \to -\infty$}.
\end{equation}
Hence by \eqref{digammaa} we have for $w$ in a compact subset of $\realR$,
\begin{equation} \label{eq:defn_u_w}
  t_w = \frac{1}{2} + w + \frac{C_{t_w}}{N},
\end{equation}
where $C_{t_w}$ is bounded uniformly as $N \to \infty$. In the subsequent proofs, we  need the property that as $s = c + iy$, the function $\Re F(s; w)$ increases fast  enough as $iy$ goes upward from $0$ to $+i\infty$ or downward from $0$ to $-i\infty$, where $c$ and $w$ are in a compact subset of $\realR$. To see it in a more precise way, we just need to compute $-\Im F'(c + iy; w)$, since
\begin{equation} \label{eq:estimate_F'}
  \begin{split}
    \frac{d}{dy} \Re F(c + iy; w) = {}& -\Im F'(c + iy; w) = \Im \psi(c + iy + N) = -\sum^{\infty}_{n = 0} \Im \frac{1}{n + N + c + iy} \\
    = {}& \sum^{\infty}_{n = 0} \frac{y}{(n + N + c)^2 + y^2} = \int^{\infty}_{N + c} \frac{(1 + \bigO(N^{-1}))y dx}{x^2 + y^2} \\
    = {}& \arctan(N^{-1} y)\big(1 + C_{c, y, w} N^{-1}\big),
  \end{split}
  \end{equation}
where $C_{c, y, w}$ is bounded in $\realR$. We also use the following positively oriented contours (see Figure \ref{intcontour1} bellow): For any $a \in (-N + 1, 1)$, $\Sigma_-(a) = \Sigma^1_-(a) \cup \Sigma^2_-(a) \cup \Sigma^3_-(a) \cup \Sigma^4_-(a) \cup \Sigma^5_-$, where
\begin{equation} \label{eq:defn_Sigma_-}
  \begin{gathered}
    \Sigma^1_-(a) = \{ a - \frac{2 - i}{4} t \mid t \in [0, 1] \}, \quad \Sigma^2_-(a) = \{ a - \frac{2 + i}{4} + \frac{2 + i}{4} t \mid t \in [0, 1] \}, \\
    \Sigma^3_-(a) = \{ -t + \frac{i}{4} \mid t \in [\frac{1}{2} - a, N - \frac{1}{2}] \}, \quad \Sigma^4_-(a) = \{ t - \frac{i}{4} \mid t \in [-N + \frac{1}{2}, a - \frac{1}{2}] \}, \\
    \Sigma^5_- = \{ -N + \frac{1}{2} - it \mid t \in [-\frac{1}{4}, \frac{1}{4}] \}.
  \end{gathered}
\end{equation}
Similarly, for any $b \in (-N + 1, 1/2)$, $\Sigma_+(b) = \Sigma^1_+(b) \cup \Sigma^2_+(b) \cup \Sigma^3_+(b) \cup \Sigma^4_+(b) \cup \Sigma^5_+$, where
\begin{equation}
  \begin{gathered}
    \Sigma^1_+(b) = \{ b + \frac{2 - i}{4} t \mid t \in [0, 1] \}, \quad \Sigma^2_+(b) = \{ b + \frac{2 + i}{4} - \frac{2 + i}{4} t \mid t \in [0, 1] \}, \\
    \Sigma^3_+(b) = \{ t - \frac{i}{4} \mid t \in [b + \frac{1}{2}, 1] \}, \quad \Sigma^4_+(b) = \{ -t + \frac{i}{4} \mid t \in [-1, -b - \frac{1}{2}] \}, \\
    \Sigma^5_+ = \{ 1 + it \mid t \in [-\frac{1}{4}, \frac{1}{4}] \}.
  \end{gathered}
\end{equation}

\subsubsection{Proof of Theorem \ref{cor:normality}}

First we prove part \ref{enu:thm:normal_1} of Theorem \ref{cor:normality} and the technical result Theorem \ref{thm:normal} below. Next we prove part \ref{largnormal} based on them.
\begin{thm} \label{thm:normal}
  Under the setting of Theorem \ref{cor:normality},  let  $\epsilon > 0$, we have
  \begin{enumerate}
  \item \label{enu:thm:normal_3}
    There exists $C_0(\epsilon) > 0$ such that
        \begin{equation}
      \int^{+\infty}_{\frac{M + 1}{N}g(1; C_0(\epsilon))} \K_N(x, x) dx < \epsilon.
    \end{equation}
  \item \label{enu:thm:normal_2}
 For each fixed $k \in \natN$,  there exists $C_k(\epsilon) > 0$ such that
    \begin{equation}
      \int^{\frac{M + 1}{N} g(k; -C_k(\epsilon))}_{\frac{M+1}{N} g(k + 1; C_k(\epsilon))} \K_N(x, x) dx < \epsilon.
    \end{equation}
  \end{enumerate}
\end{thm}

\begin{proof}[Proofs of part \ref{enu:thm:normal_1} of Theorem \ref{cor:normality} and Theorem \ref{thm:normal}]
  Define
  \begin{equation} \label{eq:defn_x_N(k)}
    x_N(k) = N(\psi(1 - k + N) - \log N),
  \end{equation}
  such that $F'(t; x_N(k)) = 0$ is solved by $t_{x_N(k)} = 1 - k$. We set out the contours for $s$ and $t$ variables in \eqref{transformK}. For $s$, we define two possible vertical contours passing beside $1-k$ as
  \begin{equation} \label{Line}
    \mathcal{L}_{1-k}=\{  1- k + 2\sqrt{N/(M + 1)} + iy \mid y \in \realR \}, \quad \mathcal{R}_{1-k}=\{  1- k - 2\sqrt{N/(M + 1)} + iy \mid y \in \realR \}.
  \end{equation}
  (Yes, $\mathcal{L}_{1 - k}$ is to the right of $\mathcal{R}_{1 - k}$.) In the proof of part \ref{enu:thm:normal_1} of Theorem \ref{cor:normality} and part \ref{enu:thm:normal_3} of Theorem \ref{thm:normal}, we choose $\mathcal{L}_{1 - k}$. Moreover, we define
  \begin{equation}
    \mathcal{L}^{\loc}_{1 - k} = \{ z \in \mathcal{L}_{1 - k} \mid \lvert \Im z \rvert \leq N^{5/8} (M+1)^{-3/8} \}, \quad \mathcal{L}^{\glob}_{1 - k} = \mathcal{L}_{1-k} \setminus \mathcal{L}^{\loc}_{1 - k}.
  \end{equation}
  For $t$, let  $\Sigma_0(1 - k)$ be the positively oriented circle centered at $1 - k$ with radius $\sqrt{N/(M + 1)}$, and then choose the contour $\Sigma$ as the union of $ \Sigma_0(1 - k)$ and $\Sigma_-(1/2 - k) \cup   \Sigma_+(3/2 - k)$ with $k = 1, 2, 3, \dotsc$, while $\Sigma_+(3/2 - k)$ is set to be empty if $k=1$; see Figure \ref{intcontour1}. Accordingly, we divide the $t$-integral into two parts and express the kernel in \eqref{transformK} as
  \begin{equation} \label{eq:normal_kernel_rescaled}
    \exp\! \bigg((k-1)(\xi-\eta)\sqrt{\frac{M+1}{N}}  \bigg) \, \K_N \bigg( \frac{M + 1}{N} g(k; \xi), \frac{M + 1}{N} g(k; \eta) \bigg) = I_1 + I_2,
  \end{equation}
  where   $I_1:= I^{\loc}_1 + I^{\glob}_1$ and
  \begin{equation} \label{I12}
    \begin{gathered}
      I^*_1= \int_{ \mathcal{L}^*_{1-k}} \frac{ds}{2\pi i} \oint_{\Sigma_0(1 - k)} \frac{dt}{2\pi i} \frac{Q_k(s,t) q_k(s, t; \xi, \eta)}{s - t}, \quad * = \text{$\loc$, $\glob$ or blank}, \\
      I_2= \int_{ \mathcal{L}_{1-k}} \frac{ds}{2\pi i} \oint_{ \Sigma_-(1/2 - k) \cup   \Sigma_+(3/2 - k)} \frac{dt}{2\pi i} \frac{Q_k(s,t) q_k(s, t; \xi, \eta)}{s - t},
    \end{gathered}
  \end{equation}
  such that
  \begin{equation}
    Q_k(s,t)= \frac{\exp \Big( (M + 1)F(t; x_N(k)) \Big)}{\exp \Big(  (M + 1)F(s; x_N(k)) \Big)} \frac{\Gamma(t)}{\Gamma(s)}, \quad q_k(s, t; \xi, \eta) = \frac{\exp((t+k-1)\xi\sqrt{\frac{M+1}{N}})}{\exp((s+k-1)\eta \sqrt{\frac{M+1}{N}})} = \frac{e^{\xi \tau}}{e^{\eta \sigma}},
  \end{equation}
  upon the change of variables
  \begin{equation} \label{eq:normal_change_variable}
    s = (1 - k) + \sigma\sqrt{\frac{N}{M + 1}}, \quad   t = (1 - k) + \tau\sqrt{\frac{N}{M + 1}}.
  \end{equation}
 So the remaining task is to obtain asymptotic estimates of $I_1 = I^{\loc}_1 + I^{\glob}_1$ and $I_2$ as $N\to \infty$ and $M/N \to \infty$. The key ingredient here is to analyze the properties of the function $Q_k(s, t)$.
Since it is not hard to find the estimate  of $\Gamma(t)/\Gamma(s)$ for  $t \in \Sigma$ and $s \in \mathcal{L}_{1 - k}$,
  we concentrate on the estimates of the functions $F(s; x_N(k))$ and $F(t; x_N(k))$ below.
 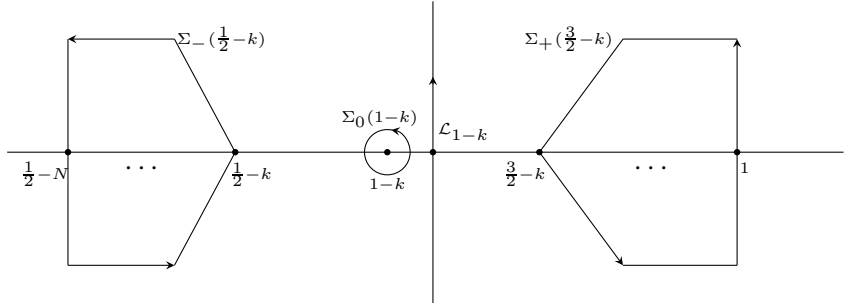
\begin{figure}[h]
      \centering
      \begin{tikzpicture}
        \draw (0, 0) -- (11, 0);
        \filldraw
          (5, 0) circle (1pt)
          (5, -0.2) node[below] {$\scriptscriptstyle 1-k$}
          (5.6, 0) circle (1pt)
          (6.0, 0) node[above] {$\scriptscriptstyle  \mathcal{L}_{1-k}
         $}
          (3, 0) circle (1pt)
          (3.2, 0) node[below] {$\scriptscriptstyle\frac{1}{2}-k$}
          (1.8, 0)  node[below] {$\cdots$}
          (0.8, 0) circle (1pt)
          (0.5, 0) node[below] {$\scriptscriptstyle\frac{1}{2}-N$}
          (9.6, 0) circle (1pt)
          (9.7, 0) node[below] {$\scriptscriptstyle 1$}
          (8.5, 0) node[below] {$\cdots$}
          (7, 0) circle (1pt)
          (6.8, 0) node[below] {$\scriptscriptstyle\frac{3}{2}-k$};

        \draw[->,>=stealth] (5.3, 0) arc (0:80:0.3);
        \draw (5.3, 0) arc (360:77:0.3);
        \draw (4.9, 0.2) node[above] {$\scriptscriptstyle\Sigma_{0}(1-k)$};
        \draw (0.8,1.5) -- (0.8,-1.5);
        \draw[->,>=stealth] (0.8,-1.5) -- (2.2,-1.5);
        \draw (2.2,-1.5) -- (3,0);
        \draw (3,0) -- (2.2,1.5);
        \draw[->,>=stealth] (2.2,1.5) -- (0.8,1.5);
        \draw (2.1, 1.5) node[right] {$\scriptscriptstyle\Sigma_{-}(\frac{1}{2}-k)$};
        \draw[->,>=stealth] (9.6,-1.5) -- (9.6,1.5);
        \draw (9.6,1.5) -- (8.1,1.5);
        \draw (8.1,1.5) -- (7,0);
        \draw[->,>=stealth] (7,0) -- (8.1,-1.5);
        \draw (8.1,-1.5) -- (9.6,-1.5);
        \draw (8.1, 1.5) node[left] {$\scriptscriptstyle\Sigma_{+}(\frac{3}{2}-k)$};
        \draw[->,>=stealth] (5.6, -2) -- (5.6, 1);
        \draw (5.6, 0.9) -- (5.6, 2);
    \end{tikzpicture}
    \caption{Schematic contours in the proof of Theorem \ref{cor:normality}}\label{intcontour1}
  \end{figure}

   When $t \in \Sigma_0(1 - k)$, under the change of variables in \eqref{eq:normal_change_variable}, by the Taylor expansion with respect to $t$ at $1 - k$, we have uniformly for all $\tau = \tau' N^{1/8}(M+1)^{1/8}$ where $\tau'$ is in a compact subset of $\compC$,
  \begin{multline} \label{eq:estimate_MF}
    (M + 1) F(t; x_N(k))=(M + 1)F(1-k; x_N(k))  - \frac{1}{2}N\psi'(1 - k+N) \tau^2 \\
    + C_{\tau'} (\tau')^3 N^{-\frac{1}{8}} (M + 1)^{-\frac{1}{8}},
  \end{multline}
  where $C_{\tau'}$ is bounded in $\compC$ and  by \eqref{digammaa}
  \begin{equation} \label{eq:new_label}
    N\psi'(1 - k+N) = 1+\bigO(1/N).
  \end{equation}
  When $s\in \mathcal{L}^{\loc}_{1-k}$, \eqref{eq:estimate_MF} is a good estimate for $F(s; x_N(k))$ with $t$ replaced by $s$. When $s \in \mathcal{L}^{\glob}_{1 - k}$, we have $s=1 - k + (2+iy)\sqrt{N/(M + 1)}$ with $|y| > N^{1/8}(M+1)^{1/8}$. Then for $N$ large enough we use the estimate \eqref{eq:estimate_F'} to find that there exists some $\epsilon > 0$ such that for all $|y| \geq  N^{1/8}(M+1)^{1/8}$,
  \begin{equation} \label{eq:estimate_vertical_normal}
    \Re (M + 1) F(s; x_{N}(k)) \geq   \Re (M + 1) F(s_{\pm}; x_{N}(k)) +
    \epsilon  N^{1/8}(M+1)^{1/8}
   (|y| - N^{1/8}(M+1)^{1/8}) ,
  \end{equation}
  where
  \begin{equation}
    s_{\pm}=1 - k + (2\pm iN^{1/8}(M+1)^{1/8} )\sqrt{N/(M + 1)}.
  \end{equation}
  To see the derivation of \eqref{eq:estimate_vertical_normal}, we assume, without loss of generality, $y > N^{1/8} (M + 1)^{1/8}$, and have by \eqref{eq:estimate_F'}
  \begin{equation}
    \begin{split}
      & \Re F(s; x_{N}(k)) - \Re F(s_{\pm}; x_{N}(k)) \\
      = {}& \int^{y \sqrt{N/(M + 1)}}_{N^{5/8} (M + 1)^{-3/8}} \frac{d}{d\zeta} \Re F \left( 1 - k + 2\sqrt{\frac{N}{M + 1}} + i \zeta, x_N(k) \right) d\zeta \\
      = {}& \int^{y \sqrt{N/(M + 1)}}_{N^{5/8} (M + 1)^{-3/8}} \arctan(N^{-1} \zeta)\big(1 + C_{1 - k + 2\sqrt{N/(M + 1)}, \zeta, x_N(k)} N^{-1} \big) d\zeta,
    \end{split}
  \end{equation}
  and then a direct computation implies \eqref{eq:estimate_vertical_normal}. We further have by using \eqref{eq:estimate_MF} and \eqref{eq:new_label} that, as both $M \to \infty$ and $N \to \infty$,
   \begin{equation} \label{eq:estimate_vertical_normal-2}
   \Re (M + 1) \big(F(s_{\pm}; x_{N}(k)) -   F(1-k; x_{N}(k)) \big)= \frac{1}{2} N^{1/4}(M+1)^{1/4} (1+o(1)).
  \end{equation}

Under the change of variables \eqref{eq:normal_change_variable},  we have for $\sigma$ and $\tau$ in a compact subset of $\compC \setminus \{ 0 \}$, and $(M + 1)/N$ large enough,
  \begin{equation} \label{eq:est_Gamma}
    \frac{\Gamma(t)}{\Gamma(s)} = \frac{\sigma}{\tau} \Big( 1 + C_{\sigma,\tau} \big (| \sigma| +|\tau|\big) \sqrt{\frac{N}{M + 1}} \Big), \quad \frac{1}{s - t} = \frac{1}{\tau - \sigma} \sqrt{\frac{M + 1}{N}},
  \end{equation}
  where $C_{\sigma, \tau}$ is bounded in $\compC$. Thus by using \eqref{eq:estimate_MF} for the estimate of $F(t; x_N(k))$ and $F(s; x_N(k))$, using \eqref{eq:est_Gamma} for $\Gamma(t)/\Gamma(s)$ for small $\sigma, \tau$ and some other elementary estimates, we have
  \begin{multline} \label{eq:est_I^loc_1_norm}
    I^{\loc}_1 = \sqrt{\frac{N}{M + 1}} \int^{2 + iN^{1/8}(M + 1)^{1/8}}_{2 - iN^{1/8}(M + 1)^{1/8}} \frac{d\sigma}{2\pi i} \oint_{\lvert \tau \rvert = 1} \frac{d\tau}{2\pi i} \\
    \left( 1 + \bigO((N(M + 1))^{-1/8}) \right) \frac{e^{-\tau^2/2}}{e^{-\sigma^2/2}} \frac{\sigma}{\tau} \frac{e^{\xi \tau}}{e^{\eta \sigma}} \frac{1}{\sigma - \tau},
  \end{multline}
  where the factors $e^{-\tau^2/2}$ and $e^{-\sigma^2/2}$ come from \eqref{eq:estimate_MF} and \eqref{eq:new_label}, and by using \eqref{eq:estimate_vertical_normal-2} and some other elementary estimates, we have for some $\epsilon' > 0$
  \begin{multline} \label{eq:est_I^glob_1_norm}
    I^{\glob}_1 = \sqrt{\frac{N}{M + 1}} e^{-\frac{1}{2} N^{1/4}(M+1)^{1/4} (1+o(1))} \int_{\Re \sigma = 2,\ \lvert \Im \sigma \rvert > N^{1/8}(M + 1)^{1/8}} \frac{d\sigma}{2\pi i}   \oint_{\lvert \tau \rvert = 1} \frac{d\tau}{2\pi i}  \\
  \bigO \left(e^{-\epsilon'  N^{1/8}(M+1)^{1/8} (|\Im \sigma| - N^{1/8}(M+1)^{1/8})} \right) \frac{\Gamma(t)}{\Gamma(s)} \frac{e^{\xi \tau}}{e^{\eta \sigma}} \frac{1}{\sigma - \tau}.
  \end{multline}

  To estimate $I_{2}$ we need the following claim, whose proof is left in the end of this subsection.

  \begin{claim} \label{claim:1}
  There exists $\epsilon > 0$ such that the inequality
  \begin{equation} \label{eq:minor_est_normal}
    \Re (M + 1) \big( F(t; x_N(k)) - F(1-k; x_N(k))\big) \leq  -\epsilon \frac{M+1}{N} | t - (1 - k)|
  \end{equation}
  holds for all $t$ on $\Sigma_-(1/2 - k)$ and $\Sigma_+(3/2 - k)$.
  \end{claim}

  Together with the above claim, combing the estimates \eqref{eq:estimate_MF} for $s\in \mathcal{L}^{\loc}_{1-k}$, and \eqref{eq:estimate_vertical_normal} and \eqref{eq:estimate_vertical_normal-2} for $s \in \mathcal{L}^{\glob}_{1 - k}$,  noting that $\lvert s - t \rvert > 1/4$ when $(M+1)/N$ is large, we have for some $\epsilon, \epsilon' > 0$
  \begin{multline} \label{eq:est_I_2_norm}
    I_2 = \int_{ \mathcal{L}_{1-k}} \frac{ds}{2\pi i} \oint_{ \Sigma_-(1/2 - k) \cup   \Sigma_+(3/2 - k)} \frac{dt}{2\pi i}   \frac{e^{\xi \tau}}{e^{\eta \sigma}}\frac{\Gamma(t)}{\Gamma(s)}  \bigO \left( e^{-\epsilon \frac{M+1}{N} | t - (1 - k)|} \right) \\
    \times
    \left\{
      \begin{aligned}
        (1 + \bigO((N(M + 1))^{-1/8}) e^{-(s - 1 + k)^2(M + 1)N^{-1}/2}, & \quad \lvert \Im s \rvert < N^{5/8}(M + 1)^{-3/8}, \\
        e^{-\frac{1}{2} N^{1/4}(M+1)^{1/4}  - \epsilon'  N^{-3/8}(M+1)^{5/8} (|\Im s| - N^{5/8}(M+1)^{-3/8})}, & \quad \text{otherwise}.
      \end{aligned}
    \right\}
  \end{multline}
  Here the estimate of the integrand when $\lvert \Im s \rvert \geq N^{5/8}(M + 1)^{-3/8}$ is parallel to that in \eqref{eq:est_I^glob_1_norm}, and we note that all the $\bigO$ estimates above are from $Q_k(s, t)$ and are independent of $\xi, \eta$, and
    for $\xi,\eta$  in a compact set of $\realR$
   \begin{equation} \label{eq:key_right_most-0}
    \lvert q_k(s, t; \xi, \eta)  \rvert \leq   e^{C \sqrt{(M+1)/N} ( | t - (1 - k)|+|s - (1 - k)|)},
  \end{equation}
   for some $C > 0$.

  Now assume that $k = 1, 2, \dotsc$ and $\xi, \eta$ are in a compact subset of $\realR$. Using the standard steepest-descent technique on \eqref{eq:est_I^loc_1_norm}, \eqref{eq:est_I^glob_1_norm} and \eqref{eq:est_I_2_norm}, we have that $\lvert I^{\glob}_1 \rvert = o(\exp(- N^{1/4}(M+1)^{1/4}/4))$, $\lvert I_2 \rvert = o(\exp(-\epsilon N^{-1}(M+1)/4))$, and under the change of variables \eqref{eq:normal_change_variable}, as $N, M/N \to \infty$ 
  \begin{equation}
   I_1= \Big(1 + \bigO \Big(\sqrt{ \frac{N}{M + 1}} \Big) \Big)
   \sqrt{\frac{N}{M + 1}}
    \int^{2 + i\infty}_{2 - i\infty} \frac{d\sigma}{2\pi i} \oint_{\lvert \tau \rvert = 1} \frac{d\tau}{2\pi i} \frac{\exp(\sigma^2/2 - \eta\sigma)}{\exp(\tau^2/2 - \xi\tau)} \frac{1}{\sigma - \tau} \frac{\sigma}{\tau}, \label{I1asym}
  \end{equation}
  uniformly for $\xi, \eta$ in any compact subset of $\realR$. We see the integral $I_1$ concentrates on $I^{\loc}_{1 - k}$ and $I_2$ is negligible compared with $I_1$. Actually, it is easy to check that  the RHS   integral  equals     $e^{-\eta^2/2}/\sqrt{2\pi}$. Hence we prove part \ref{enu:thm:normal_1} of Theorem \ref{cor:normality}.

  Next, assume $k = 1$ and $\xi, \eta > 0$. Then as $s \in \mathcal{L}_0$ and $t \in \Sigma_0(0) \cup \Sigma_-(-1/2)$, we have
  \begin{equation} \label{eq:key_right_most}
    \lvert q_k(s, t; \xi, \eta) \rvert = |e^{\xi \tau-\eta \sigma} |\leq e^{\xi - 2\eta}, \quad \text{if $\xi, \eta > 0$}.
  \end{equation}
  Hence from \eqref{eq:est_I^loc_1_norm}, \eqref{eq:est_I^glob_1_norm} and \eqref{eq:est_I_2_norm}, noting that all the estimates are independent of $\xi, \eta$, we have that there exists some constant $C, \epsilon > 0$ such that
 \begin{equation} \label{Destimate1}
   |I^{\loc}_1| \leq  C \sqrt{ \frac{N}{M + 1}} e^{\xi-2\eta}, \quad \lvert I^{\glob}_1 \rvert \leq e^{-\epsilon N^{5/4}(M+1)^{1/4}}e^{\xi-2\eta}, \quad |I_2| \leq e^{- \epsilon \frac{M+1}{4N} } e^{\xi-2\eta},
  \end{equation}
  if $\xi, \eta >0$. Then by \eqref{eq:normal_kernel_rescaled} we conclude that there exist $C > 0$ such that for all $\xi,\eta > 0$
  \begin{equation} \label{eq:estimate_right_most}
    \left\lvert \K_N \bigg( \frac{M + 1}{N} g(k; \xi), \frac{M + 1}{N} g(k; \eta) \bigg) \right\rvert \leq C \sqrt{ \frac{N}{M + 1}} e^{\xi - 2\eta}.
  \end{equation}
  Under the change of variables \eqref{eq:change_of_v_norm} we obtain
  \begin{equation} \label{eq:integral_right_most_change}
    \int^{+\infty}_{\frac{M + 1}{N}g(1; C_0(\epsilon))} \K_N(x, x) dx = \sqrt{ \frac{M+1}{N}} \int^{+\infty}_{C_0(\epsilon)} \K_N\Big( \frac{M + 1}{N}g(1; \xi), \frac{M + 1}{N}g(1; \xi)\Big) d\xi \leq C e^{-C_0(\epsilon)}.
  \end{equation}
  Hence part \ref{enu:thm:normal_3} of Theorem \ref{thm:normal} is proved if $C_0(\epsilon) > \log(C/\epsilon)$.

  At last, we consider part \ref{enu:thm:normal_2} of Theorem \ref{thm:normal}. Suppose $x \in ((M + 1)N^{-1}g(k + 1; 0), (M + 1)N^{-1} g(k; 0))$ where $k = 1, 2, \dotsc$, then we can express
  \begin{equation} \label{eq:xi_and_xi'_together}
    x = \frac{M + 1}{N} g(k + 1; \xi) = \frac{M + 1}{N} g(k; \xi'),
  \end{equation}
  where $0<\xi,-\xi'  < \sqrt{(M + 1)/N}(x_N(k) - x_N(k + 1))$. We rewrite $\K_N(x, x) :=\K^{L}_N(x, x) + \K^{R}_N(x, x)$, where
  \begin{equation}
    \begin{split}
      K^L_N(x, x) = {}& \int_{\mathcal{L}_{-k}} \frac{ds}{2\pi i} \oint_{\Sigma_0(- k)\cup  \Sigma_-(-1/2 - k)} \frac{dt}{2\pi i} \frac{Q_{k + 1}(s,t) q_{k + 1}(s, t; \xi, \xi)}{s - t} \\
      K^R_N(x, x) = {}& \int_{\mathcal{R}_{1-k}} \frac{ds}{2\pi i} \oint_{\Sigma_0(1- k)\cup \Sigma_+(3/2 - k)} \frac{dt}{2\pi i} \frac{Q_k(s,t) q_k(s, t; \xi', \xi')}{s - t},
    \end{split}
  \end{equation}
  and $\Sigma_+(3/2 - k)$ is set to be empty if $k = 1$.

  Suppose $x > (M + 1)N^{-1}g(k + 1; 0)$ and then $\xi$ defined in \eqref{eq:xi_and_xi'_together} is positive. Then analogous to \eqref{eq:key_right_most}, we have $\lvert q_{k + 1}(s, t; \xi,  \xi) \rvert \leq e^{-\xi}$ when $s \in \mathcal{L}_{-k}$ and $t \in \Sigma_0(- k)\cup \Sigma_-(-1/2 - k)$, and analogous to \eqref{eq:estimate_right_most}, we have, for some $C > 0$,
  \begin{equation} \label{eq:estimate_left_between}
    \K^L_N \bigg( \frac{M + 1}{N} g(k+1; \xi), \frac{M + 1}{N} g(k+1; \xi) \bigg) \leq C \sqrt{ \frac{N}{M + 1}} e^{-\xi}.
  \end{equation}
  On the other hand, suppse $x < (M + 1)N^{-1}g(k; 0)$ and then $\xi'$ defined in \eqref{eq:xi_and_xi'_together} is negative. Symmetric to \eqref{eq:key_right_most}, we have $\lvert q_k(s, t; \xi,  \xi) \rvert \leq e^{\xi'}$ when $s \in \mathcal{R}_{1-k}$ and $t \in \Sigma_0(1- k)\cup \Sigma_+(3/2 - k)$, and symmetric to \eqref{eq:estimate_left_between}, we have, for some $C > 0$,
  \begin{equation}
    \K^R_N \bigg( \frac{M + 1}{N} g(k ; \xi'), \frac{M + 1}{N} g(k; \xi') \bigg) \leq C \sqrt{ \frac{N}{M + 1}} e^{\xi'}.
  \end{equation}
  Hence analogous to \eqref{eq:integral_right_most_change},
  \begin{equation}
    \begin{split}
      & \int^{\frac{M + 1}{N} g(k; -C_k(\epsilon))}_{\frac{M+1}{N} g(k + 1; C_k(\epsilon))} \K_N(x, x) dx \\
      = {}& \sqrt{\frac{M+1}{N}}
      \int^{   \sqrt{\frac{M+1}{N}} (x_N(k) - x_N(k + 1))   - C_k(\epsilon) }_{C_k(\epsilon)} \K^{L}_N \left( \frac{M + 1}{N} g(k + 1; \xi), \frac{M + 1}{N} g(k + 1; \xi) \right) d\xi
      \\
      & + \sqrt{\frac{M+1}{N}} \int^{-C_k(\epsilon)}_{  \sqrt{\frac{M+1}{N}}  (x_N(k + 1) - x_N(k)) + C_k(\epsilon)} \K^{R}_N \left( \frac{M + 1}{N} g(k; \xi'), \frac{M + 1}{N} g(k; \xi') \right) d\xi' \\
      \leq {}& 2C e^{-C_k(\epsilon)}.
    \end{split}
  \end{equation}
  Hence part \ref{enu:thm:normal_2} of Theorem  \ref{thm:normal} is proved if $C_k(\epsilon) > \log(2C/\epsilon)$.
\end{proof}

Now we continue to prove part \ref{largnormal} of  Theorem \ref{cor:normality}. We need a general result for point processes:
\begin{claim} \label{claim:2}
  If a point process has correlation functions $R^{(n)}(x_1, \dotsc, x_n)$ for $n = 1, 2, \dots$ (see e.g.~\eqref{eq:n-pt_corr_func}), then for any subset $E$, the probability $\Prob(E; n)$ that there are at least $n$ particles in $E$
  \begin{equation} \label{eq:general_prob_corr_ineq}
    \Prob(E; n) \leq \sum^{\infty}_{k = n} \binom{k}{n} \Prob(\text{there are exactly $k$ particles in $E$}) = \frac{1}{n!} \int_{E^n} R^{(n)}(x_1, \dotsc, x_n) d^nx.
  \end{equation}
\end{claim}

This result is not hard to see: The inequality is obvious, and the identity is from the definition of the $n$-point correlation function $R^{(n)}(x_1, \dotsc, x_n)$.

\begin{proof}[Proof of  part \ref{largnormal} of  Theorem \ref{cor:normality}]
  For notational simplicity, throughout this proof we denote $T=(M+1)/N$ and for $a < b$ let
  \begin{equation}
    I_{k}[a, b]= \left[T g(k; a),  Tg(k; b) \right].
  \end{equation}
  Then the probability $\Prob(I_{ k}[-C,C]; 2)$ that there are at least $2$ eigenvalues of $\log(\Pi^*_M \Pi_M)$ in the interval $I_{ k}[-C,C]$, satisfies the inequality that is a specialization of Claim \ref{claim:2}:
  \begin{equation} \label{eq:2pt_prob}
    \Prob(I_{ k}[-C,C]; 2) \leq \frac{1}{2} \int_{I_{ k}[-C,C]} \int_{I_{ k}[-C,C]}\widetilde{R}^{(2)}_N(x, y)dxdy,
  \end{equation}
  where $\widetilde{R}^{(2)}_N(x, y)$ is the $2$-point correlation function associated to kernel $\K$ in \eqref{transformK}. As $M, N \to \infty$ in the setting of Theorem \ref{cor:normality}, by part \ref{enu:thm:normal_1} of Theorem \ref{cor:normality}, the right-hand side of \eqref{eq:2pt_prob} converges to
  \begin{equation}
    \frac{1}{4\pi} \iint_{[-C, C]^2}
    \begin{vmatrix}
      e^{-\frac{1}{2}\xi^2} & e^{-\frac{1}{2}\eta^2} \\
      e^{-\frac{1}{2}\xi^2} & e^{-\frac{1}{2}\eta^2}
    \end{vmatrix}
     d\xi d\eta = 0.
  \end{equation}
  Hence we have that as $ N, M/N \to \infty$ in the setting of Theorem   \ref{cor:normality}, the probability that there are tow or more eigenvalues of $\log(\Pi^*_M \Pi_M)$ in $I_{ k}[-C,C]$ satisfies
  \begin{equation} \label{eq:2_pt_no}
    \lim_{N \to 0} \Prob( I_{ k}[-C,C]; 2) = 0.
  \end{equation}

  Now we consider the distribution of $\xi_k$, the $k$-th largest eigenvalue of $\log(\Pi^*_M \Pi_M)$. First, we consider the $k = 1$ case. For any $x \in \realR$, and any $\epsilon > 0$, we write $\Prob \left( \xi_1 \geq T g(1; x) \right) = P_1 + P_2$, where $P_1, P_2$ depend on $C > x$ as
  \begin{equation} \label{eq:P_1_P_2}
    P_1 = \Prob \left( \xi_1 \in I_{ 1}[x,C]\right), \quad P_2 = \Prob \left( \xi_1 > T g(1; C) \right),
  \end{equation}
  and note that
  \begin{equation}
    P_1 = \Prob \left(I_{ 1}[x,C]; 1 \right) - P_3,
  \end{equation}
  such that
  \begin{equation}
    P_3 = \Prob \left( \text{$\xi_1>Tg(1;C)$ and there is $\xi_j \in I_{ 1}[x,C]$ with $j > 1$} \right).
  \end{equation}
  To estimate $P_2$, we suppose that $C$ is large enough such that
  \begin{equation}
    \frac{1}{\sqrt{2\pi}} \int^{\infty}_C e^{-t^2/2}  dt <\frac{\epsilon}{2} \label{tail}
  \end{equation}
  and for large enough $N$, by part \ref{enu:thm:normal_3} of Theorem \ref{thm:normal},
  \begin{equation}
    \int^{\infty}_{Tg(1; C)} \K_N(x, x) dx < \frac{\epsilon}{4}.
  \end{equation}
  Hence by the $n = 1$ case of \eqref{eq:general_prob_corr_ineq}, for large enough $N$
  \begin{equation} \label{eq:est_right_to_infty}
    P_2 = \Prob \left( \left( T g(1; C), \infty \right); 1 \right) \leq \int^{\infty}_{T g(1; C)} \K_N(x, x) dx < \frac{\epsilon}{4}.
  \end{equation}
  Then by \eqref{eq:2_pt_no}, we have
  \begin{equation} \label{eq:P_3_P_2}
    \limsup_{N \to \infty} P_3 = \limsup_{N \to \infty} P_2 \leq \frac{\epsilon}{4}.
  \end{equation}
  Next, using \eqref{eq:2_pt_no} first and then part \ref{enu:thm:normal_1} of Theorem  \ref{cor:normality}, we have from \eqref{eq:general_prob_corr_ineq} that
  \begin{equation} \label{eq:bulk_norm_1}
    \lim_{N \to \infty} \Prob \left(  I_{ 1}[x,C]; 1 \right) = \lim_{N \to \infty} \int _{I_{ 1}[x,C]} \K_N(t, t) dt =  \frac{1}{\sqrt{2\pi}}\int^C_x e^{-\frac{t^2}{2}} dt.
  \end{equation}
  Thus we derive that
  \begin{equation} \label{eq:est_xi_1}
    \begin{split}
      & \limsup_{N \to \infty} \Big\lvert \Prob \left( \xi_1 \geq T g(1; x) \right) - \frac{1}{\sqrt{2\pi}}\int^{\infty}_x e^{-\frac{t^2}{2}}  dt \Big\rvert \\
      < {}& \limsup_{N \to \infty} \Big\lvert \Prob \left( \xi_1 \geq T g(1; x) \right) -  \frac{1}{\sqrt{2\pi}} \int^C_x e^{-\frac{t^2}{2}} dt \Big\rvert + \frac{\epsilon}{2}\leq \epsilon.
    \end{split}
  \end{equation}
  Since $\epsilon > 0$ is arbitrary and all results above hold   uniformly for $x$ in a compact set of $\realR$, we prove  part \ref{largnormal} of  Theorem \ref{cor:normality}
   for $k = 1$.

  Next, we consider $\xi_2$. For any $x \in \realR$ and any $\epsilon > 0$, analogous to \eqref{eq:P_1_P_2}, we write $\Prob \left( \xi_2 \geq T g(2; x) \right) = P_4 + P_5$, where $C$ is large enough and
  \begin{equation} \label{eq:P_4_P_5}
    P_4 = \Prob \left( \xi_2 \in  I_{2}[x,C]   \right), \quad P_5 = \Prob \left( \xi_2 > T g(2; C) \right).
  \end{equation}
  We denote $P_6, P_7, P_8$ and estimate them when $N$ is large enough as follows by parts \ref{enu:thm:normal_3} and \ref{enu:thm:normal_2} of Theorem \ref{thm:normal} and \eqref{eq:2_pt_no}
  \begin{align}
    P_6 := {}& \Prob \left( \xi_2 > T g(1; C) \right) \leq \Prob \left( \left( Tg(1; C), \infty \right); 1 \right) \leq \int^{\infty}_{T g(1; C)} \K_N(x, x) dx< \frac{\epsilon}{12}, \\
    P_7 := {} &\Prob \big( \xi_2 \in \left( T g(2; C), T g(1; -C) \right) \big) \leq {} \Prob \left( \left( T g(2; C), Tg(1; -C) \right); 1 \right) \notag\\
                \leq {}& \int^{T g(1; -C)}_{Tg(2; C)} \K_N(x, x) dx< \frac{\epsilon}{12},\\
    P_8 := {}& \Prob \left( \xi_2 \in I_{1}[-C,C] \right)= \Prob \left( \xi_2, \xi_1 \in I_{1}[-C,C]  \right) + \Prob \left( \xi_2 \in I_{1}[-C,C],   \xi_1 > T g(1; C) \right) \notag \\
    \leq {}& \Prob \left(  I_{1}[-C,C] ; 2 \right) + \Prob \left( \left( T g(1; C), \infty \right); 1 \right) \notag \\
    \leq {}& \Prob \left(  I_{1}[-C,C] ; 2 \right) + \int^{\infty}_{T g(1; C)} \K_N(x, x) dx< \frac{\epsilon}{12}.
  \end{align}
  Then we have
  \begin{equation} \label{eq:noncompact_xi_2}
    P_5 :=P_6 + P_7 + P_8<   \frac{\epsilon}{4}.
  \end{equation}

  On the other hand, note that
  \begin{equation}
    P_4 = \Prob \left( I_2[x,C] ; 1 \right) - P_9,
  \end{equation}
  such that
  \begin{equation}
    \begin{split}
      P_9 = {}& \Prob \left( \text{$\xi_2>T g(2;C)$ and there is $\xi_j \in  I_2[x,C]$ with $j > 2$} \right) +  \Prob \left( \xi_2<T g(1;x), \xi_1 \in   I_2[x,C]    \right) \\
      \leq {}& P_5 + \Prob(\xi_1 \in   I_2[x,C]).
    \end{split}
  \end{equation}
  From the limiting distribution result for $\xi_1$ that we just obtained, the probability that $\xi_1$ lies in $I_2[x,C]$ vanishes as $N \to \infty$. Hence analogous to \eqref{eq:P_3_P_2}
  \begin{equation}
    \limsup_{N \to \infty} P_9 \leq \limsup_{N \to \infty} P_5 \leq \frac{\epsilon}{4}.
  \end{equation}
  Then using \eqref{eq:2_pt_no} first and then part \ref{enu:thm:normal_1} of Theorem \ref{thm:normal}, we have analogous to \eqref{eq:bulk_norm_1}
  \begin{equation} \label{eq:compact_limit_xi_2}
    \lim_{N \to \infty} \Prob \left( I_2[x,C]; 1 \right) = \lim_{N \to \infty} \int_{I_2[x,C]}   \K_N(t, t) dt = \frac{1}{\sqrt{2\pi}}\int^C_x e^{-\frac{t^2}{2}} dt.
  \end{equation}
  At last,  when  $C$ is large enough such that \eqref{tail} is satisfied, we collect the results above and similar to \eqref{eq:est_xi_1}  derive
  \begin{equation}
    \limsup_{N \to \infty} \left\lvert \Prob \left( \xi_2 \geq T g(2; x) \right) -  \frac{1}{\sqrt{2\pi}} \int^{\infty}_x e^{-\frac{t^2}{2}} dt  \right\rvert < \epsilon,
  \end{equation}
  and by the arbitrariness of $\epsilon > 0$, we prove part \ref{largnormal} of  Theorem \ref{cor:normality} for $k = 2$.

  The $k = 3, 4, \dots$ cases can be proved similarly by induction, and we omit the details.
\end{proof}

\begin{proof}[Proof of Claim \ref{claim:1}]
 By \eqref{eq:F_derivatives} and \eqref{eq:digamma_series}, we have that uniformly for $t \in B(1 - k, 1 + k)$,
  \begin{equation}
    NF'''(t; x_N(k)) = N\sum^{\infty}_{n = 0} \frac{2}{(n + t + N)^3} = \bigO(N^{-1}).
  \end{equation}
  So as $t$ is on $\Sigma^1_-(1/2 - k) \cup \Sigma^2_-(1/2 - k)$ and $\Sigma_+(3/2 - k)$, $F(t; x_N(k))$ can be approximated by
  \begin{equation} N F(t; x_N(k)) =N F(1 - k; x_N(k)) - \frac{1}{2}N\psi'(1 - k+N)\big(t - (1 - k)\big)^2 + \bigO(N^{-1}) \end{equation}
   where $\lim_{N \to \infty} N\psi'(1 - k+N)= 1$, see \eqref{digammaa}.  Hence on this part of the contour, $\Re NF(t; x_N(k))$ decreases as $t$ moves away from $1 - k$, and we have
  \begin{equation} \label{eq:end_pts_deri}
     \frac{d}{dx} \Re NF(x \pm i/4; x_N(k)) \big\rvert_{x = -k} =1+o(1)> 0.
  \end{equation}

  Next, we have that for $t = x \pm i/4 \in \Sigma^3_-(1/2 - k) \cup \Sigma^4_-(1/2 - k)$,
  \begin{equation} \label{eq:est_deriv_hori_Sigma}
    \frac{d^2}{dx^2} \Re F(x \pm i/4; x_N(k)) = -\Re \psi'(x + N \pm i/4) = -\sum^{\infty}_{n = 0} \frac{(n + N + x)^2 - \frac{1}{16}}{((n + N + x)^2 + \frac{1}{16})^2} < 0.
  \end{equation}
  Together with the result \eqref{eq:end_pts_deri} on the right endpoints of $\Sigma^3_-(1/2 - k) \cup \Sigma^4_-(1/2 - k)$ and the left  endpoints of $\Sigma^3_+(3/2 - k) \cup \Sigma^4_+(3/2 - k)$, we conclude that on these horizontal contours, $\Re NF(t; x_N(k)) $ decreases as $t$ moves away from $1 - k$, and the decreasing rate is higher as $t$ is farther away.

  At last, for $t = -N + \frac{1}{2} + iy \in \Sigma^5_-$,  by \eqref{stirling} and  \eqref{digammaa} a  simple calculation shows   \begin{equation}\Re \big(F(t; x_N(k)) -F(1-k; x_N(k)) \big)= -(N+1-k)+ \bigO(1). \end{equation}

Collecting all the estimates, we derive the result in \eqref{eq:minor_est_normal}.
\end{proof}

\subsubsection{Proof of Theorem \ref{thm:crit}}

In our proof of Theorem \ref{thm:crit}, without loss of generality, we assume that $(M + 1)/N = \gamma$.

We need the  following estimates of $F(t; w)$ and $F(s; w')$   for $s \in \{ 1 + iy: y\in \realR \}$ and $t \in \Sigma_-(1/2)$. 
First,  for a fixed $w \in \realR$, if $\lvert t \rvert < N^{1/4}$, noting the assumption $M/N\to \gamma$, we have, by the Taylor expansion of $F(t; w)$ (cf. equation \eqref{stirling}), that
\begin{equation} \label{eq:est_MF_central_crit}
  (M + 1)F(t; w) = (M + 1)(\log N - \psi(N) + w/N)t - \frac{1}{2} (M + 1)\psi'(N) t^2 + \bigO(N^{-1/4}),
\end{equation}
and we have
\begin{equation} \label{eq:est_coe_crit}
  (M + 1)(\log N - \psi(N) + w/N) = (w + \frac{1}{2})\gamma + \bigO(N^{-1}), \quad (M + 1)\psi'(N) = \gamma + \bigO(N^{-1}).
\end{equation}
So as  $s, t = \bigO(N^{1/4})$, we have uniformly  for $w,w'$  in a compact set of $\realR$ (similarly for $s, w'$)
\begin{equation} \label{eq:limit_crit}
  (M + 1)F(t; w) = -\frac{1}{2} \gamma t^2 + (w + \frac{1}{2})\gamma t + \bigO(N^{-1/4}).  
\end{equation}

Next we estimate the integrand when either $s$ or $t$ is not $\bigO(N^{1/4})$. For $s = 1 + iy$, by \eqref{eq:estimate_F'} we have similar to \eqref{eq:estimate_vertical_normal} that there exists $\epsilon > 0$ such that
\begin{equation} \label{eq:est_F_vert}
  \Re (M + 1)F(1 + iy; w') \geq
    \Re (M + 1) F(1 \pm i N^{\frac{1}{4}}) + \epsilon N^{\frac{1}{4}} (y - N^{\frac{1}{4}}), \quad \text{for $|y| \geq N^{\frac{1}{4}}$}.
\end{equation}
Particularly, it dominates $-\log \Gamma(1+iy)$  on the vertical contour since   $-\log|\Gamma(1+iy)|\sim \frac{1}{2}\pi |y|$ as $|y| \to \infty$ (cf. \cite[5.11.9]{Boisvert-Clark-Lozier-Olver10}).

On the other hand, we have that for $t \in \Sigma_-(1/2) \setminus B(0, N^{1/4})$, or $t \in \Sigma_-(-1/2)$, like \eqref{eq:minor_est_normal}, we have for some $\epsilon > 0$
\begin{equation} \label{eq:est_exp_F_crit}
  e^{(M + 1)F(t; w)} \leq e^{-\epsilon \lvert t \rvert}.
\end{equation}
Its  proof is quite analogous to that  of \eqref{eq:minor_est_normal}. In fact,
 as $t$ moves  along $\Sigma^3_-(1/2)\cup \Sigma^4_-(1/2)$),  we can prove   that 
  \begin{equation}
     \frac{d}{dx} \Re NF(x \pm i/4; w) \big\rvert_{x = -1} =1+o(1)> 0,  \quad  \frac{d^2}{dx^2} \Re F(x \pm i/4; w)<0,
  \end{equation}
and for $t \in \Sigma^5_-$
  \begin{equation}
    \Re F(t; w) = -N +\bigO(1).
  \end{equation}
 These estimates  guarantee     \eqref{eq:est_exp_F_crit}.

Also in this proof, we introduce the countour $\Sigma'_-(a) = \Sigma^1_-(a) \cup \Sigma^2_-(a) \cup \Sigma^{3, '}_-(a) \cup \Sigma^{4, '}_-(a)$, such that $\Sigma^1_-(a)$ and $\Sigma^2_-(a)$ are defined in \eqref{eq:defn_Sigma_-}, and
\begin{equation}
  \Sigma^{3, '}_-(a) = \{ -t + \frac{i}{4} \mid t \in [\frac{1}{2} - a, +\infty) \}, \quad \Sigma^{4, '}_-(a) = \{ t - \frac{i}{4} \mid t \in (-\infty, a - \frac{1}{2}] \}.
\end{equation}
We may specialize $\Sigma_{\infty}$ in the statement of Theorem \ref{thm:crit} as $\Sigma'_-(1/2)$.
\begin{proof}[Proof of part \ref{enu:thm:crit_1} of Theorem \ref{thm:crit}]
  With $w, w'$ in a compact subset of $\realR$,
  \begin{multline} \label{eq:alt_K_N_crit}
    \K_N\big((M + 1)(\log N + w/N), (M + 1)(\log N + w'/N)\big) = \\
    \int^{1 + i\infty}_{1 - i\infty} \frac{ds}{2\pi i} \oint_{\Sigma_-(1/2)} \frac{dt}{2\pi i} \frac{\exp \big[ (M + 1)F(t; w) \big]}{\exp \big[ (M + 1)F(s; w') \big]} \frac{\Gamma(t)}{\Gamma(s)} \frac{1}{s - t},
  \end{multline}
  where the contour $\Sigma_-(1/2)$ is defined by \eqref{eq:defn_Sigma_-}.
  By the estimates of $F(t; w)$ given above in \eqref{eq:est_MF_central_crit} and \eqref{eq:est_exp_F_crit}, we know that the integrand of the double contour integral \eqref{eq:alt_K_N_crit} vanishes fast as $t$ moves away from $0$ along $\Sigma_{-(1/2)}$, and by \eqref{eq:est_F_vert}, the integrand vanishes fast as $s$ moves away from $0$ along the line $\{z \in \mathbb{C}: \Re z = 1 \}$. These estimates, together with the bahaviour of $\Gamma(z)$ as $z \to \infty$, imply  that the double contour integral \eqref{eq:alt_K_N_crit} concentrates on the region $s, t \in B(0, N^{1/4})$, and by approximation \eqref{eq:limit_crit}, we have 
  \begin{multline} \label{eq:limit_kernel_crit}
     \K_N \big((M + 1)(\log N + w/N), (M + 1)(\log N + w'/N)\big) = \\
   \left(1 + \bigO(N^{-1/4})\right) \int^{1 + iN^{1/4}}_{1 - iN^{1/4}} \frac{ds}{2\pi i} \oint_{\Sigma'_-(1/2) \cap B(0, N^{1/4})} \frac{dt}{2\pi i} \frac{e^{\frac{\gamma}{2} s^2 - \gamma(w' + 1/2)s}}{e^{\frac{\gamma}{2} t^2 - \gamma(w + 1/2)t}} \frac{\Gamma(t)}{\Gamma(s)}\frac{1}{s - t}.
  \end{multline}

  It is obvious that if we change $N^{1/4}$ to be $\infty$, \eqref{eq:limit_kernel_crit} also holds. At last, by changing $w, w'$ in the proof by $\gamma^{-1}\xi - 1/2$ and $\gamma^{-1}\eta - 1/2$ respectively, we prove part \ref{enu:thm:crit_1}.
\end{proof}

Part \ref{enu:thm:crit_2} of Theorem \ref{thm:crit} follows from a technical result:
\begin{lem} \label{lem:first_trace_norm}
  Let $C \in \realR$, $\mathbf{\K}_N$ and $\mathbf{K}_{\crit}$ be the integral operators on $L^2((C, +\infty))$ whose kernel are $\K_N(g(\xi), g(\eta))$  defined in Theorem \ref{thm:crit} and $K_{\crit}(\xi, \eta; \gamma)$ respectively. Then $\mathbf{\K}_N$ and $\mathbf{K}_{\crit}$ are trace class operators and as $N \to \infty$, $\mathbf{\K}_N \to \mathbf{K}_{\crit}$ in trace norm.
\end{lem}

Below we prove the trace norm convergence by a widely used strategy \cite{Reed-Simon80}: Let $*$ denote either a natural number or $\infty$. If integral operators $\mathbf{K}_*$, $\mathbf{G}_*$ and $\mathbf{H}_*$ on $L^2((C, +\infty))$ satisfies $\mathbf{K}_* = \mathbf{G}_* \mathbf{H}_*$, then $\mathbf{K}_*$ is trace norm class if both $\mathbf{G}_*$ and $\mathbf{H}_*$ are Hilbert-Schmidt operators \cite[Theorem VI 22(h)]{Reed-Simon80}, and furthermore,  $\mathbf{K}_n \to \mathbf{K}_{\infty}$ in trace norm if both $\mathbf{G}_n \to \mathbf{G}_{\infty}$ and $\mathbf{H}_n \to \mathbf{H}_{\infty}$ in Hilbert-Schmidt norm \cite[Exercise VI 28(c)]{Reed-Simon80}. More concretely, if $\mathbf{K}_*$, $\mathbf{G}_*$, $\mathbf{H}_*$ are represented by kernel functions $K_*(x, y)$, $G_*(x, r)$, $H_*(r, y)$ respectively, then the factorization $\mathbf{K}_* = \mathbf{G}_* \mathbf{H}_*$ is equivalent to
\begin{equation}
  K_*(x, y) = \int^{\infty}_C G_*(x, r) H_*(r, y) dr.
\end{equation}
$\mathbf{G}_*$, $\mathbf{H}_*$ are Hilbert-Schmidt operators if $\iint_{(C, \infty)^2} dx dy \lvert X_*(x, y) \rvert^2 < \infty$ where $X = G$ or $H$ \cite[Theorem VI 23]{Reed-Simon80}, and the two Hilbert-Schmidt convergences are equivalent to (ibid.)
\begin{equation}
  \lim_{n \to \infty} \iint_{(C, \infty)^2} dx dy \lvert X_n(x, y) - X_{\infty}(x, y) \rvert^2 = 0, \quad X = G \text{ or } H. 
\end{equation}
\begin{proof}[Proof of Lemma \ref{lem:first_trace_norm}]
  We note that the kernels $\K_N(g(x), g(y))$ and $K_{\crit}(x, y; \gamma)$ can be written as
  \begin{equation} \label{eq:convolutions}
    \int^{\infty}_C G_N(x, r)H_N(r, y) dr \quad \text{and} \quad \int^{\infty}_C G_{\infty, \gamma}(x, r) H_{\infty, \gamma}(r, y) dr
  \end{equation}
  respectively, where
  \begin{align}
    G_N(x, r) = {}& \oint_{\Sigma_-(1/2)} g_N(t; x, r) \frac{dt}{2\pi i}, & g_N(t; x, r) = {}& e^{(M + 1)F(t; \frac{x}{\gamma} - \frac{1}{2}) + t(r - C)} \Gamma(t), \label{eq:G_N} \\
    G_{\infty, \gamma}(x, r) = {}& \oint_{\Sigma'_-(1/2)} g_{\infty}(t; x, r) \frac{dt}{2\pi i}, & g_{\infty}(t; x, r) = {}& e^{-\frac{\gamma t^2}{2} + (x + r - C)t} \Gamma(t) \label{eq:G_infty} \\
    H_N(r, y) = {}& \int^{1 + i\infty}_{1 - i\infty} h_N(s; r, y) \frac{ds}{2\pi i}, & h_N(s; r, y) = {}& \frac{e^{-(M + 1)F(s; \frac{y}{\gamma} - \frac{1}{2}) - s(r - C)}}{ \Gamma(s)}, \\
    H_{\infty, \gamma}(r, y) = {}& \int^{1 + i\infty}_{1 - i\infty} h_{\infty}(s; r, y) \frac{ds}{2\pi i}, & h_{\infty}(s; r, y) = {}& \frac{e^{\frac{\gamma s^2}{2} - (y + r - C)s}}{\Gamma(s)}.
  \end{align}
  Hence we only need to show
  \begin{align}
    \lim_{N \to \infty} \iint_{(C, \infty)^2} dx dr \left\lvert G_N(x, r) - G_{\infty, \gamma}(x, r) \right\rvert^2 = {}& 0, \label{eq:HS_norm_G} \\
    \lim_{N \to \infty} \iint_{(C, \infty)^2} dr dy \left\lvert H_N(r, y) - H_{\infty, \gamma}(r, y) \right\rvert^2 = {}& 0. \label{eq:HS_norm_H}
  \end{align}

  To prove \eqref{eq:HS_norm_G}, we divide the integral domain $(C, \infty)^2$ into two regions: $I_1(R) = \{ (x, r) \mid x \geq C,\ r \geq C, \text{ and } x + r \leq R \}$ and $I_2(R) = \{ (x, r) \mid x \geq C,\ r \geq C, \text{ and } x + r > R \}$, where $R > C$ is a constant.

  On the region $I_1(R)$, by the estimate \eqref{eq:limit_crit}, we have that
  \begin{equation}
    g_N(t; x, r) = g_{\infty}(t; x, r)(1 + \bigO(N^{-1/4}))
  \end{equation}
  on their common part of the contour $\Sigma_-(1/2) \cap B(0, N^{1/4}) = \Sigma'_-(1/2) \cap B(0, N^{1/4})$ as $M, N \to \infty$, where the $\bigO(N^{-1/4})$ term is uniform with respect to $(x, r) \in I_1$ and $t$ on the contour. On the other hand, on the remaining contour $\Sigma_-(1/2) \setminus B(0, N^{1/4})$, by estimate \eqref{eq:est_exp_F_crit}, the inegral in \eqref{eq:G_N} is negligible uniformly for $(x, r) \in I_1(R)$, and also it is straightforward to check that on the remaining contour $\Sigma'_-(1/2) \setminus B(0, N^{1/4})$, the integral in \eqref{eq:G_infty} is negligible uniformly for $(x, r) \in I_1(R)$. Hence we prove that on $I_1(R)$, $G_N(x, r) - G_{\infty, \gamma}(x, r)$ approaches $0$ uniformly.

  We note that both  the integrands in the definitions of $G_N(x, r)$ and $G_{\infty, \gamma}(x, r)$  have residue $1$ at $t = 0$, so we can deform the contours $\Sigma_-(1/2)$ and $\Sigma'_-(1/2)$ in the two definitions, and have
  \begin{equation} \label{eq:diff_G}
    G_N(x, r) - G_{\infty, \gamma}(x, r) = \oint_{\Sigma_-(-1/2)} g_N(t; x, r) \frac{dt}{2\pi i}  - \oint_{\Sigma'_-{(-1/2)}} g_{\infty}(t; x, r) \frac{dt}{2\pi i}.
  \end{equation}
  For all $t \in \Sigma_-(-1/2)$, by inequality \eqref{eq:est_exp_F_crit} and that $\lvert e^{(x + r - C)t} \rvert \leq e^{-(x + r - C)/2}$ if $x + r - C> 0$, we have that for $(x, y) \in I_2(R)$,
  \begin{equation} \label{eq:line_int_G_N}
    \lvert g_N(t; x, r) \rvert \leq e^{-\epsilon \lvert t \rvert} \cdot e^{-(x + r - C)/2} \cdot \lvert \Gamma(t) \rvert.
  \end{equation}
  Hence we have $\lvert G_N(x, r) \rvert \leq \text{constant} \cdot e^{-(x + r - C)/2}$. Analogously, we verify by similar method that $\lvert G_{\infty}(x, r) \rvert \leq \text{constant} \cdot e^{-(x + r - C)/2}$ for $(x, y) \in I_2(R)$. Hence by \eqref{eq:diff_G} we conclude that $\lvert G_N(x, r) - G_{\infty, \gamma}(x, r) \rvert < \text{constant} \cdot e^{-(x + r - C)/2}$ for $(x, r) \in I_2$.

  Using the arbitrariness of $R$, we prove \eqref{eq:HS_norm_G} by the dominated  convergence theorem.

  The convergence \eqref{eq:HS_norm_H} can be proved in the same way: Dividing the integral domain $[C, \infty)^2$ into $I_1(R)$ and $I_2(R)$, depending on $r + y \leq R$ or $> R$, for any $R > C$, then showing that $\lvert H_N(r, y) - H_{\infty, \gamma}(r, y) \rvert$ converges uniformly to $0$ in $I_1(R)$, and is bounded by $\text{constant} \cdot e^{-(r - C + y)}$ in $I_2(R)$ (here we use $e^{-(r - C + y)}$, because $\Re s = 1$ for all $s$ on the contour).  Then the convergence is proved. Technically it is simpler because we do not need to deform the integral contour. The details are omitted.
\end{proof}

\subsection{Proof of Theorem \ref{thm1.3}} \label{sect2.2}

The proofs of  parts \ref{enu:thm:sine}  and \ref{enu:thm:airy}  of Theorem \ref{thm1.3}   are similar to the proofs of \cite[Theorems 1.1 and 1.2]{Liu-Wang-Zhang14} where $M$ is assumed to a fixed integer as $N\to \infty$.    The main difference is that we need to choose an $M$-dependent spectral parameter $v_{M}(\theta)$  defined in  \eqref{parameter-vm}   and choose a different contour for $t$. In order to see  why the sine and Airy kernels still hold even when $M$ goes to infinity  but much more slowly than $N$, we need to do some more-refined analyses.

Before the proof, we set up some notation. Let
\begin{multline}
  f_{M, N}(t; \theta) = \frac{(M + 1)^2}{N} \log\Gamma \left( N + \frac{N}{M + 1} t \right) - \frac{M + 1}{N} \log\Gamma \left( \frac{N}{M + 1} t \right) \\
  - (M\log N + \log(M + 1) + v_M(\theta))t + C_{M, N},
\end{multline}
where the variable $t \in \compC$, the parameter $\theta \in [0, \pi)$,   $v_{M}(\theta)$ is defined by \eqref{parameter-vm} and
\begin{equation}
  C_{M,N} = (M + 1)^{2}\Big( 1 - \big( 1- \frac{1}{2N}\big) \log N -  \frac{1}{2N} \log(2\pi)\Big) +\frac{M + 1}{2N}\log\frac{M + 1}{N}.
\end{equation}
Also,  let
\begin{equation} \label{fm-infty}
  f_M(t; \theta) = (M + 1)^{2}(1 + \frac{t}{M + 1})\log(1 + \frac{t}{M + 1}) - t\log{t} - Mt - v_{M}(\theta)t.
\end{equation}
By Stirling's formula \eqref{stirling}, for $t$ in a compact subset of $\compC$ whose distance to $\realR_-$ is positive, we have that, as $N/M$ tends to infinity,
\begin{equation} \label{eq:est_f_MN_by_fm}
  f_{M,N}(t; \theta) = f_{M}(t; \theta) + \frac{M+1}{N}\varphi_{M,N}(t; \theta),
\end{equation}
where $\varphi_{M,N}(t; \theta)$ is uniformly bounded and analytic in certain compact subset involving the chosen contour below. We also note that
\begin{equation} \label{eq:deriv_f_MN}
  f'_{M, N}(t; \theta) = (M + 1) \psi(N + \frac{N}{M + 1}t) - \psi(\frac{N}{M + 1}t) - \big(M \log N + \log (M + 1) + v_{M}(\theta)\big),
\end{equation}
and by \eqref{digammaa} and \eqref{eq:est_f_MN_by_fm}, for $t$ in a compact subset of $\compC$ whose distance to $\realR_-$ is positive, as $N/M \to \infty$,
\begin{equation} \label{eq:est_f'_MN}
  f'_{M, N}(t;\theta) = f'_M(t; \theta) + \bigO(M/N),
  \quad f'_M(t; \theta) = (M + 1)\log(1 + t/(M + 1)) - \log t - v_{M}(\theta).
\end{equation}

Next we define for $\phi \in (-\pi, \pi)$
\begin{equation} \label{eq:defn_q_M_h_M}
  q_{M}(\phi) = (M + 1)\left( \frac{ \sin\phi}{\sin(1-\frac{1}{M + 1})\phi}e^{i\frac{\phi}{M + 1}}-1 \right), \quad \text{and} \quad h_{M}(\phi) = \re{q_{M}(\phi)},
\end{equation}
and note that $h_M$ maps $[0, \pi)$ bijectively to $(-M - 1, (M + 1)/M]$. Then we define the contours
\begin{equation} \label{eq:defn_Sigma_pm}
  \Sigma_+ = \{t = q_{M}(\phi) \mid \phi \in [0, \pi)\}, \quad \Sigma_- = \{t = q_{M}(-\phi) = \overline{q_{M}(\phi)} \mid \phi \in [0, \pi)\},
\end{equation}
and
\begin{equation} \label{eq:defn_h}
  \mathcal{C}_{\theta} = \{h_{M}(\theta) + iy \mid y \in \realR \}.
\end{equation}
We also note that as $M$ is large, for $\theta$ in a compact subset of $(-\pi, \pi)$,
\begin{align}
  v_M(\theta) = {}& v_{\infty}(\theta) + \bigO(M^{-1}), & v_{\infty}(\theta) = {}& \theta\cot\theta + \log \frac{\theta}{\sin\theta},\label{vtheta} \\
  q_M(\theta) = {}& q_{\infty}(\theta) +  \bigO(M^{-1}), & q_{\infty}(\theta) = {}& \frac{\theta}{\sin\theta} e^{i\theta}, \label{eq:q_M_q_infty}
\end{align}
and then for $t$ in a compact subset of $\compC$ that is away from $\realR_-$,
\begin{equation} \label{eq:approx_f_M}
  f_M(t; \theta) = f_{\infty}(t; \theta) + \bigO(M^{-1}), \quad f_{\infty}(t; \theta) = \frac{1}{2}t^2 - t(\log t - 1) - v_{\infty}(\theta)t.
\end{equation}

The proof of Theorem \ref{thm1.3} needs the following technical lemma, whose proof will be given in Section \ref{subsubsec:proof_lemma_saddle_t}:
\begin{lem} \label{saddle-t}
  Let $f_M(t; \theta)$ be defined in \eqref{fm-infty},  then we have the following:
  \begin{enumerate}
  \item \label{enu:saddle-t:1}
    Let  $\theta \in [0, \pi)$. As $t$ runs along $\Sigma_+$ (resp.~$\Sigma_-$), $\Re f_M(t; \theta)$ attains its unique minimum at the point $q_M(\theta)$ (resp.~$q_M(-\theta)$). Moreover, we have
    \begin{equation}
      \frac{d}{d\phi} \Re f_M(q_M(\phi); \theta)
      \begin{cases}
        < 0, & \phi \in (0, \theta), \\
        > 0, & \phi \in (\theta, \pi),
      \end{cases}
      \quad
      \frac{d}{d\phi} \Re f_M(\overline{q_M(\phi)}; \theta)
      \begin{cases}
        < 0, & \phi \in (0, \theta), \\
        > 0, & \phi \in (\theta, \pi).
      \end{cases}
    \end{equation}
  \item \label{enu:saddle-t:2}
    Let $\theta \in (0, \pi)$. As $s$ runs along $\mathcal{C}_{\theta}$, $\Re f_M(s; \theta)$ attains its global maximum at two points $q_M(\pm\theta)$. Moreover,
    \begin{equation}
      \frac{d}{dy} \Re f_M(\Re q_M(\pm \theta) + iy; \theta)
      \begin{cases}
        < 0, & y > \Im q_M(\theta), \\
        > 0, & y \in (0, \Im q_M(\theta)), \\
        < 0, & y \in (\Im q_M(-\theta), 0), \\
        > 0, & y < \Im q_M(-\theta).
      \end{cases}
    \end{equation}
  \item \label{enu:saddle-t:3}
    Let $\realR \ni x_0 > h_M(0) = 1 + M^{-1}$, and let $\tilde{\mathcal{C}} = \{ x_0 + iy \mid y \in \realR \}$. Then $\Re f_M(s; 0)$ attains its global maximum at $x_0$. Moreover,
    \begin{equation}
      \frac{d}{dy} \Re f_M(x_0 + iy; 0)
      \begin{cases}
        < 0, & y > 0, \\
        > 0, & y < 0.
      \end{cases}
    \end{equation}
  \end{enumerate}
\end{lem}

\subsubsection{Proof of part \ref{enu:thm:sine} of Theorem \ref{thm1.3}}
  In this subsubsection, $\theta \in (0, \pi)$ is a constant, and $f_M(t), f_{M, N}(t), \rho_{M, N}$ stand for $f_M(t; \theta), f_{M, N}(t; \theta), \rho_{M, N}(\theta)$ respectively.

  After change of variables $s, t \to sN/(M + 1),  tN/(M + 1)$,  we obtain

  \begin{equation} \label{eq:airy_kernel_rescaled}
    \K_{N}(g(\xi),g(\eta)) = \frac{N}{M + 1} \int_{c-i\infty}^{c + i\infty} \frac{ds}{2\pi i} \oint_{\Sigma_{(M + 1)/N}} \frac{dt}{2 \pi i} \frac{1}{s - t} e^{\frac{N}{M + 1}\big( f_{M,N}(s) - f_{M,N}(t)  \big)} e^{\frac{N}{M + 1} \frac{\xi t - \eta s}{\rho_{M, N}}},
  \end{equation}
  where  $\Sigma_{(M + 1)/N}$ is a positive contour enclosing $0, -(M + 1)/N, -2(M + 1)/N, \dotsc, -(N - 1)(M + 1)/N$, and $c$ is chosen to make the vertical contour away from $\Sigma_{(M + 1)/N}$.

  For the asymptotic analysis, we deform $\Sigma_{(M + 1)/N}$ into $\Sigma^1_+ \cup \Sigma^2_+ \cup \Sigma^3_+ \cup \Sigma^4 \cup \Sigma^1_- \cup \Sigma^2_- \cup \Sigma^3_- \cup \Sigma_L \cup \Sigma_R$ which is defined below. Recall mapping $h_{M}$ defined in \eqref{eq:defn_q_M_h_M}. Since $h_{M}$ is bijective, its inverse $h_{M}^{-1}: (-M - 1, (M + 1)/M] \to [0, \pi)$ is well defined. Let  $C>0$ (independent of $M, N$) be sufficiently large such that $\log C + v_{M}(\theta) > 0$ for all large enough $M$. Then with $\Sigma_{\pm}$ defined in \eqref{eq:defn_Sigma_pm}, a small enough $\delta>0$, we define (without loss of generality, we assume that $N h_M(\theta) \notin \intZ$)
  \begin{equation} \label{eq:defn_Sigama_case3}
    \begin{split}
      \Sigma^1_{\pm} = {}& \Big\{ t \in \Sigma_{\pm} \mid 0 \leq \pm\arg t \leq \theta - \delta \Big\},\\
      \Sigma^2_{\pm} = {}& \Big\{ t \in \Sigma_{\pm} \mid \pm \arg t > \theta + \delta \text{ and } \Re t \geq -C \Big\}, \\
      \Sigma^3_{\pm} = {}& \Big\{ t = x \pm i \im\{q_{M}(h_{M}^{-1}(-C))\} \mid x \in [-(M + 1) + \frac{M + 1}{2N}, -C] \Big\}, \\
      \Sigma^4 = {}& \Big\{t = -(M + 1) + \frac{M + 1}{2N} + iy \mid y \in [-\im q_{M}(h_{M}^{-1}(-C)), \im q_{M}(h_{M}^{-1}(-C))]\Big\}.
    \end{split}
  \end{equation}
  Let $\Sigma_L$ be the vertical line segment connecting the right ends of $\Sigma^2_{\pm}$, and $\Sigma_R$ be the vertical line segment connecting the left ends of $\Sigma^1_{\pm}$. On the other hand, we choose the contour for $s$ to be the vertical line $\mathcal{C}_{\theta}$ defined in  \eqref{eq:defn_h}.

  Hence $\Sigma_{(M + 1)/N}$ is the union of two separate closed contours $\Sigma^1_+ \cup \Sigma_R \cup \Sigma^1_-$ and $\Sigma^2_+ \cup \Sigma^3_+ \cup \Sigma^4 \cup \Sigma^3_- \cup \Sigma^2_- \cup \Sigma_L$; cf. Figure \ref{intcon-sine}. We assume they are both positively oriented. The contour for $s$ goes between these two closed contours. We also divide $\Sigma_{(M + 1)/N}$ into the ``outer'' part $\Sigma_{\out} = \Sigma^1_+ \cup \Sigma^2_+ \cup \Sigma^3_+ \cup \Sigma^4 \cup \Sigma^1_- \cup \Sigma^2_- \cup \Sigma^3_-$ and the ``inner'' part $\Sigma_{\inner} = \Sigma_L \cup \Sigma_R$.

  \begin{figure}[h]
      \centering
      \begin{tikzpicture}
        \draw (-6,0) -- (2,0);
        \draw[domain=-pi+0.8:-pi+0.96,->,>=stealth] plot({\x*(cos(\x r))*(sin(\x r))^(-1)},{\x});
        \draw[domain=-pi+0.95:-pi+1.12] plot({\x*(cos(\x r))*(sin(\x r))^(-1)},{\x});
        \draw (-1.6,-pi+0.8) node[above]{$\Sigma_{-}^{2}$};
        \draw[->,>=stealth] (-4.51, {-pi+.8}) -- (-3.5, {-pi+.8});
        \draw (-3.52, {-pi+.8}) -- ({-2.27}, {-pi+0.8});
        \draw (-3.5,-pi+0.65) node[above]{$\Sigma_{-}^{3}$};
        \draw[domain=-1:pi/3.3,->,>=stealth] plot({\x*(cos(\x r))*(sin(\x r))^(-1)},{\x});
        \draw (0.9,1.5) node[below]{$\Sigma_{+}^{1}$};
        \draw[domain=0.93:3*pi/5-0.15] plot({\x*(cos(\x r))*(sin(\x r))^(-1)},{\x});
        \draw (0.9,-1.6) node[above]{$\Sigma_{-}^{1}$};
        \draw[domain=-3*pi/5+0.15:-0.95,->,>=stealth] plot({\x*(cos(\x r))*(sin(\x r))^(-1)},{\x});
        \draw[domain=pi-1.12:pi-0.95,->,>=stealth] plot({\x*(cos(\x r))*(sin(\x r))^(-1)},{\x});
        \draw[domain= pi-0.96:pi-.8] plot({\x*(cos(\x r))*(sin(\x r))^(-1)},{\x});
        \draw (-1.6,pi-.9) node[below]{$\Sigma_{+}^{2}$};
        \draw[->,>=stealth] ({-2.27}, {pi-.8}) -- (-3.5, {pi-.8});
        \draw (-3.45, {pi-.8}) -- (-4.51, {pi-.8});
        \draw (-3.5,pi-.8) node[below]{$\Sigma_{+}^{3}$};
        \draw[->,>=stealth] (-4.5, {pi-.8}) -- (-4.5, {1});
        \draw[->,>=stealth] (-4.5, {1.05}) -- (-4.5, {-1.05});
        \draw (-4.5, {-1}) -- (-4.5, {-pi+.8});
        \draw (-4.5,0.5) node[left]{$\Sigma^{4}$};

        \draw[->,>=stealth] ({3*pi*(cos(deg(3*pi/5-0.165)))*((sin(deg(3*pi/5-0.165)))^(-1))/5}, {3*pi/5-0.15}) -- ({3*pi*(cos(deg(3*pi/5-0.165)))*((sin(deg(3*pi/5-0.165)))^(-1))/5}, -0.72);
        \draw ({3*pi*(cos(deg(3*pi/5-0.165)))*((sin(deg(3*pi/5-0.165)))^(-1))/5}, -0.7) -- ({3*pi*(cos(deg(3*pi/5-0.165)))*((sin(deg(3*pi/5-0.165)))^(-1))/5}, {-3*pi/5+0.15});
        \draw (-0.4,0.5) node[right]{$\Sigma_{R}$};
        \draw ({3*pi*(cos(deg(3*pi/5+0.165)))*((sin(deg(3*pi/5+0.165)))^(-1))/5}, {3*pi/5+0.15}) -- ({3*pi*(cos(deg(3*pi/5+0.165)))*((sin(deg(3*pi/5+0.165)))^(-1))/5}, 0.7);
        \draw[<-,>=stealth] ({3*pi*(cos(deg(3*pi/5+0.165)))*((sin(deg(3*pi/5+0.165)))^(-1))/5}, 0.72) -- ({3*pi*(cos(deg(3*pi/5+0.165)))*((sin(deg(3*pi/5+0.165)))^(-1))/5}, {-3*pi/5-0.15});
        \draw (-0.9,-0.5) node[left]{$\Sigma_{L}$};
        \draw[->,>=stealth] ({3*pi*(cos(deg(3*pi/5)))*((sin(deg(3*pi/5)))^(-1))/5}, {-pi}) -- ({3*pi*(cos(deg(3*pi/5)))*((sin(deg(3*pi/5)))^(-1))/5}, {pi-.5});
        \draw ({3*pi*(cos(deg(3*pi/5)))*((sin(deg(3*pi/5)))^(-1))/5}, {pi-.52}) -- ({3*pi*(cos(deg(3*pi/5)))*((sin(deg(3*pi/5)))^(-1))/5}, {pi});
        \draw (0,pi-.7) node[left]{$\mathcal{C}_{\theta}$};
        \filldraw
          ({3*pi*(cos(deg(3*pi/5)))*((sin(deg(3*pi/5)))^(-1))/5}, {3*pi/5}) circle (1pt) node[right] {$t_{+}$}
          ({3*pi*(cos(deg(3*pi/5)))*((sin(deg(3*pi/5)))^(-1))/5}, {-3*pi/5}) circle (1pt) node[right] {$t_{-}$};
      \end{tikzpicture}
      \caption{Schematic contours in the proof of part \ref{enu:thm:sine} of Theorem \ref{thm1.3}.}\label{intcon-sine}
    \end{figure}
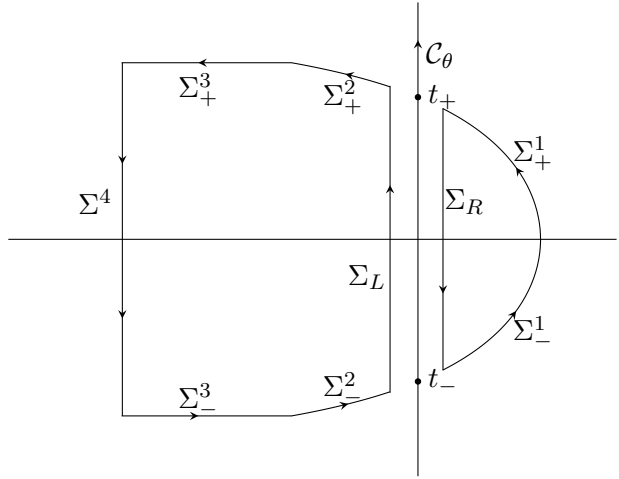

  Define
  \begin{align}
    I_1 = {}& \lim_{\delta \to 0} \int_{\mathcal{C}_{\theta}} \frac{ds}{2\pi i} \oint_{\Sigma_{\out}} \frac{dt}{2 \pi i} \frac{1}{s - t} e^{\frac{N}{M + 1}\big( f_{M,N}(s) - f_{M,N}(t)  \big)} e^{\frac{N}{M + 1} \frac{\xi t - \eta s}{\rho_{M, N}}}, \label{eq:defn_I_1} \\
    I_2 = {}& \lim_{\delta \to 0} \int_{\mathcal{C}_{\theta}} \frac{ds}{2\pi i} \oint_{\Sigma_{\inner}} \frac{dt}{2 \pi i} \frac{1}{s - t} e^{\frac{N}{M + 1}\big( f_{M,N}(s) - f_{M,N}(t)  \big)} e^{\frac{N}{M + 1} \frac{\xi t - \eta s}{\rho_{M, N}}}.
  \end{align}
Applying the residue theorem, we obtain
  \begin{equation} \label{asymp:sin-2}
    \begin{split}
      I_{2} = {}& -\int_{q_{M}(-\theta)}^{q_{M}(\theta)} \frac{ds}{2\pi i} e^{\frac{N}{M + 1} \frac{\xi - \eta}{\rho_{M, N}} s} \\
      = {}& - \rho_{M, N} \frac{M + 1}{N} e^{\frac{N}{M + 1} \frac{\xi - \eta}{\rho_{M, N}} \re q_{M}(\theta)} \frac{\sin \left( \frac{N}{M + 1} \frac{\xi - \eta}{\rho_{M, N}} \Im q_M(\theta) \right)}{\pi(\xi - \eta)} \\
      = {}& -\rho_{M, N} \frac{M + 1}{N} \exp\Big(\pi(\xi - \eta)\cot\theta + \bigO\big(M^{-1}\big)\Big) \frac{\sin(\pi(\xi - \eta))}{\pi(\xi - \eta)}.
    \end{split}
  \end{equation}
  Later we are going to prove the estimate
  \begin{equation} \label{asymp:sin-1}
    I_{1}  = \bigO((M/N)^{2/5}).
  \end{equation}

  Recalling \eqref{eq:form_rho}  and  noting  for $\theta\in (0,\pi)$  the approximation
   $\rho_{M, N} =  \frac{N}{\pi(M + 1)}(1 + \bigO(M^{-1}))$, we have  from \eqref{asymp:sin-2} and \eqref{asymp:sin-1}  that $I_2$ dominates $I_1$, and thus completes the proof of part \ref{enu:thm:sine}.

  The remaining part of the proof is on \eqref{asymp:sin-1} by steepest-descent analysis. We need to have some estimates of $\Re f_{M, N}(t)$ as $t \in \Sigma_{\out}$ and $\Re f_{M, N}(s)$ as $s \in \mathcal{C}_{\theta}$. First we have

  \begin{globalestA} \label{globalest_A}
  \begin{enumerate}[label=(A.\arabic*)]
  \item \label{enu:A_Sigma_3}
    When $t \in \Sigma^3_{\pm}$, we can use the estimate \eqref{eq:est_f'_MN}.
    By the assumption on $C$ in \eqref{eq:defn_Sigama_case3}, we have $\Re{\log t} + v_{M}(\theta) > 0$ for all $t \in \Sigma^3_{\pm}$, and also $\lvert 1 + t/(M + 1) \rvert < 1$ for all $t \in \Sigma^3_{\pm}$. Hence $\Re f'_{M, N}(t) < 0$ for all $t \in \Sigma^3_{\pm}$ and we conclude that over $\Sigma^3_{\pm}$, $\Re f'_{M, N}(t)$ attains its maximum at the two right end points $-C \pm ih_{M}^{-1}(-C)$.
  \item \label{enu:A_Sigma_4}
    When $t \in \Sigma^4$, we consider the imaginary part of $f'_{M, N}(t)$ by \eqref{eq:deriv_f_MN}. We denote $\sigma = N + \frac{N}{M + 1}t$, so that $\Re \sigma = \frac{1}{2}$ as $t \in \Sigma^4$.
    \begin{equation}
      \Im f'_{M, N}(t) = (M + 1) \Im\psi(\sigma) - \Im\psi(\sigma - N) = -M \left( \sum^{\infty}_{n = 0} \Im \frac{1}{\sigma + n} \right) + \sum^{N}_{n = 1} \Im \frac{1}{\sigma - 1},
    \end{equation}
    where we use \eqref{eq:digamma_series}. Now it is straightforward to see that as $t \in \Sigma^4$, $\Im f'_{M, N}(t) > 0$ if $\Im t > 0$ and $\Im f'_{M, N}(t) < 0$ if $\Im t < 0$. Hence we have that the maximum of $\Re f_{M, N}(t)$ over $\Sigma^4$ is attained at $t = -(M + 1) + \frac{M + 1}{2N}$. Furthermore, by direct computation, we have $\Re f_{M, N}(-(M + 1) + \frac{M + 1}{2N}) < \Re f_{M, N}(q_M(\pm\theta)) - \epsilon$ for some $\epsilon > 0$.
  \item \label{enu:A_Sigma_12}
    When $t \in \Sigma^1_{\pm} \cup \Sigma^2_{\pm}$, we assume that $\delta = 0$ in the definition \eqref{eq:defn_Sigama_case3}. We use the approximation \eqref{eq:est_f_MN_by_fm}, and can uniformly approximate $f_{M, N}(t)$ by $f_{M}(t)$. Then by part \ref{enu:saddle-t:1} of Lemma \ref{saddle-t}, we have that $f_{M}(t)$ attains its unique minimum at $\Sigma^1_+ \cup \Sigma^2_+$ at $q_{M}(\theta)$, and attains its unique minimum at $\Sigma^1_- \cup \Sigma^2_-$ at $q_{M}(-\theta)$. We note that the two minimal values are the same.
  \item \label{enu:A_C}
    We consider the infinite long contour $\mathcal{C}_{\theta}$ in two parts, one is the finite part $\mathcal{C}^1_{\theta} = \{ s \in \mathcal{C}_{\theta} \mid \lvert \Im s \rvert \leq K \}$ and $\mathcal{C}^2_{\theta} = \mathcal{C}_{\theta} \setminus \mathcal{C}^1_{\theta}$, where $K$ is a large positive constant with $K > h_{M}^{-1}(-C)$. For $s \in \mathcal{C}^1_{\theta}$, by \eqref{eq:est_f_MN_by_fm} $f_{M, N}(s)$ can be approximated by $f_{M}(s)$, and then by  part \ref{enu:saddle-t:2} of Lemma \ref{saddle-t}, $f_{M}(s)$ attains its maximal value at two points $q_{M}(\theta)$ and $q_{M}(-\theta)$. For $s \in \mathcal{C}^2_{\theta}$, we find that the approximation \eqref{eq:est_f'_MN} holds, and it is straightforward to check that if $K$ is large enough, $\Im f'_{M, N}(s) > 1$ if $s \in \mathcal{C}^2_{\theta} \cap \compC_+$ and $\Im f'_{M, N}(s) < -1$ if $s \in \mathcal{C}^2_{\theta} \cap \compC_-$. Hence $f_{M, N}(s)$ decreases at least linearly fast as $s$ moves to $\pm i\infty$ along $\mathcal{C}^2_{\theta}$.
  \end{enumerate}
  \end{globalestA}

  Next we estimate $\Re f_{M, N}(t)$ and $\Re f_{M, N}(s)$ ``locally'' around $q_{M}(\pm\theta)$. To this end, we divide $\mathcal{C}_{\theta}$ into $\mathcal{C}_{\loc, +} \cup \mathcal{C}_{\loc, -} \cup \mathcal{C}_{\glob}$, and divide $\Sigma_{\out}$ into $\Sigma_{\loc, +} \cup \Sigma_{\loc, -} \cup \Sigma_{\glob}$, such that
  \begin{equation} \label{eq:local_contours_sine}
    \begin{aligned}
      \mathcal{C}_{\loc, \pm} = {}& \mathcal{C}_{\theta} \cap B\big(q_{M}(\pm\theta), ((M+1)/N)^{2/5}\big), & \mathcal{C}_{\glob} = {}& \mathcal{C}_{\theta} \setminus (\mathcal{C}_{\loc, +} \cup \mathcal{C}_{\loc, -}), \\
      \Sigma_{\loc, \pm} = {}& \Sigma_{\out} \cap B\big(q_{M}(\pm\theta), ((M+1)/N)^{2/5}\big), & \Sigma_{\glob} = {}& \Sigma_{\out} \setminus (\Sigma_{\loc, +} \cup \Sigma_{\loc, -}).
    \end{aligned}
  \end{equation}
  We also define $\Sigma^0_{\loc, \pm} = \Sigma_{\pm} \cap B\big(q_{M}(\pm\theta), ((M+1)/N)^{2/5}\big)$.

  Around $q_{M}(\pm\theta)$, $f_{M, N}(t)$ can be approximated by $f_{M}(t)$ as in \eqref{eq:est_f_MN_by_fm}. From \eqref{eq:est_f'_MN} we check by direct computation 
  \begin{equation} \label{eq:derivative_f_M}
    f'_M(q_M(\pm\theta)) = 0.
  \end{equation}
  To have some intuition of the identity before going into the calculation, we note that as $M \to \infty$, this identity converges to $f'_{\infty}(q_{\infty}(\pm\theta)) = 0$, which is much easier to verify. So in $B\big(q_{M}(\pm\theta), ((M+1)/N)^{2/5}\big)$,
  \begin{equation} \label{eq:f_MN_est}
    f_{M, N}(t) = f_{M, N}(q_M(\pm\theta)) + \frac{f''_M(q_M(\pm\theta))}{2}\big(t - q_M(\pm\theta)\big)^2 + \bigO\big((M/N)^{\frac{6}{5}}\big).
  \end{equation}
  Then we have the following lemma:
  \begin{lem} \label{lem:est_local}
    There exists $\epsilon > 0$, such that $s \in \mathcal{C}_{\loc, \pm}$ and $t \in \Sigma^0_{\loc, \pm}$, for all $M$ large enough,
    \begin{align}
      \Re\big(f_M(s) - f_M(q_M(\pm\theta))\big) \leq {}& -\epsilon \lvert s - q_M(\pm\theta) \rvert^2, \label{eq:f_M_along_C} \\
      \Re\big(f_M(t) - f_M(q_M(\pm\theta))\big) \geq {}& \epsilon \lvert t - q_M(\pm\theta) \rvert^2, \label{eq:f_M_along_Sigma}
    \end{align}
    where all $\pm$ are the same.
  \end{lem}
  \begin{proof}
    We prove the lemma when all the $\pm$ are $+$. To see \eqref{eq:f_M_along_C}, we note that by \eqref{eq:q_M_q_infty}
    \begin{equation} \label{eq:approx_f''_M(q_M)}
      f''_M(q_M(\theta)) = \left( 1 + \frac{q_M(\theta)}{M + 1} \right)^{-1} - \frac{1}{q_M(\theta)} = 1 - \frac{\sin\theta}{\theta} e^{-i\theta} + \bigO(M^{-1}).
    \end{equation}
    So if $M$ is large enough, we have that $\Re f''_M(q_M(\theta)) > 0$, and also that in $B\big(q_{M}(+\theta), ((M+1)/N)^{2/5}\big)$, if $s$ moves along the vertical line $\mathcal{C}_{\loc, +}$, $\Re f_M(s)$ decreases like a quadratic function. Similarly, since the curve $\Sigma^0_{\loc, +}$ defined by \eqref{eq:defn_Sigma_pm} has the tangent direction
    \begin{equation}
      \arg q'_M(\theta) = \arg q'_{\infty}(\theta) + \bigO(M^{-1}) = \arg \left( \frac{-\theta}{\sin^2 \theta} + \frac{e^{i\theta}}{\sin \theta} \right) + \bigO(M^{-1}).
    \end{equation}
    By the calculation
    \begin{equation}
      \Re \Big( \Big( \frac{-\theta}{\sin^2 \theta} + \frac{e^{i\theta}}{\sin \theta} \Big)^2 (1 - \frac{\sin\theta}{\theta} e^{-i\theta}) \Big) = \frac{1}{\sin^2\theta} \Big( \Big( \frac{\theta}{\sin\theta} - \cos\theta \Big)^2 + \sin^2\theta \Big) > 0,
    \end{equation}
    we have that in $B\big(q_{M}(+\theta), ((M+1)/N)^{2/5}\big)$, if $t$ moves along the vertical line $\Sigma^0_{\loc, +}$, $\Re f_M(s)$ increases like a quadratic function.
  \end{proof}

  Hence we conclude that (all the $\pm$ being the same)
  \begin{multline} \label{eq:local_est_sine}
    \begin{aligned}
      & e^{-\frac{N}{M + 1} \frac{\xi  - \eta }{\rho_{M,N} }q_{M}(\pm\theta)}    \lim_{\delta \to 0} \int_{\mathcal{C}_{\loc, \pm}} \frac{ds}{2\pi i} \oint_{\Sigma_{\loc, \pm}} \frac{dt}{2 \pi i} \frac{1}{s - t} e^{\frac{N}{M + 1}\big( f_{M,N}(s) - f_{M,N}(t)  \big)} e^{\frac{N}{M + 1} \frac{\xi t - \eta s}{\rho_{M,N}}}  \\
      = {}& \PV \oint_{\Sigma_{\loc, \pm}} \frac{dt}{2 \pi i} \int_{\mathcal{C}_{\loc, \pm}} \frac{ds}{2\pi i} \frac{1}{s - t} e^{\frac{N}{M + 1}\big( f_{M,N}(s) - f_{M,N}(t)  \big)} e^{\frac{N}{M + 1} \frac{\xi (t - q_M(\pm\theta)) - \eta (s - q_M(\pm\theta))}{\rho_{M,N}}}  \\
      = {}& \PV \oint_{\Sigma^0_{\loc, \pm}} \frac{dt}{2 \pi i} \int_{\mathcal{C}_{\loc, \pm}} \frac{ds}{2\pi i} \frac{1}{s - t} \frac{e^{\frac{N}{M + 1} \big( \frac{f_{M}''(q_{M}(\pm\theta))}{2} (s - q_M(\pm\theta))^2 - \eta(s - q_M(\pm\theta)) \big)}}{e^{\frac{N}{M + 1} \big( \frac{f_{M}''(q_{M}(\pm\theta))}{2} (t - q_M(\pm\theta))^2 - \xi(t - q_M(\pm\theta)) \big)}}
    \end{aligned} \\
    \times \left(1  + \bigO((M/N)^{\frac{1}{5}}) \right),
  \end{multline}
  where the $(1  + \bigO((M/N)^{1/5}))$ term is uniform with respect to $s, t$. We give an estimate to the integral in \eqref{eq:local_est_sine} when the sign $\pm$ is $+$. By \eqref{eq:f_MN_est} and Lemma \ref{lem:est_local}, we have, with $w = (s - q_M(\theta)) \sqrt{N/(M + 1)}$ and $z = (t - q_M(\theta)) \sqrt{N/(M + 1)}$,
  \begin{equation}
    \begin{split}
      & \int_{\mathcal{C}_{\loc, +}} \frac{ds}{2\pi i} \frac{1}{s - t} e^{\frac{N}{M + 1} \big( \frac{f_{M}''(q_{M}(\theta))}{2} (s - q_M(\theta))^2 - \eta(s - q_M(\theta)) \big)} \Big(1  + \bigO((M/N)^{\frac{1}{5}})\Big) \\
      = {}& \int^{(N/(M + 1))^{1/10} i}_{-(N/(M + 1))^{1/10} i} \frac{dw}{2\pi i} \frac{1}{w - z} e^{\frac{f_{M}''(q_{M}(\theta))}{2} w^2 - \sqrt{\frac{N}{M + 1}} \eta w} \Big(1  + \bigO((M/N)^{\frac{1}{5}})\Big) \\
      = {}& \bigO(1),
    \end{split}
  \end{equation}
  and the $\bigO(1)$ estimate is uniform in $t \in \Sigma^0_{\loc, +} \setminus \{ q_M(\theta) \}$ even if $t \to q_M(\theta)$. Then the outer $\PV$ integral can be estimated simply and we have that the right-hand side of \eqref{eq:local_est_sine} is
  \begin{equation} \label{eq:final_est_loc}
    \PV \oint_{\Sigma_{\loc, +}} \frac{dt}{2 \pi i} \frac{\bigO(1)}{e^{\frac{N}{M + 1} \big( \frac{f_{M}''(q_{M}(\theta))}{2} (t - q_M(\theta))^2 - \xi(t - q_M(\theta)) \big)}} = \bigO((M/N)^{2/5}).
  \end{equation}
  Here we only use the estimate that the exponential term in the denominator is bounded by $1$ in absolute value. Actually the inequality \eqref{eq:f_M_along_Sigma} can improve the right-hand side of \eqref{eq:final_est_loc} to $\bigO((M/N)^{1/2})$. Interested readers may consult \cite[Equations (2.28)--(2.37)]{Liu-Wang-Zhang14}. We conclude that at the two ``local'' parts of the contour, namely $(s, t) \in \mathcal{C}_{\loc, +} \times \Sigma^0_{\loc, +}$ or $\mathcal{C}_{\loc, -} \times \Sigma^0_{\loc, -}$, the double contour integral that defines $I_1$ in \eqref{eq:defn_I_1} is estimated as $\bigO((M/N)^{2/5})$.

  On the other hand, combining Lemmas \ref{saddle-t} and \ref{lem:est_local}, the ``global'' and ``local'' estimates of $f_{M, N}(t)$ on $\Sigma_{\mathrm{out}}$ and $\mathcal{C}_{\theta}$, we know that there exists $\epsilon > 0$ such that  for $t \in \Sigma_{\glob}$ and $s \in \mathcal{C}_{\glob}$
  \begin{align}
    \Re f_{M, N}(s) + \epsilon \left( \frac{M+1}{N} \right)^{\frac{4}{5}} < \Re f_{M, N}(q_{M}(\pm\theta)) < \Re f_{M, N}(t) -  \epsilon\left( \frac{M+1}{N} \right)^{\frac{4}{5}}
  \end{align}
  and $\Re f_{M, N}(s)$ goes to $-\infty$ fast as $\im{s} \to \pm \infty$ along $\mathcal{C}_{\glob}$. Hence we have the estimate that
  \begin{multline} \label{global_est_sine}
    \iint_{\mathcal{C}_{\theta} \times \Sigma_{\out} \setminus (\mathcal{C}_{\loc, +} \times \Sigma_{\loc, +} \cup \mathcal{C}_{\loc, -} \times \Sigma_{\loc, -})} \frac{ds}{2\pi i} \frac{dt}{2 \pi i} \frac{1}{s - t} e^{\frac{N}{M + 1}\big( f_{M,N}(s) - f_{M,N}(t)  \big)} e^{\frac{N}{M + 1} \frac{\xi t - \eta s}{\rho_{M,N}}}  \\
   = \bigO(e^{-\epsilon (N/(M+1))^{1/5}}).
  \end{multline}
  Here the key point is that the factor $1/(s - t)$ is bounded below by $((M + 1)/N)^{2/5}$ in absolute value. We remark that similar computation appears in \cite[Equations (2.38)--(2.40)]{Liu-Wang-Zhang14}.

  By \eqref{eq:local_est_sine} and \eqref{global_est_sine}, we prove \eqref{asymp:sin-1}, and thus  finish the proof.

\subsubsection{Proof of parts \ref{enu:thm:airy} and \ref{enu:thm:airy_trace} of Theorem \ref{thm1.3}}

Before we give the proof of parts \ref{enu:thm:airy} and \ref{enu:thm:airy_trace} of Theorem \ref{thm1.3}, we establish some results to be used. Throughout this subsubsection, $f_M(t), f_{M, N}(t), \rho_{M, N}$ stand for $f_M(t; 0), f_{M, N}(t; 0), \rho_{M, N}(0)$ respectively.

Recalling that  $ q_{M}(0) = (M + 1)/M$ by \eqref{eq:defn_q_M_h_M}, define, analogous to \eqref{eq:local_contours_sine},
\begin{equation}
  \begin{split}
    \Sigma_{\loc, +} = {}& \{ q_{M}(0) - (2(M + 1)/N)^{1/3} + re^{2\pi i/3} \mid 0 \leq r \leq (2(M + 1)/N)^{3/10} \}, \\
    \Sigma_{\loc, -} = {}& \{ q_{M}(0) - (2(M + 1)/N)^{1/3} + re^{\pi i/3} \mid -(2(M + 1)/N)^{3/10} \leq r \leq 0 \}, \\
    \mathcal{C}_{\loc, +} = {}& \{ q_{M}(0) + (2(M + 1)/N)^{1/3} + re^{4\pi i/3} \mid -(2(M + 1)/N)^{3/10} \leq r \leq 0 \}, \\
    \mathcal{C}_{\loc, -} = {}& \{ q_{M}(0) + (2(M + 1)/N)^{1/3} + re^{5\pi i/3} \mid 0 \leq r \leq (2(M + 1)/N)^{3/10} \}.
  \end{split}
\end{equation}
Then we denote $\Sigma_{\loc} = \Sigma_{\loc, +} \cup \Sigma_{\loc, -}$ and $\mathcal{C}_{\loc} =\mathcal{C}_{\loc, +} \cup \mathcal{C}_{\loc, -}$. Next in the contour integral formula for $\K_N(x, y)$, deform the contour for $s$ as $\mathcal{C}_0 = \mathcal{C}_{\loc} \cup \mathcal{C}_{\glob}$, such that
\begin{multline}
  \mathcal{C}_{\glob} = \{ q_M(0) + (2(M + 1)/N)^{1/3} - (2(M + 1)/N)^{3/10} e^{4\pi i/3} + iy \mid y \geq 0 \} \\
  \cup \{ q_{M}(0) + (2(M + 1)/N)^{1/3} + (2(M + 1)/N)^{3/10} e^{5\pi i/3} + iy \mid y \leq 0 \},
\end{multline}
and deform the contour $\Sigma$ for $t$ as $\Sigma_{\loc} \cup \Sigma^1_+ \cup \Sigma^2_+ \cup \Sigma^3_+ \cup \Sigma^4 \cup \Sigma^3_- \cup \Sigma^2_- \cup \Sigma^1_-$. To define these contours, let $\varphi$ be the unique solution on $(0, \pi)$ such that
\begin{equation}
  h_M(\varphi) = q_M(0) - (2(M + 1)/N)^{1/3} - \frac{1}{2} (2(M + 1)/N)^{3/10},
\end{equation}
where $h_M$ is defined by \eqref{eq:defn_q_M_h_M}, and let $C$ be a large enough positive constant independent of $M, N$ such that $\log C + v_M(0) > 0$ for all large enough $M$. Then
\begin{equation}
  \begin{split}
    \Sigma^1_{\pm} = {}& \text{vertical line connecting the left end of $\Sigma_{\loc, \pm}$ to $q_M(\pm\varphi)$}, \\
    \Sigma^2_{\pm} = {}& \Big\{ t \in \Sigma_{\pm}: \pm \arg t > \varphi \text{ and } \Re t \geq -C \Big\}, \\
    \Sigma^3_{\pm} = {}& \Big\{ t = x \pm i \im\{q_{M}(h_{M}^{-1}(-C))\}: x \in [-(M + 1) + \frac{M + 1}{2N}, -C] \Big\}, \\
    \Sigma^4 = {}& \Big\{t = -(M + 1) + \frac{M + 1}{2N} + iy: y \in [-\im q_{M}(h_{M}^{-1}(-C)), \im q_{M}(h_{M}^{-1}(-C))]\Big\}.
  \end{split}
\end{equation}
The orientation of the contours defined above are determined by
\begin{enumerate*}
\item
  $\Sigma$ is positively oriented, and
\item
  $\mathcal{C}_0$ is from $-i\infty$ to $+i\infty$.
\end{enumerate*}; see Figure \ref{intcon-airy}.
\begin{figure}[h]
  \centering
  \begin{tikzpicture}
    \draw (-6,0) -- (2,0);
    \draw[domain=-pi+0.8:-3*pi/5+0.01,->,>=stealth] plot({\x*(cos(\x r))*(sin(\x r))^(-1)},{\x});
    \draw (-1,-2.1) node[above]{$\Sigma_{-}^{2}$};
    \draw[->,>=stealth] (-4.51, {-pi+.8}) -- (-3.5, {-pi+.8});
    \draw (-3.52, {-pi+.8}) -- ({-2.27}, {-pi+0.8});
    \draw (-3.5,-pi+0.65) node[above]{$\Sigma_{-}^{3}$};
    \draw (0.6, {-.5}) -- (0.9, 0);
    \draw (0.8,0.3) node[left]{$\Sigma_{\mathrm{local}}$};
    \draw (0.9, 0)-- (0.6, {0.5});
    \draw[domain=1.05:3*pi/5,->,>=stealth] plot({\x*(cos(\x r))*(sin(\x r))^(-1)},{\x});
    \draw[domain=-3*pi/5:-1.05] plot({\x*(cos(\x r))*(sin(\x r))^(-1)},{\x});
    \draw[domain= 3*pi/5-0.01:pi-.8] plot({\x*(cos(\x r))*(sin(\x r))^(-1)},{\x});
    \draw (-1,2) node[below]{$\Sigma_{+}^{2}$};
    \draw[->,>=stealth] ({-2.26}, {pi-.8}) -- (-3.5, {pi-.8});
    \draw (-3.45, {pi-.8}) -- (-4.51, {pi-.8});
    \draw (-3.5,pi-.8) node[below]{$\Sigma_{+}^{3}$};
    \draw[->,>=stealth] (-4.5, {pi-.8}) -- (-4.5, {1});
    \draw[->,>=stealth] (-4.5, {1.05}) -- (-4.5, {-1.05});
    \draw (-4.5, {-1}) -- (-4.5, {-pi+.8});
    \draw (-4.5,0.5) node[left]{$\Sigma^{4}$};

    \draw (0.6, -0.49) -- (0.6, {-1.06});
    \draw (0.7,0.8) node[left]{$\Sigma_{+}^{1}$};
    \draw ({0.6}, {0.49}) -- (0.6, 1.06);
    \draw (0.7,-0.8) node[left]{$\Sigma_{-}^{1}$};
    \draw (1.4, {-2.2}) -- (1.4, {-.49});
    \draw (1.4, {-.5}) -- (1.1, 0);
    \draw (1.2,0.3) node[right]{$\mathcal{C}_{\mathrm{local}}$};
    \draw (1.1, 0)-- (1.4, {0.5});
    \draw[->,>=stealth] (1.4, {0.49}) -- (1.4, {1.5});
    \draw (1.4, 1.49)-- (1.4, {2.2});
    \draw (1.3,1.3) node[right]{$\mathcal{C}_{\mathrm{global}}$};
    \filldraw (1, 0) circle (1pt);
    \draw (1, -0.2) node[below] {$\frac{M+1}{M}$};
  \end{tikzpicture}
  \caption{Schematic contours in the proof of part \ref{enu:thm:airy} of Theorem \ref{thm1.3}}\label{intcon-airy}
\end{figure}
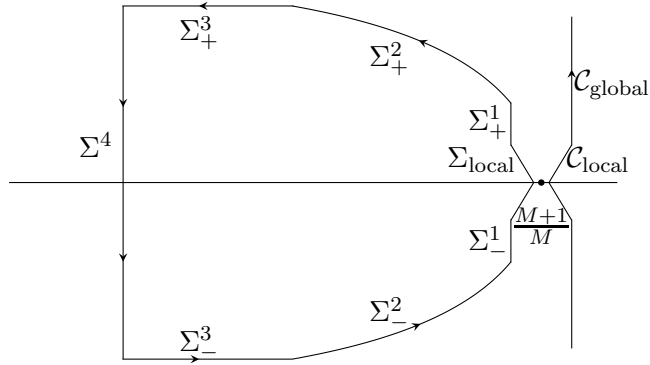

Note that at $q_M(0)$, we have similar to \eqref{eq:derivative_f_M} that
\begin{equation}
  f'_M(q_M(0)) = 0.
\end{equation}
Furthermore, since $f''_M(t) = (1 + \frac{t}{M + 1})^{-1} - \frac{1}{t}$ and $f'''_M(t) = t^{-2} - (M + 1)(M + 1 + t)^{-2}$, we have 

\begin{equation}
  f''_M(q_M(0)) = 0, \quad f'''_M(q_M(0)) = 1 + \bigO(M^{-1}).
\end{equation}
Hence in $B\big(q_M(0), ((M+1)/N)^{3/10}\big)$, we have
\begin{equation} \label{eq:approx_f_MN_local}
  f_{M, N}(t) = f_{M, N}(q_M(0)) + \frac{1}{6} f'''_M(q_M(0))(t - q_M(0))^3 + \bigO((M/N)^{6/5}\big).
\end{equation}

On the other hand, we have the ``global'' estimate of $\Re f_{M, N}(t)$ for $t \in \Sigma \setminus \Sigma_{\loc}$ and of $\Re f_{M, N}(s)$ for $s \in \mathcal{C}_{\glob}$ that is parallel to that in Global Estimates \hyperref[globalest_A]{A}: 
\begin{globalestB}
  \begin{enumerate}[label=(B.\arabic*)]
  \item \label{enu:Sigma_3_B}
    When $t \in \Sigma^3_{\pm}$, by argument the same as that in \ref{enu:A_Sigma_3} with only $\theta$ replaced by $0$, $\Re f_{M, N}(t)$ increases as $t$ moves leftward, and attains its minimum at the two right end points $-C \pm ih_{M}^{-1}(-C)$.
  \item \label{enu:Sigma_4_B}
    When $t \in \Sigma^4$, by argument the same as that in \ref{enu:A_Sigma_4}, we have that the maximum of $\Re f_{M, N}(t)$ over $\Sigma^4$ is attained at $t = -(M + 1) + \frac{M + 1}{2N}$. Furthermore, by direct computation, we have $\Re f_{M, N}(-(M + 1) + \frac{M + 1}{2N}) < \Re f_{M, N}(q_M(\pm\theta)) - \epsilon$ for some $\epsilon > 0$.
  \item \label{enu:Sigma_2_B}
    When $t \in \Sigma^2_{\pm}$, by part \ref{enu:saddle-t:1} of Lemma \ref{saddle-t}, $\Re f_M(t)$ increases as $t$ moves leftward. Then by the approximation \eqref{eq:est_f_MN_by_fm}, $\Re f_{M, N}(t)$ has the same property. The argument is the same as in \ref{enu:A_Sigma_12}.
  \item \label{enu:Sigma_1_B}
    When $t \in \Sigma^1_{\pm}$, by part \ref{enu:saddle-t:2} of Lemma \ref{saddle-t}, $\Re f_M(t)$ increases as $t$ moves upward on $\Sigma^1_+$, or $t$ moves downward on $\Sigma^1_-$, and so does $\Re f_{M, N}(t)$ due to the approximation \eqref{eq:est_f_MN_by_fm}.
  \item
    By argument the same as that in \ref{enu:A_C}, we consider the infinite long contour $\mathcal{C}_{\glob}$ in two parts, one is $\mathcal{C}^1_{\glob} = \{ s \in \mathcal{C}_{\glob} \mid \lvert \Im s \rvert \leq K \}$ and  $\mathcal{C}^2_{\theta} = \mathcal{C}_{\theta} \setminus \mathcal{C}^1_{\theta}$, where $K$ is a large positive constant. Then by part \ref{enu:saddle-t:3} of Lemma \ref{saddle-t}, $\Re f_M(s)$ decreases as $s$ moves upward on $\mathcal{C}_{\glob} \cap \compC_+$ or $s$ moves downward on $\mathcal{C}_{\glob} \cap \compC_-$. We can also show by direct computation that $\Im f'_{M, N}(t) > 1$ if $s \in \mathcal{C}^2_{\glob} \cap \compC_+$ and $\Im f'_{M, N}(s) < -1$ if $s \in \mathcal{C}^2_{\glob} \cap \compC_-$. Hence $f_{M, N}(s)$ decreases at least linearly fast as $s$ moves to $\pm i\infty$ along $\mathcal{C}^2_{\glob}$.
  \end{enumerate}
\end{globalestB}

\begin{proof}[Proof of part \ref{enu:thm:airy} of Theorem \ref{thm1.3}]
  Upon the change of variables $t = q_M(0) + (2(M + 1)/N)^{1/3} \tau$ and $s = q_M(0) + (2(M + 1)/N)^{1/3} \sigma$ and noting the value of $\rho_{M, N} = \rho_{M, N}(0)$ given in \eqref{eq:form_rho}, (cf.~\cite[Equations (2.63)--(2.65)]{Liu-Wang-Zhang14})
  \begin{equation} \label{eq:lical_airy_est}
    \begin{split}
      & \left( \frac{N}{2(M + 1)} \right)^{1/3} e^{-(\xi - \eta)\frac{N}{M + 1} \frac{q_M(0)}{\rho_{M, N}}} \int_{\mathcal{C}_{\loc}} \frac{ds}{2\pi i} \int_{\Sigma_{\loc}} \frac{dt}{2 \pi i} \frac{1}{s - t}  e^{\frac{N}{M + 1}(f_{M, N}(s) - f_{M, N}(t))} e^{\frac{N}{M + 1} \frac{\xi t - \eta s}{\rho_{M, N}}} \\
      = {}& \left( \frac{N}{2(M + 1)} \right)^{1/3} \int_{\mathcal{C}_{\loc}} \frac{ds}{2\pi i} \int_{\Sigma_{\loc}} \frac{dt}{2 \pi i} \frac{1}{s - t} e^{\frac{N}{M + 1}[(s - q_M(0))^3 - (t - q_M(0))^3] + \bigO \left( (M/N)^{1/5} \right)} \\
      & \phantom{\smash{\left( \frac{N}{2(M + 1)} \right)^{1/3} \int_{\mathcal{C}_{\loc}} \frac{ds}{2\pi i} \int_{\Sigma_{\loc}} \frac{dt}{2 \pi i} \frac{1}{s - t}}} \times e^{\frac{N}{M + 1} \frac{\xi (t - q_M(0)) - \eta (s - q_M(0))}{\rho_{M, N}}} \\
      = {}& \int_{\mathcal{C}^{(N/(2M + 2))^{1/30}}_<} \frac{d\sigma}{2\pi i} \int_{\Sigma^{(N/(2M + 2))^{1/30}}_>} \frac{d\tau}{2\pi i} \frac{1}{\sigma - \tau} \frac{e^{\frac{\sigma^3}{3} - \eta\sigma}}{e^{\frac{\tau^3}{3} - \xi\tau}}
    \left( 1+ \bigO \left( (M/N)^{1/5}\right) \right) \\
      = {}& K_{\Airy}(\xi, \eta) + \bigO \left( (M/N)^{1/5} \right).
    \end{split}
  \end{equation}
Here,  similar to   \eqref{eq:defn_Airy_contours},
for real $R>0$  the contours are defined to be  upwards and parametrized as
\begin{equation} \label{eq:defn_Airy_contours-12}
  \begin{split}
    \Sigma^R_> = {}& \{ -1 + re^{2\pi i/3} \mid 0 \leq r \leq R \} \cup \{ -1 + re^{\pi i/3} \mid -R \leq r \leq 0 \}, \\
    \mathcal{C}^R_< = {}& \{ 1 + re^{4\pi i/3} \mid -R \leq r \leq 0 \} \cup \{ 1 + re^{5\pi i/3} \mid 0 \leq r \leq R \}.
  \end{split}
\end{equation}

  On the other hand, by Global Estimates B, if $\xi, \eta$ are in a compact subset of $\realR$, $\rho_{M, N} = \rho_{M, N}(0)$ defined in \eqref{eq:form_rho}, there exist $\epsilon > 0$ such that for $t \in \Sigma \setminus \Sigma_{\loc}$ and $s \in \mathcal{C}_{\glob}$

  \begin{multline}
    \Re \left[ f_{M, N}(s; 0) - \frac{\eta (s - q_M(0))}{\rho_{M, N}} \right] + \epsilon \left( \frac{M+1}{N} \right)^{9/10} \\
    \begin{aligned}
      < {}& \Re f_{M, N}(q_{M}(\pm\theta)) \\
      < {}& \Re \left[ f_{M, N}(t) - \frac{\xi (t - q_M(0))}{\rho_{M, N}} \right] -  \epsilon\left( \frac{M+1}{N} \right)^{9/10},
    \end{aligned}
  \end{multline}
  and $\Re f_{M, N}(s)$ goes to $-\infty$ fast as $\im{s} \to \pm \infty$ along $\mathcal{C}_{\glob}$. Hence combining the estimates above and the behaviour of $f_{M, N}(t; 0)$ on $\Sigma_{\loc}$ and $\mathcal{C}_{\loc}$, we have the estimate (cf.~\cite[Equations (2.66)--(2.69)]{Liu-Wang-Zhang14})
  \begin{equation} \label{global_est_airy}
    \begin{split}
      & \left( \frac{N}{2(M + 1)} \right)^{1/3} e^{-(\xi - \eta)\frac{N}{M + 1} \frac{q_M(0)}{\rho_{M, N}}} \iint_{\mathcal{C}_0 \times \Sigma \setminus \mathcal{C}_{\loc} \times \Sigma_{\loc}} \frac{ds}{2\pi i} \frac{dt}{2 \pi i} \frac{1}{s - t} e^{\frac{N}{M + 1}\big( f_{M,N}(s) - f_{M,N}(t)  \big)} \\
      & \phantom{\smash{\left( \frac{N}{2(M + 1)} \right)^{1/3} e^{-(\xi - \eta)\frac{N}{M + 1} \frac{q_M(0)}{\rho_{M, N}}} \iint_{\mathcal{C}_0 \times \Sigma \setminus \mathcal{C}_{\loc} \times \Sigma_{\loc}} \frac{ds}{2\pi i} \frac{dt}{2 \pi i} \frac{1}{s - t}}} \times e^{\frac{N}{M + 1} \frac{\xi t - \eta s}{\rho_{M, N}}} \\
      = {}& \left( \frac{N}{2(M + 1)} \right)^{1/3} \iint_{\mathcal{C}_0 \times \Sigma \setminus \mathcal{C}_{\loc} \times \Sigma_{\loc}} \frac{ds}{2\pi i} \frac{dt}{2 \pi i} \frac{1}{s - t} \frac{e^{\frac{N}{M + 1} \left( f_{M,N}(s) - \frac{\eta (s - q_M(0))}{\rho_{M, N}} \right)}}{e^{\frac{N}{M + 1} \left( f_{M,N}(t) - \frac{\xi (t - q_M(0))}{\rho_{M, N}} \right)}} \\
   = {}& \bigO(e^{-\epsilon (N/(M+1))^{1/10}}).
    \end{split}
  \end{equation}
  Combining \eqref{eq:lical_airy_est} and \eqref{global_est_airy}, we have
  \begin{equation}
    \begin{split}
      & \frac{1}{\rho_{M, N}} e^{-(\xi - \eta)\frac{N}{M + 1} \frac{q_M(0)}{\rho_{M, N}}} \K_N(g(\xi), g(\eta)) \\
      = {}& \left( \frac{N}{2(M + 1)} \right)^{1/3} e^{-(\xi - \eta)\frac{N}{M + 1} \frac{q_M(0)}{\rho_{M, N}}} \iint_{\mathcal{C}_0 \times \Sigma} \frac{ds}{2\pi i} \frac{dt}{2 \pi i} \frac{1}{s - t} e^{\frac{N}{M + 1}\big( f_{M,N}(s) - f_{M,N}(t)  \big)} e^{\frac{N}{M + 1} \frac{\xi t - \eta s}{\rho_{M, N}}} \\
      = {}& K_{\Airy}(\xi, \eta) + \bigO\big((M/N)^{1/5}\big).
    \end{split}
  \end{equation}
  Part \ref{enu:thm:airy} of Theorem \ref{thm1.3} is thus proved, by noting that $q_M(0) = (M+1)/M$.
\end{proof}

The proof of part \ref{enu:thm:airy_trace} of Theorem \ref{thm1.3} relies on the following technical result that is analogous to Lemma \ref{lem:first_trace_norm}:
\begin{lem} \label{lem:second_trace_norm}
  Let $C \in \realR$, $\mathbf{\K}_N$ and $\mathbf{K}_{\Airy}$ be the integral operators on $L^2((C, +\infty))$ whose kernel are $e^{-\frac{N}{M + 1}  \frac{\xi -\eta}{\rho_{M,N}}} \frac{1}{\rho_{M,N}} \K_{N}\big(g(\xi), g(\eta)\big)$  defined in Theorem \ref{thm1.3} and $K_{\Airy}(\xi, \eta)$ respectively. Then $\mathbf{\K}_N$ and $\mathbf{K}_{\Airy}$ are trace class operators and as $N \to \infty$, $\mathbf{\K}_N \to \mathbf{K}_{\Airy}$ in trace norm.
\end{lem}

\begin{proof}[Proof of Lemma \ref{lem:second_trace_norm}]
  The strategy of the proof is the same as that for part \ref{enu:thm:crit_2} of Theorem \ref{thm:crit}. We note that the kernels $e^{-N(x - y)/((M + 1) \rho_{M,N})} (1/\rho_{M,N}) \K_{N}\big(g(x), g(y)\big)$ and $K_{\Airy}(x, y)$ can be written as
  \begin{equation} \label{eq:convolutions_Airy}
    \int^{\infty}_C G_N(x, r)H_N(r, y) dr \quad \text{and} \quad \int^{\infty}_C G_{\infty, \gamma}(x, r) H_{\infty, \gamma}(r, y) dr
  \end{equation}
  respectively, where
   \begin{align}
    G_N(x, r) = {} \oint_{\Sigma} g_N(t; x, r) \frac{dt}{2\pi i}, \ g_N(t; x, r) = e^{\frac{N}{M + 1} \big( f_{M, N}(q_{M}(0)) -f_{M, N}(t) + \frac{(x + r - C)(t - q_M(0))}{\rho_{M, N}} \big)}, \label{eq:G_N_Airy} \end{align}
       \begin{align}
    G_{\infty, \gamma}(x, r) = {}& \oint_{\Sigma^{\infty}_>} g_{\infty}(\tau; x, r) \frac{dt}{2\pi i}, & g_{\infty}(\tau; x, r) = {}& e^{-\frac{\tau^3}{3} + (x + r - C)\tau}, \label{eq:G_infty_Airy} \\
    H_N(r, y) = {}& \int_{\mathcal{C}_0} h_N(s; r, y) \frac{ds}{2\pi i}, & h_N(s; r, y) = {}& e^{\frac{N}{M + 1} \big( f_{M, N}(s)- f_{M, N}(q_{M}(0)) - \frac{(y + r - C)(s - q_M(0))}{\rho_{M, N}} \big)}, \\
    H_{\infty, \gamma}(r, y) = {}& \int_{\mathcal{C}^{\infty}_<} h_{\infty}(\sigma; r, y) \frac{ds}{2\pi i}, & h_{\infty}(\sigma; r, y) = {}& e^{\frac{\sigma^3}{3} - (y + r - C)\sigma}.
  \end{align}
  We only need to prove the convergence
  \begin{align}
    \lim_{N \to \infty} \iint_{(C, \infty)^2} dx dr \left\lvert G_N(x, r) - G_{\infty, \gamma}(x, r) \right\rvert^2 = {}& 0, \label{eq:HS_norm_G_Airy} \\
    \lim_{N \to \infty} \iint_{(C, \infty)^2} dr dy \left\lvert H_N(r, y) - H_{\infty, \gamma}(r, y) \right\rvert^2 = {}& 0. \label{eq:HS_norm_H_Airy}
  \end{align}

  To prove \eqref{eq:HS_norm_G_Airy}, we use the same method as for \eqref{eq:HS_norm_G}: We divide the integral domain $[C, \infty)^2$ into two regions: $I_1(R) = \{ (x, r) \mid x \geq C,\ r \geq C, \text{ and } x + r \leq R \}$ and $I_2(R) = \{ (x, r) \mid x \geq C,\ r \geq C, \text{ and } x + r > R \}$, where $R > C$ is a constant.

  On the region $I_1(R)$, by the estimate \eqref{eq:approx_f_MN_local} of $f_{M, N}(t, w)$ on $\Sigma_{\loc}$, we have that the integrand in \eqref{eq:G_N_Airy} and the integrand in \eqref{eq:G_infty_Airy} satisfy the relation
  \begin{equation}
     g_N(t; x, r) = g_{\infty}(\tau; x, r) + \bigO((M/N)^{1/5}) \quad \text{where} \quad t = q_M(0) + (2(M + 1)/N)^{1/3} \tau,
  \end{equation}
  where the $\bigO \left( (M/N)^{1/5} \right)$ estimate is uniform in $t \in \Sigma_{\loc}$ and in $(x, r) \in I_1(R)$. Hence with the contours defined in \eqref{eq:defn_Airy_contours-12}
  \begin{equation}
    \int_{\Sigma_{\loc}} g_N(t; x, r) \frac{dt}{2\pi i} = \int_{\Sigma^{(N/(2M + 2))^{1/30}}_>} g_{\infty}(\tau; x, r) \frac{dt}{2\pi i} + \bigO ((M/N)^{1/5}),
  \end{equation}
   uniform as $(x, r) \in I_1(R)$.

   On the other hand, by the estimate \ref{enu:Sigma_3_B}, \ref{enu:Sigma_4_B} \ref{enu:Sigma_2_B} and \ref{enu:Sigma_1_B} for $g_N(t; x, r)$ on the remaining contour $\Sigma_{\glob}$, and some straightforward estimate of $g_{\infty}(\tau; x, r)$ on $\Sigma^{\infty}_> \setminus \Sigma^{(N/(2M + 2))^{1/30}}_>$, we have
  \begin{equation}
    \int_{\Sigma_{\glob}} g_N(t; x, r) \frac{dt}{2\pi i} = \bigO(e^{-\epsilon \left( \frac{N}{M+1} \right)^{1/10}}), \int_{\Sigma^{\infty}_> \setminus \Sigma^{(N/(2M + 2))^{1/30}}_>} g_{\infty}(\tau; x, r) \frac{d\tau}{2\pi i} = \bigO(e^{-\epsilon \left( \frac{N}{M+1} \right)^{1/10}}),
  \end{equation}
  for some $\epsilon > 0$ uniformly for $(x, r) \in I_1(R)$. Hence we prove that on $I_1(R)$, $G_N(x, r) - G_{\infty, \gamma}(x, r)$ approaches $0$ uniformly.

  For $(x, y) \in I_2(R)$, we use the estimate of $g_N(t; x, r)$ as shown above, to conclude that
  \begin{equation}
    \Re f_{M, N}(t) > 0 \quad \text{and} \quad \lvert g_N(t; x, r) \rvert \leq e^{\frac{N}{M + 1} \frac{(x + r - C)(\Re (t - q_M(0)))}{\rho_{M, N}}}
  \end{equation}
  as $t \in \Sigma$. Hence we conclude that $\lvert G_N(x, r) \rvert < \text{constant} \cdot e^{-(x + r - C)/2}$ for $(x, r) \in I_2$. By direct computation, we also find that $\lvert G_{\infty}(x, r) \rvert < \text{constant} \cdot e^{-(x + r - C)/2}$.

  Using the arbitrariness of $R$, we prove \eqref{eq:HS_norm_G_Airy} by the dominance convergence theorem.

  The convergence \eqref{eq:HS_norm_H_Airy} can be proved in the same way: Dividing the integral domain $(C, \infty)^2$ into $I_1(R)$ and $I_2(R)$, depending on $r + y \leq R$ or $> R$, for any $R > C$, then show that $\lvert H_N(r, y) - H_{\infty, \gamma}(r, y) \rvert$ converges uniformly to $0$ in $I_1(R)$, and is bounded by $\text{constant} \cdot e^{-(r - C + y)/2}$ in $I_2(R)$. Then the convergence is proved. The details are omitted.
\end{proof}

\subsubsection{Proof of Lemma \ref{saddle-t}} \label{subsubsec:proof_lemma_saddle_t}
The proof of this lemma is just the same as that of \cite[Lemmas 3.1 and 3.2]{Liu-Wang-Zhang14}, but to make  the proof of Theorem \ref{thm1.3}  self-contained we give it here.

Recalling that $v_{M}(\theta)$ is defined as in \eqref{parameter-vm}, a direct computation shows that
\begin{align}
  \frac{d}{d \phi} f_{M}(q_{M}(\phi); \theta)
  =(v_{M}(\phi) - v_{M}(\theta))\frac{d}{d \phi} q_{M}(\phi).
\end{align}
One can show that both $v_{M}(\phi)$ and $\Re{q_{M}(\phi)}$ are decreasing functions in $\phi \in [0, \pi)$, then part \ref{enu:saddle-t:1} follows.

To prove the rest statements we need the following two computations
\begin{equation}
  \begin{split}
    \frac{d}{d y} \Re{f_{M}(x + iy; \theta)} = {}& - \left.\Im{\frac{d}{d z}f_{M}(z; \theta)}\right|_{z = x + iy}\\
    = {}& - (M + 1) \arctan \frac{y}{M+1+x} + \arctan\frac{y}{x}
  \end{split}
\end{equation}
and
\begin{equation}
  \begin{split}
    \frac{d^{2}}{d y^{2}} \Re{f_{M}(x + iy; \theta)} = {}& - \left.\Re{\frac{d^{2}}{d z^{2}}f_{M}(z; \theta)}\right|_{z = x + iy}\\
    = {}& - \frac{(M+1+x)x(M x - (M+1))+ y^{2}((M+1)(M+1+x) - x)}{((M+1+x)^{2} + y^{2}) (x^{2} + y^{2})}.
  \end{split}
\end{equation}
When $\theta \in (0, \pi)$, $\Re{q_{M}(\theta)} < 1 + M^{-1}$ there exists a unique $y_{0} > 0$ such that
\begin{equation}
  \left.\frac{d^{2}}{d y^{2}} \Re{f_{M}(\Re{q_{M}(\theta)} + iy; \theta)}\right|_{y = y_{0}} = 0.
\end{equation}
It is easy to see that $\frac{d}{d y} \Re{f_{M}(\Re{q_{M}(\theta)} + iy; \theta)}$ is increasing for $y \in (0, y_{0})$ and decreasing for $y > y_{0}$. Combining the fact that
\begin{equation} \label{enu-saddle-temp0}
  \left.\frac{d}{d y} \Re{f_{M}(x + iy; \theta)}\right|_{y=0} = \left.\frac{d}{d y} \Re{f_{M}(x + iy; \theta)}\right|_{y=\Im{q_{M}(\theta)}} = 0,
\end{equation}
the first two inequalities in the part \ref{enu:saddle-t:2} follow. By the symmetry of $\Re{f_{M}(x + iy; \theta)}$ the rest two inequalities also hold.

Since when $x > 1 + M^{-1}$, $\frac{d^{2}}{d y^{2}} \Re{f_{M}(x + iy; 0)} <0$ for all $y \in \mathbb{R}$, then part \ref{enu:saddle-t:3} can be deduced from \eqref{enu-saddle-temp0}.

\section{Further discussion} \label{sect:discuss}

In this last section we discuss a few relevant questions and add some comments.

\subsection{Critical kernel revisited}

The critical correlation kernel $K_{\crit}(x, y; \gamma)$ defined as in \eqref{softlimitK} admits an ``integrable"   form which may be more convenient for application. To this end, we introduce two families of functions
\begin{equation}
  f_{-1}(x)=  \oint_{\Sigma_{-\infty}} \frac{dt}{2 \pi i}  \Gamma(t) e^{ -\frac{\gamma}{2} t^{2} + x t}, \quad g_{-1}(x)= \int^{1 + i\infty}_{1 - i\infty} \frac{ds}{2\pi i} \frac{1}{\Gamma(s)} e^{ \frac{\gamma}{2} s^{2} -x s}
\end{equation}
and  for $ k=0,1,\ldots,$
\begin{align}
  f_{k}(x)=  \oint_{\Sigma_{-\infty}} \frac{dt}{2 \pi i}  \frac{\Gamma(t)}{t +k} e^{ -\frac{\gamma}{2} t^{2} +x t}, \quad g_{k}(x)= \int^{1 + i\infty}_{1 - i\infty}\frac{ds}{2\pi i} \frac{1}{(s +k)\Gamma(s)} e^{ \frac{\gamma}{2} s^{2} -x s},
\end{align}
where $\Sigma_{-\infty}$ is defined as in \eqref{softlimitK}. By the identity
\begin{equation}
  \frac{1}{s-t}=\int_{0}^{\infty} e^{-(s-t)u}   du, \quad \Re(s - t) > 0,
\end{equation}
and noting that we can assume $\Re(s - t) > 0$ for all $s, t$ in \eqref{softlimitK}, we have
\begin{align}
  K_{\crit}(x, y;\gamma)= \int_{0}^{\infty}  f_{-1}(u+x)  g_{-1}(u+y) du.
\end{align}
By the identity
 \begin{equation}
   (x-y)e^{xt-ys}= \Big(\frac{\partial}{\partial t }+\frac{\partial}{\partial s}\Big)e^{xt-ys},
 \end{equation}
we apply the integration by parts over infinite contours $\Sigma_{-\infty}$ for $t$ and from $1 - i\infty$ to $1 + i\infty$ for $s$, such that the integrand functions all vanish as either $s$ or $t$ moves to $\infty$ along their contours, to derive
\begin{equation}
  \begin{split}
    (y - x) K_{\crit}(x, y; \gamma) = {}& -\int^{1 + i\infty}_{1 - i\infty} \frac{ds}{2\pi i} \oint_{\Sigma_{-\infty}} \frac{dt}{2\pi i} \frac{1}{s - t} \frac{\Gamma(t)}{\Gamma(s)} \frac{e^{\frac{\gamma s^2}{2}}}{e^{\frac{\gamma t^2}{2}}} \Big(\frac{\partial}{\partial t }+\frac{\partial}{\partial s}\Big)e^{xt-ys} \\
    = {}& \int^{1 + i\infty}_{1 - i\infty} \frac{ds}{2\pi i} \oint_{\Sigma_{-\infty}} \frac{dt}{2\pi i} \bigg[ \frac{\partial}{\partial t} \Big( \frac{1}{s - t} \Gamma(t) e^{-\frac{\gamma t^2}{2}} \Big) \frac{1}{\Gamma(s) e^{-\frac{\gamma s^2}{2}}}
    \\
    & \phantom{ \int^{1 + i\infty}_{1 - i\infty} \frac{ds}{2\pi i} \oint_{\Sigma_{-\infty}} \frac{dt}{2\pi i}}  + \Gamma(t) e^{-\frac{\gamma t^2}{2}} \frac{\partial}{\partial s} \Big( \frac{1}{s - t} \frac{1}{\Gamma(s) e^{-\frac{\gamma s^2}{2}}} \Big) \bigg] e^{xt - ys}
     \\
    = {}& \gamma \int^{1 + i\infty}_{1 - i\infty} \frac{ds}{2\pi i} \oint_{\Sigma_{-\infty}} \frac{dt}{2\pi i} \frac{\Gamma(t)}{\Gamma(s)} \frac{e^{\frac{\gamma s^2}{2} - ys}}{e^{\frac{\gamma t^2}{2} - xt}} \\
    & +  \int^{1 + i\infty}_{1 - i\infty} \frac{ds}{2\pi i} \oint_{\Sigma_{-\infty}} \frac{dt}{2\pi i} \frac{\psi(s) - \psi(t)}{s - t} \frac{\Gamma(t)}{\Gamma(s)} \frac{e^{\frac{\gamma s^2}{2} - ys}}{e^{\frac{\gamma t^2}{2} - xt}}.
  \end{split}
\end{equation}
Then by the series expansion \eqref{eq:digamma_series} of the digamma function $\psi(z)$, we have
\begin{equation}
  \frac{\psi(s) - \psi(t)}{s - t} = \sum^{\infty}_{n = 0} \frac{-1}{s - t} \Big( \frac{1}{n + s} - \frac{1}{n + t} \Big) = \sum^{\infty}_{n = 0} \frac{1}{(n + s)(n + t)},
\end{equation}
from which
\begin{equation}
  K_{\crit}(x, y; \gamma) = \frac{1}{y - x} \Big( \gamma f_{-1}(x) g_{-1}(y) + \sum^{\infty}_{n = 0} f_n(x) f_n(y) \Big).
\end{equation}


\subsection{Criticality  in the bulk} \label{Sectbulkcrit}

Besides the critical scaling limit at the soft edge, Akemann, Burda and Kieburg also investigated the local bulk statistics in the critical regime and obtained a new interpolating kernel in \cite{Akemann-Burda-Kieburg18}. Inspired by their result \cite[Equation (13)]{Akemann-Burda-Kieburg18} (Actually, more details  appear in a   forthcoming  paper \cite{Akemann-Burda-Kieburg18b} by the same authors), we consider the bulk critical limit and give a different derivation.

To state the main theorem in this subsection, we need the Jacobi theta function, defined as
\begin{equation}
  \vartheta(z;\tau) = \sum^{\infty}_{n=-\infty}  e^{\pi i  n^2 \tau+ 2\pi i n z}, \quad \Im{\tau}>0,\  z\in \compC.
\end{equation}
Let $[x]$ be the greatest integer less than or equal to $x$.

\begin{thm} [Bulk criticality in case \ref{enu:case_2}] \label{thm:bulkcrit}
  Suppose that  $\lim_{N\to\infty} M/N =\gamma \in (0, \infty)$. For any given $u\in (0,1)$, let $\gamma'=\gamma/(1-u)$ and
  \begin{equation}
    g(\xi) = M\log N(1-u)  +\log\frac{1-u}{u} +\frac{M+1}{N(1-u)} \left(Nu-[Nu] - \frac{1}{2} \right) - \xi.
  \end{equation}
  With the kernel  $\K_N$ in \eqref{transformK},  we have uniformly for $\xi, \eta$ in a compact subset of $\realR$
  \begin{equation}
    \lim_{N \to \infty}  e^{(g(\xi) -g(\eta) )[Nu]}  \K_N \big(g(\xi) , g(\eta) \big) = K^{(\bulk)}_{\crit}(\xi, \eta; \gamma'),
  \end{equation}
  where
  \begin{equation} \label{bulkcrit}
    K^{(\bulk)}_{\crit}(\xi, \eta; \gamma') = \frac{1}{\sqrt{8 \pi \gamma'}}\int^{1}_{-1} dw \,  e^{\frac{1}{2\gamma'}(\pi w-i\eta)^2}  \vartheta\left(  \frac{1}{2\pi}(\pi w-i\xi);  \frac{i }{2\pi}\gamma' \right).
  \end{equation}
\end{thm}

\begin{proof}
  Use the  Euler's reflection formula  for the gamma function and the identity
  \begin{equation}
    \frac{\sin \pi s}{\sin \pi t}= \frac{\sin \pi (s-t)}{\sin \pi t} e^{i\pi t}+e^{-i\pi (s-t)},
  \end{equation}
  we rewrite $\K_N$ in \eqref{transformK}  as
  \begin{equation}
    \begin{split}
      \K_{N}(x, y) = {}& \int^{c + i\infty}_{c - i\infty} \frac{ds}{2\pi i} \oint_{\Sigma} \frac{dt}{2 \pi i} \frac{ e^{xt-ys} }{s-t} \frac{\Gamma(1-s)}{\Gamma(1-t)}  \left( \frac{\Gamma(s+N)}{\Gamma(t+N)}\right)^{M+1}  \left( \frac{\sin \pi (s-t)}{\sin \pi t} e^{i\pi t}+\frac{e^{i\pi t}}{e^{i\pi s}}  \right) \\
      = {}& \int^{c + i\infty}_{c - i\infty} \frac{ds}{2\pi i} \oint_{\Sigma} \frac{dt}{2 \pi i} \frac{  \sin \pi (s-t)}{s-t} \frac{e^{i\pi t}}{\sin \pi t}  \frac{e^{-ys}  \Gamma(1-s)}{e^{-xt} \Gamma(1-t)}  \left( \frac{\Gamma(s+N)}{\Gamma(t+N)}\right)^{M+1},
    \end{split}
  \end{equation}
  because the poles with respect to $t$ within $\Sigma$ only occur as zeros of $\sin \pi t$.

  Make change of variables $s \to s-[Nu]$, $t \to t-[Nu]$ and deform the contours accordingly, we obtain
  \begin{equation} \label{bulkeqn1}
    \K_{N}(g(\xi), g(\eta)) = e^{(g(\eta) - g(\xi)) [Nu]} \int^{ i\infty}_{-i\infty} \frac{ds}{2\pi i} \oint_{\Sigma_{\Box}} \frac{dt}{2 \pi i} \frac{ \sin \pi (s-t)}{s-t}
    \frac{e^{i\pi t}}{\sin \pi t}  \frac{e^{f_{N}(\eta,s)}}{e^{f_{N}(\xi,t)}},
  \end{equation}
  where
  \begin{equation}
    f_{N}(\xi,t)= - t g(\xi) +\log \frac{\Gamma(1+[Nu]-t)}{\Gamma(1+[Nu])} + (M+1)\log \frac{\Gamma(N-[Nu]+t)}{\Gamma(N-[Nu])},
  \end{equation}
  and $\Sigma_{\Box}$ is a rectangular contour with four vertexes $\frac{1}{2}+[Nu]\pm \frac{i}{2}$ and $-N+[Nu]+\frac{1}{2}\pm \frac{i}{2}$.

  For $t \in B(0, N^{1/4})$ and $\xi$ in a compact subset of $\realR$, using estimate \eqref{stirling} and \eqref{digammaa}, we have uniformly
  \begin{equation} \label{bulkeqn}
    f_{N}(\xi,t) = \frac{1}{2}(\gamma' +o(1)) t^2 -\xi t +\bigO(N^{-1/4}).
  \end{equation}
  So as $s \in   i\mathbb{R} $ and $t \in \Sigma_{\Box}$ and $s, t \in B(0, N^{1/4})$, noting that $f_{N}(\xi,0)$ is a constant independent of $\xi$ that
  \begin{equation}
    \frac{e^{f_{N}(\eta,s)}}{e^{f_{N}(\xi,t)}} = \frac{e^{\frac{1}{2}(\gamma' +o(1))s^2 - \eta s}}{e^{\frac{1}{2}(\gamma' +o(1))t^2 - \xi t}} (1 + \bigO(N^{-1/4}).
  \end{equation}

  Next, we estimate the integrand when either $s$ or $t$ is not in $B(0, N^{1/4})$. For $s = iy$, like \eqref{eq:estimate_F'},
  \begin{equation} \label{eq:1st_dev_bulk}
    \begin{split}
      \frac{d}{dy} \Re f_{N}(\eta,iy) = {}& \Im \psi(1+[Nu]- iy) -\Im \big\{(M+1)\psi(N-[Nu]+iy)\big\} \\
      = {}& - \Big(\arctan \frac{y}{1+[Nu]}+(M+1)\arctan\frac{y}{N-[Nu]}\Big) \Big(1 + \bigO(N^{-1})\Big).
    \end{split}
  \end{equation}
  Thus, there exists $\epsilon > 0$ such that
  \begin{equation}
    \Re \{f_{N}(\eta,iy) \}\leq  - \epsilon N^{1/4} (|y| - N^{1/4}), \quad  |y| \geq N^{1/4}.
  \end{equation}
  Then it dominates the factor $\sin \pi (s-t)$ as $s \to \infty$ on the vertical contour.

  On the other hand,  as $t$ moves to the right endpoint along $\{x\pm \frac{i}{2}: x\in [-N+[Nu]+\frac{1}{2}, -N^{1/4}]\}$ or to the left endpoint along $\{x\pm \frac{i}{2}: x\in [N^{1/4},[Nu]+\frac{1}{2}]\}$,  $\Re f_N(\eta,t)$ decreases monotonically. To see it, we check that on these horizontal contours, the second derivative
  \begin{equation}
    \begin{split}
      & \frac{d^2}{dx^2} \Re f_{N}(\xi, x \pm \frac{i}{2}) \\
      = {}& \Re \big\{ \psi'(1+[Nu]-(x \pm \frac{i}{2}))+(M+1)  \psi'(N-[Nu]+(x \pm \frac{i}{2}))\big\} \\
      = {}& \sum^{\infty}_{n = 0}\bigg( \frac{(n +1+[Nu]-x)^2 - \frac{1}{4}}{((n + 1+[Nu] -x)^2 + \frac{1}{4})^2} + (M+1)\frac{(n +N-[Nu]+x)^2 - \frac{1}{4}}{((n +N-[Nu] +x)^2 + \frac{1}{4})^2} \bigg)>0,
    \end{split}
  \end{equation}
  and the first derivative of $\Re f_N(\xi, x \pm \frac{i}{2})$ satisfies, by arguments as in \eqref{eq:1st_dev_bulk},
  \begin{equation}
    \left. \frac{d}{dx} \Re f_{N}\big(\xi, x \pm \frac{i}{2}\big) \right\rvert_{x = N^{1/4}} > 0,  \quad  \left. \frac{d}{dx} \Re f_{N}\big(\xi, x \pm \frac{i}{2}\big) \right\rvert_{x = -N^{1/4}} < 0.
  \end{equation}
  So $\Re f_{N}(\xi, x \pm \frac{i}{2})$ increases monotonically on these horizontal contours to the endpoints $\pm N^{1/4}\pm \frac{i}{2}$.

  At last, on the two vertical lines of $\Sigma_{\Box}$, applying the Stirling formula leads to  for any $-1/2\leq y\leq 1/2$
  \begin{equation}
    \Re\big\{ f_{N}(\xi, -N+[Nu]+\frac{1}{2}+iy)\big\} = (M+1)N(1-u)+ \bigO \big(N\log N\big),
  \end{equation}
  and
  \begin{equation}
    \Re\big\{ f_{N}(\xi, [Nu]+\frac{1}{2}+iy)\big\} = (M+1)N(-u-\log(1-u))+ \bigO \big(N\log N\big).
  \end{equation}

  Combine these estimates and we know that the double integral \eqref{bulkeqn1} concentrates on the region of $s, t \in B(0, N^{1/4})$, that is, like \eqref{eq:limit_kernel_crit},
  \begin{multline} \label{eq:intermediate_bulk}
    e^{(g(\xi) -g(\eta) )[Nu]}  \
      \K_{N}(g(\xi), g(\eta)) = \left(1 + \bigO(N^{-1/4})\right) \times \\
 \left( \int^{ iN^{1/4}}_{-iN^{1/4}} \frac{ds}{2\pi i}
      \int^{ -N^{1/4}+\frac{i}{2}}_{N^{1/4}+\frac{i}{2}} \frac{dt}{2\pi i}
      +\int^{ iN^{1/4}}_{-iN^{1/4}} \frac{ds}{2\pi i}
      \int^{ N^{1/4}-\frac{i}{2}}_{ -N^{1/4}-\frac{i}{2}} \frac{dt}{2\pi i} \right)  \frac{ \sin \pi (s-t)}{s-t}
    \frac{e^{i\pi t}}{\sin \pi t}  \frac{e^{\frac{\gamma'}{2} s^2 - \eta s}}{e^{\frac{\gamma'}{2} t^2 - \xi t}}.
  \end{multline}
  This further gives us
  \begin{multline} \label{bulkint2line}
    \lim_{N \to \infty}  e^{(g(\xi) -g(\eta) )[Nu]}  \K_N \big(g(\xi) , g(\eta) \big) = \\
    \left( \int^{ i\infty}_{-i\infty} \frac{ds}{2\pi i}
      \int^{ \frac{i}{2}-\infty}_{\frac{i}{2}+\infty}\frac{dt}{2\pi i} +\int^{ i\infty}_{-i\infty} \frac{ds}{2\pi i}
      \int^{ -\frac{i}{2}+\infty}_{-\frac{i}{2}-\infty}\frac{dt}{2\pi i}    \right)  \frac{ \sin \pi (s-t)}{s-t}
    \frac{e^{i\pi t}}{\sin \pi t}  \frac{e^{\frac{\gamma'}{2} s^2 - \eta s}}{e^{\frac{\gamma'}{2} t^2 - \xi t}}.
  \end{multline}

  Use first the simple fact that
  \begin{equation}
    \frac{\sin \pi(s-t)}{s-t}=\frac{\pi}{2}\int_{-1}^{1} dw\, e^{-i\pi (s-t)w},
  \end{equation}
  and then integrate out variable $s$, we simplify the right-hand side of \eqref{bulkint2line} into
  \begin{equation}   \label{thetakernelint}
    \sqrt{\frac{\pi}{8 \gamma'}} \int^{1}_{-1} dw\,  e^{\frac{(\pi w-i\eta)^2}{2\gamma'} }
    \left[ \left(  \int^{ -\frac{i}{2}+\infty}_{-\frac{i}{2}-\infty}\frac{dt}{2\pi i}
      -\int^{ \frac{i}{2}+\infty}_{\frac{i}{2}-\infty}
      \frac{dt}{2\pi i}    \right)
    \frac{e^{i\pi t}}{\sin \pi t}  e^{-\frac{\gamma'}{2} t^2 + (\xi + i\pi w) t} \right].
  \end{equation}
  Then using the residue theorem to evaluate the inner integral with respect to $t$ in \eqref{thetakernelint}, where the poles are $z \in \intZ$, we simplify it into $\frac{1}{\pi} \vartheta\left(  \frac{1}{2\pi}(\pi w-i\xi);  \frac{i }{2\pi}\gamma' \right)$.  We thus complete the proof.
\end{proof}

\begin{rem}
  Our expression form for the critical limit in the bulk \eqref{bulkcrit} is actually the same as that in \cite[Equation (13)]{Akemann-Burda-Kieburg18}, just by noting that the summation in \cite[Equation (13)]{Akemann-Burda-Kieburg18} can simplify to an integral in terms of the Jacobi theta function. Besides, as $\gamma' \to 0$ one can recover the sine kernel; see \cite[Equation (16)]{Akemann-Burda-Kieburg18}.
\end{rem}

\subsection{Transition from critical kernels}

According to  the meaning of the parameter $\gamma$  as a limit of the ratio  $M/N$ and the  main results displayed in Theorems  \ref{cor:normality}  and \ref{thm1.3},  we expect  to observe  the Tracy-Widom phenomenon (Airy kernel) as $\gamma\to 0$ and the Gaussian phenomenon as  $\gamma \to  \infty$. Indeed, we have
\begin{thm} \label{thm1.4}
The following hold  uniformly for $x, y$ in a compact subset of $\mathbb{R}$:
  \begin{equation} \label{tog}
    \lim_{\gamma \to \infty}  \sqrt{\gamma} K_{\mathrm{crit}}\left( \sqrt{\gamma}x, \sqrt{\gamma}y;\gamma \right)  = \frac{1}{\sqrt{2\pi}}e^{-\frac{1}{2}y^2}
     \end{equation}
  and with $k = k(\gamma) = 2^{-\frac{1}{3}} \gamma^{\frac{2}{3}}$ and  $t_0$ being the unique positive solution of $\psi'(t_0)=\gamma$,
   \begin{equation} \label{toa}
       \lim_{\gamma \to 0}  k e^{ k t_0(y - x)}
      K_{\crit}(\gamma t_0 - \psi(t_0) + kx, \gamma t_0 - \psi(t_0) + ky;\gamma)= K_{\mathrm{Airy}}(x,y).
  \end{equation}
  \end{thm}

\begin{proof}[Sketch of proof]
  As to \eqref{tog}, after change of variables  $s \mapsto s/\sqrt{\gamma}$ and $t \mapsto t/\sqrt{\gamma}$ the kernel \eqref{softlimitK} becomes
  \begin{equation} \label{tog2}
    \sqrt{\gamma} K_{\crit} \big(\sqrt{\gamma}x, \sqrt{\gamma}y;\gamma\big) = \int^{\frac{1}{2} + i\infty}_{\frac{1}{2} - i\infty} \frac{ds}{2\pi i} \oint_{\Sigma'_{-\infty} \cup \Sigma_0} \frac{dt}{2 \pi i} \frac{1}{s-t} \frac{\Gamma(\frac{t}{\sqrt{\gamma}})}{\Gamma(\frac{s}{\sqrt{\gamma}})} \frac{ e^{ \frac{1}{2} s^{2} - y s}}{e^{ \frac{1}{2} t^{2} - x t}}.
  \end{equation}
  Here $\Sigma'_{-\infty} $ denotes a counterclockwise contour consisting of two rays and one line segment:  $\{ x - i \mid -\infty < x \leq -\sqrt{\gamma}/2 \} \cup \{ x + i \mid -\infty < x \leq -\sqrt{\gamma}/2 \} \cup \{ -\sqrt{\gamma}/2 + iy \mid -1 \leq y \leq 1 \}$, while $\Sigma_0$ is a small circle around the origin. As $\gamma \to +\infty$, the contour $\Sigma'_{-\infty}$ shrinks to $-\infty$, and we see that the integral over $\Sigma_0$ with respect to $t$ dominates that over $\Sigma'_{-\infty}$. With $\Sigma'_{-\infty} \cup \Sigma_1$ replaced by $\Sigma_1$, we see that the double contour integral becomes
  \begin{equation} \label{eq:transit_to_gaussian}
    \sqrt{\gamma} \int^{\frac{1}{2} + i\infty}_{\frac{1}{2} - i\infty} \frac{ds}{2\pi i} \frac{e^{\frac{1}{2} s^2 - ys}}{s\Gamma(\frac{s}{\sqrt{\gamma}})} = \int^{\frac{1}{2} + i\infty}_{\frac{1}{2} - i\infty} \frac{ds}{2\pi i} \frac{e^{\frac{1}{2} s^2 - ys}}{\Gamma(\frac{s}{\sqrt{\gamma}} + 1)},
  \end{equation}
  and as $\gamma \to +\infty$ we have that the integral  on the right-hand side of  \eqref{eq:transit_to_gaussian} converges to $ e^{-\frac{1}{2} y^2}/\sqrt{2\pi}$, so we prove the convergence of $\sqrt{\gamma} K_{\mathrm{crit}}\left( \sqrt{\gamma}x, \sqrt{\gamma}y;\gamma \right)$.


  Next, we turn to \eqref{toa}. Change $s,t$ to $t_0 + k^{-1} s, t_0 + k^{-1} t$ and we rewrite \eqref{softlimitK} as
  \begin{equation} \label{toa2}
    k K_{\crit}(\gamma t_0 - \psi(t_0) + kx, \gamma t_0 - \psi(t_0) + ky; \gamma) = \int_{\mathcal{C}^{\infty}_<} \frac{ds}{2\pi i} \int_{\Sigma^{\infty}_> } \frac{dt}{2 \pi i} \frac{1}{s-t}  e^{ xt-ys} e^{f_{\gamma}(s)-f_{\gamma}(t)},
  \end{equation}
  where
  \begin{equation}
    f_{\gamma}(t) = \log\Gamma(t_0 + k^{-1} t) + \frac{\gamma t_0 - \psi(t_0) - \gamma t_0}{k}t - \frac{\gamma}{2k^2}t^2,
  \end{equation}
  and the contours $\Sigma^{\infty}_>$ and $\mathcal{C}^{\infty}_<$ are defined in \eqref{eq:defn_Airy_contours}. We note that $\mathcal{C}^{\infty}_<$ is a ``bent'' version of the vertical contour for $s$ in \eqref{softlimitK}, while $\Sigma^{\infty}_>$ ``opens'' $\Sigma_{-\infty}$ in \eqref{softlimitK}.

  By direct calculation we have $f'_{\gamma}(0) = f''_{\gamma}(0) = 0$ and, using \eqref{eq:behaviour_t_0} we have
  \begin{equation}
    f'''_{\gamma}(0) = k^{-3} \psi''(t_0) = k^{-3} \left(-t_{0}^{-2} + \bigO(t^{-3}_0) \right) = -2 + \bigO(\gamma).
  \end{equation}
  Furthermore, $f^{(4)}_{\gamma}(z)$ converges to $0$ pointwise as $\gamma \to 0$. So as $\gamma \to 0$, the integrand of the double contour integral on the right-hand side of \eqref{toa2} converges pointwise to
  \begin{equation}
    \frac{1}{s - t} \frac{e^{\frac{s^3}{3} - ys}}{e^{\frac{t^3}{3} - xt}},
  \end{equation}
  and if we plug this pointwise limit to the double contour integral formula, we have the Airy kernel on the right-hand side of \eqref{toa}. Hence the limit of $k e^{k t_0(y - x)}  K_{\crit}(\gamma t_0 - \psi(t_0) + kx, \gamma t_0 - \psi(t_0) + ky; \gamma)$ is formally proved.
  A complete proof requires estimates of $f_{\gamma}(t)$ on $\Sigma^{\infty}_>$  and $f_{\gamma}(s)$ on $\Sigma^{\infty}_<$ , and we omit the details.
\end{proof}

Similar to the transition from the Airy kernel to the sine kernel,  we also observe  a transition from   the critical  kernel  at the soft edge $K_{\mathrm{crit}}$ defined in \eqref{softlimitK} to the critical bulk kernel  $K^{(\mathrm{bulk})}_{\crit}$ in \eqref{bulkcrit}.
\begin{thm} \label{thm1.5}
  Given   a positive integer $k$, we have
  \begin{equation}
    \lim_{k \to \infty}  e^{k(x-y) }   K_{\crit}(-\gamma k-\log k+x, -\gamma k-\log k+y;\gamma) = K^{(\bulk)}_{\crit}(x,y; \gamma),
  \end{equation}
  uniformly for $x, y$ in a compact subset of $\realR$.
\end{thm}

\begin{proof}[Sketch of proof] Let $g(x)=-\gamma k-\log k+x$,
  after the change of variables $s \to s-k$, $t \to t-k$ we obtain
 \begin{equation} \label{eq:edge_to_bulk}
   e^{k(x-y)} K_{\crit}(g(x), g(y)) = \int_{-i\infty}^{ i\infty} \frac{ds}{2\pi i} \oint_{\Sigma_{-\infty,k}} \frac{dt}{2 \pi i} \frac{ \sin \pi (s-t)}{s-t} \frac{e^{i\pi t}}{\sin \pi t}  \frac{e^{f_{k}(y,s)}}{e^{f_{k}(x,t)}},
 \end{equation}
 where
 \begin{equation}
   f_{k}(x,t)=    \log \Gamma(1+k-t)+   \frac{1}{2}\gamma t^2 -(x-\log k)t,
 \end{equation}
 and $\Sigma_{-\infty, k}$ is a counterclockwise contour similar to $\Sigma'_{-\infty}$ \eqref{tog2}, consisting of two rays and one line segment:  $\{ x - i \mid -\infty < x \leq k + \frac{1}{2} \} \cup \{ x + i \mid -\infty < x \leq k + \frac{1}{2} \} \cup \{ k + \frac{1}{2} + iy \mid -1 \leq y \leq 1 \}$.

 In the open ball $B(0, k^{1/4})$, By the Stirling formula we have
 \begin{equation}
   f_{k}(x,t)-f_{k}(x,0) = \frac{1}{2}\gamma t^2 -x t +\bigO (k^{-1/4}),
 \end{equation}
 and by estimating $f_k(t)$ on $\Sigma_{-\infty, k}\! \setminus\! B(0, k^{1/4})$ and $f_k(s)$ on $\{ iy \mid y \in \realR \}\! \setminus\! B(0, k^{1/4})$, we can show that the double contour integral concentrates in the region $s, t \in B(0, k^{1/4})$ (The details of the estimates are omitted). Hence the right-hand side of \eqref{eq:edge_to_bulk} is approximated by the right-hand side of \eqref{eq:intermediate_bulk} with $\gamma'$ replaced by $\gamma$. Then the conclusion is proved by similar arguments \eqref{bulkint2line}-\eqref{thetakernelint} as  in the proof of Theorem \ref{thm:bulkcrit}.
\end{proof}

\subsection{With different sizes}

When each $X_j$ is a complex Ginibre matrix of size $(\nu_j +N) \times (\nu_{j-1} +N)$ with $\nu_0=0$ and $\nu_{1}, \ldots, \nu_M \geq 0$,  the  eigenvalues of $\log\big( \Pi_{M}^{*}\Pi_{M}\big)$ with the product $\Pi_{M}$ in \eqref{Mproduct} also form  a determinantal point process with correlation kernel
\begin{equation} \label{2integral-2}
  \K^{(\nu)}_{N}(x, y) = \int_{c-i\infty}^{c + i\infty} \frac{ds}{2\pi i} \oint_{\Sigma} \frac{dt}{2 \pi i} \frac{e^{xt-ys}  }{s-t} \frac{\Gamma(t)}{\Gamma(s)}  \prod_{j=0}^{M} \frac{\Gamma(s+\nu_j +N)}{\Gamma(t+\nu_j+ N)},
\end{equation}
see \cite{Kuijlaars-Zhang14}. In this general case, as $N$ fixed, all Lyapunov exponents can be expressed as certain time average (with some upper bound for $\nu_j$)
\begin{equation}
  \lambda^{(\nu)}_{k} =  \lim_{M\to \infty} \frac{1}{2(M+1)} \sum_{j=0}^{M}\psi\big(\nu_j+N-k+1\big), \quad k=1, \ldots, N, \label{ipsenlimits}
\end{equation}
whenever the limits exist; see \cite{Ipsen15}.

Our main  results in Section \ref{mainresults}  have straightforward generalizations to the case above with some $\nu_j > 0$ but under certain assumptions on $\nu_j $. Here  we just state and sketch the proof of the critical result  at the soft edge (case \ref{enu:case_2}) that is a generalization of part \ref{enu:thm:crit_1} of  Theorem \ref{thm:crit}. Note that it is unnecessary to assume  the existence of the limits  in \eqref{ipsenlimits} in order to state   the following result.

\begin{thm} \label{thm1.2+}
  Suppose that
  \begin{equation} \label{ratiocond}
    \lim_{N\to\infty} \sum_{j=0}^M \frac{1}{\nu_j+N} =\gamma \in  (0,\infty).
  \end{equation}
  Let
  \begin{equation}
    g^{(\nu)}(\xi) = \sum^M_{j = 0} \left( \log(\nu_j + N) -\frac{1}{2(\nu_j + N)} \right) + \xi,
  \end{equation}
  then uniformly for $\xi,\eta$ in a compact subset of $\realR$ we have
  \begin{equation} \label{critical+limit}
    \lim_{N \to \infty} \K^{(\nu)}_{N} \left( g^{(\nu)}(\xi),  g^{(\nu)}(\eta) \right) = K_{\crit}(\xi, \eta; \gamma),
  \end{equation}
  where $K_{\crit}(\xi, \eta;\gamma)$ is defined in \eqref{softlimitK}.
\end{thm}
\begin{proof}[Sketch of proof]
  Since the proof is a straightforward generalization to that of part \ref{enu:thm:crit_1} of Theorem \ref{thm:crit}, we give only the ``local'' computation but omit the ``global'' estimates. Without loss of generality, we assume that $
   \sum_{j=0}^M 1/(\nu_j+N)
    =\gamma$.

  Like \eqref{eq:alt_K_N_crit}, we write
  \begin{multline} \label{eq:general_crit}
    \K^{(\nu)}_N\Big( \sum^M_{j = 0} \Big( \log(\nu_j + N) + \frac{1}{\nu_j + N} w \Big), \sum^M_{j = 0} \Big( \log(\nu_j + N) + \frac{1}{\nu_j + N} w' \Big) \Big) = \\
    \int^{1 + i\infty}_{1 - i\infty} \frac{ds}{2\pi i} \oint_{\Sigma_-(1/2)} \frac{dt}{2\pi i} \frac{\exp \big[ (M + 1)F^{(\nu)}(t; w) \big]}{\exp \big[ (M + 1)F^{(\nu)}(s; w') \big]} \frac{\Gamma(t)}{\Gamma(s)} \frac{1}{s - t},
  \end{multline}
  where $\Sigma_-(1/2)$ is the same as in \eqref{eq:alt_K_N_crit}, and $F^{(\nu)}$ is a generalization of $F$ in \eqref{eq:defn_F}
  \begin{equation}
    F^{(\nu)}(t; w) =  \frac{1}{M + 1} \sum^M_{j = 0} \left( \Big(\log(\nu_j + N) + \frac{w}{\nu_j + N}\Big) t - \log \Gamma(t + \nu_j + N) + \log \Gamma(\nu_j + N) \right).
  \end{equation}
  For $\lvert t \rvert < N^{1/4}$, we see from  the Taylor expansion of the gamma function (cf. \eqref{stirling} and \eqref{digammaa}) that
  \begin{align}    \begin{split}
   &-\log \frac{\Gamma(t+\nu_j+N)}{\Gamma(\nu_j+N)} +\Big(\log(\nu_j+N)+\frac{w}{2(\nu_j+N)}\Big)t\\
   &  =  \Big( \log(\nu_j + N) - \psi(\nu_j + N) + \frac{w}{\nu_j + N} \Big)t - \frac{1}{2} \psi'(\nu_j + N) t^2
+ \bigO\Big(\frac{t^3}{(\nu_j+N)^2}\Big).
     \end{split}
  \end{align}
  Noting  the assumption \eqref{ratiocond},    analogous to \eqref{eq:est_MF_central_crit},
  we have
  \begin{multline}
    (M + 1)F^{(\nu)}(t; w) = \\
    \sum^M_{j = 0} \left[ \left( \log(\nu_j + N) - \psi(\nu_j + N) + \frac{w}{\nu_j + N} \right)t - \frac{1}{2} \psi'(\nu_j + N) t^2 \right] + \bigO(N^{-1/4}),
  \end{multline}
  and then analogous to \eqref{eq:est_coe_crit}, by \eqref{ratiocond},
  \begin{equation}
    \lim_{N \to \infty} \sum^M_{j = 0}\Big( \log(\nu_j + N) - \psi(\nu_j + N) + \frac{w}{\nu_j + N}\Big) = \gamma(w + \frac{1}{2}), \quad \lim_{N \to \infty} \sum^M_{j = 0} \psi'(\nu_j + N) = \gamma,
  \end{equation}
  where  the error  term $\bigO(N^{-1/4})$ is uniform for all   parameters $\nu_j$.

  Hence if we restrict the double contour integral in \eqref{eq:general_crit} in a region $s, t = \bigO(N^{1/4})$, we have
  \begin{multline} \label{eq:limit_kernel_crit_again}
    \int^{1 + iN^{1/4}}_{1 - iN^{1/4}} \frac{ds}{2\pi i} \oint_{\Sigma_-(1/2) \cap B(0, N^{1/4})} \frac{dt}{2\pi i} \frac{\exp \big[ (M + 1)F^{(\nu)}(t; w) \big]}{\exp \big[ (M + 1)F^{(\nu)}(s; w') \big]} \frac{\Gamma(t)}{\Gamma(s)} \frac{1}{s - t} = \\
     \left(1 + \bigO(N^{-1/4})\right) \int^{1 + iN^{1/4}}_{1 - iN^{1/4}} \frac{ds}{2\pi i} \oint_{\Sigma'_-(1/2) \cap B(0, N^{1/4})} \frac{dt}{2\pi i} \frac{e^{\frac{\gamma}{2} s^2 - \gamma(w' + 1/2)s}}{e^{\frac{\gamma}{2} t^2 - \gamma(w + 1/2)t}} \frac{\Gamma(t)}{\Gamma(s)}\frac{1}{s - t},
  \end{multline}
  where the right-hand side is formally the same as \eqref{eq:limit_kernel_crit}. Suppose that the double contour integral in \eqref{eq:general_crit} concentrates in the  region $s, t = \bigO(N^{1/4})$, we are virtually done with the proof, noting that after $w, w'$ are changed into $\gamma^{-1}\xi - 1/2$ and $\gamma^{-1}\eta - 1/2$ respectively and $N^{1/4}$ is extended to $+\infty$, the right-hand side of \eqref{eq:limit_kernel_crit_again} (or rather \eqref{eq:limit_kernel_crit}) is $\left(1 + \bigO(N^{-1/4})\right)K_{\crit}(\xi, \eta; \gamma)$.

  The result that the double contour integral in \eqref{eq:general_crit} concentrates in the region $s, t = \bigO(N^{1/4})$ can be proved by estimates of $F^{(\nu)}(t; w)$ like those  of $F(t; w)$ that we obtained in \eqref{eq:est_F_vert} and \eqref{eq:est_exp_F_crit}. We omit the details.
\end{proof}

\subsection{Open questions} \label{sec:open_ques}

As discussed in the Introduction, the product of $M$ random matrices of size $N\times N$ relates classical law of large numbers and central limit theorems, and Lyapunov exponents when $M\to \infty$ and $N$ is fixed, to RMT statistics when $N\to \infty$ and $M$ is fixed.  As both $M$ and $N$ go to infinity such that $M/N \to \gamma \in (0,\infty)$, there is a phase transition phenomenon as observed in Theorem \ref{thm:crit} and \cite{Akemann-Burda-Kieburg18}. These draw us to conclude this last section with a few questions which are worth considering.

\begin{que} \label{que:1}
  Prove that $F_{\crit}(x; \gamma)$ defined by \eqref{distri} is a distribution function and find an explicit Painlev\'{e}-type expression for it, like that for the Tracy-Widom distribution; cf. \cite{Tracy-Widom94}.
\end{que}

\begin{que}
  Consider the product of real Gaussian random matrices and prove a phase transition from GOE statistics to Gaussian. Furthermore, find an explicit interpolating process associated with the largest Lyapunov exponent.
\end{que}

\begin{que}
  Prove the phase transition phenomenon for the product of truncated unitary/orthogonal matrices; see \cite{Forrester15} and \cite{Kieburg-Kuijlaars-Stivigny15}.
\end{que}

\begin{que}
  Verify Theorems \ref{cor:normality}--\ref{thm1.3} for singular values of products of non-Hermitian random matrices with i.i.d.~entries under certain moment assumptions. This is one of the most challenging and difficult problems related  to  infinite products of large random matrices; see \cite{Erdos-Peche-Ramirez-Schlein-Yau10,Tao-Vu11a} or \cite{Erdos-Yau17}  for a significant breakthrough on Wigner matrices.
\end{que}

\textbf{Added note:} We note that after the paper was posted on arXiv, Ahn \cite{Ahn19}, Gorin and Sun \cite{Gorin-Sun18}, Hanin and Nica \cite{Hanin-Nica19} made substantial progress on the topic of infinite products of large random matrices.

\vspace{10pt}
\paragraph{{\bf Acknowledgements}}
We would like to  thank G. Akemann, Z. Burda and M. Kieburg  for useful discussions and for  sharing their results, particularly the same critical  scaling  limit at the soft edge (cf.~\cite{Akemann-Burda-Kieburg18}), at the workshop on Sums and Products of Random Matrices (27-31 August 2018, Bielefeld). We are especially grateful to M.~Kieburg for his valuable comments and suggestions. We acknowledge support by the National Natural Science Foundation of China (\#11771417, \#12090012), the Youth Innovation Promotion Association CAS \#2017491 (DZL), and by Singapore AcRF Tier 1 grant R-146-000-262-114, the National Natural Science Foundation of China \#11871425 (DW), and by the National Natural Science Foundation of China \#11901161 (YW).


\end{document}